\documentclass[10pt,a4paper,english,reqno]{smfart}

\usepackage[utf8]{inputenc}
\usepackage{babel}
\usepackage[T1]{fontenc}
\usepackage{lmodern}

\usepackage[hidelinks,pdfencoding=unicode]{hyperref}
\usepackage[shortcuts]{extdash}
\usepackage{xpatch}

\usepackage{amsmath}
\usepackage{amssymb}
\usepackage{stmaryrd}
\usepackage{mathtools} 
\usepackage{diagbox}
\usepackage{graphicx}
\usepackage{bbm}
\usepackage{leftindex}

\usepackage[bb=stix]{mathalpha}

\usepackage{todonotes}

\usepackage[all,2cell]{xy}
\UseAllTwocells
\SilentMatrices
\SelectTips{cm}{10}

\usepackage{tikz}

\usetikzlibrary{cd}
\usetikzlibrary{calc}
\usetikzlibrary{math}
\usetikzlibrary{arrows}
\usetikzlibrary{fit}
\usetikzlibrary{positioning}

\usepackage{version}

\usepackage{enumitem}
\setlist{nosep}
\setlist[enumerate, 1]{label={\rm (}\emph{\alph*}{\rm )}}
\setlist[enumerate, 2]{label={\rm (}\emph{\alph{enumi}.\arabic*}{\rm )}}

\makeatletter

\newcommand\mysubsection{%
 \@startsection{subsection}{1}
    {\z@}%
    {\bigskipamount}%
    {\medskipamount}%
    {\noindent\textbf}}

\makeatother

\let\infty\omega

\newcommand\forlang\emph
\newcommand\ndef\emph
\newcommand\nbd\nobreakdash
\newcommand\skipout\medskip

\newcommand\textstatement\emph
\makeatletter
\newcommand\setstretch[1]{%
  \def\baselinestretch{#1}%
  \@currsize
}
\makeatother

\renewcommand\epsilon\varepsilon
\newcommand\e\epsilon
\renewcommand\phi\varphi
\renewcommand\le\leqslant
\renewcommand\ge\geqslant

\newcommand\quadtext[1]{\quad\text{#1}\quad}

\newcommand\quadand{\quadtext{and}}

\newif\iffr
\newcommand\fren[2]{\iffr #1\else #2\fi}

\newcommand\journal\emph

\renewcommand\and{\fren{et}{and}\xspace}

\newcommand\zbox[1]{\makebox[0pt][l]{#1}}
\newcommand\dpbox[1]{\zbox{\quad#1}}
\newcommand\mdpbox[1]{\zbox{\,\,#1}}
\newcommand\mpbox[1]{\zbox{\,#1}}

\newtheorem{theorem}{Theorem}
\newtheorem*{theorem*}{Theorem}
\newtheorem{proposition}[theorem]{Proposition}
\newtheorem*{proposition*}{Proposition}

\newtheorem{corollary}[theorem]{Corollary}
\newtheorem*{corollary*}{Corollary}

\swapnumbers

\swapnumbers
\theoremstyle{remark}
\newtheorem{remark}[theorem]{Remark}

\newtheorem{example}[theorem]{Example}

\theoremstyle{definition}

\let\paragraph\undefined
\newtheorem{paragraph}[theorem]{}
\newtheorem{paragr}[theorem]{}

\preto\chapter{\numberwithin{theorem}{chapter}}
\preto\section{\numberwithin{theorem}{section}}
\preto\subsection{\numberwithin{theorem}{subsection}}

\newcommand\uomega{\unichar{"03C9}}

\newcommand\pdfoo{\texorpdfstring{$\infty$}{\uinfty}}

\let\pdfoo\pdfomega

\newcommand\var{{-}}

\newcommand{\except}{\mathchoice{\raise 1.8pt\hbox{${\scriptstyle\kern
2.5pt\smallsetminus\kern 2.5pt}$}}{\raise 1.8pt\hbox{${\scriptstyle\kern
2.5pt\smallsetminus\kern 2.5pt}$}}{\raise
1.8pt\hbox{${\scriptscriptstyle\kern 1.5pt\smallsetminus\kern
1.5pt}$}}{\raise 1.8pt\hbox{${\scriptscriptstyle\kern
1.5pt\smallsetminus\kern 1.5pt}$}}}

\newcommand\dunion\coprod
\renewcommand\emptyset\varnothing

\newcommand\N{\mathbb{N}}
\newcommand\Z{\mathbb{Z}}

\newcommand{\Ob}{\operatorname{\mathsf{Ob}}}
\newcommand\id[1]{1_{#1}}
\newcommand\idd[1]{\mathbbm{1}_{#1}}
\newcommand{\Hom}{\operatorname{\mathsf{Hom}}}

\newcommand{\tr}[2]{\mathchoice
  {#1\raise -1.8pt\vbox{\hbox{$\kern -.8pt/#2$}}}
  {#1\raise -1.8pt\vbox{\hbox{$\kern -.8pt/#2$}}\kern .8pt}
  {#1\raise -1.8pt\vbox{\hbox{$\scriptstyle\kern -.8pt /#2$}}}
  {#1\raise -1.8pt\vbox{\hbox{$\scriptscriptstyle\kern -.8pt /#2$}}}}
\newcommand{\wtr}[2]{\mathchoice
  {#1\raise -1.8pt\vbox{\hbox{$\kern -.8pt/_{\raisebox{-1pt}{\!$\scriptstyle\Cyl$}}#2$}}}
  {#1\raise -1.8pt\vbox{\hbox{$\kern -.8pt/_{\raisebox{-2pt}{\!$\scriptstyle\Cyl$}}#2$}}\kern .8pt}
  {#1\raise -1.8pt\vbox{\hbox{$\scriptstyle\kern -.8pt/_{\!\raisebox{-1pt}{$\scriptscriptstyle\Cyl$}}#2$}}}
  {TODO}}

\newcommand{\wtrD}[3]{\mathchoice
  {#2\raise -1.8pt\vbox{\hbox{$\kern -.8pt\stackrel{\rm #1}{/_{\raisebox{-1pt}{$\!\scriptstyle\Cyl$}}}\!#3$}}}
  {#2\raise -1.8pt\vbox{\hbox{$\kern -.8pt\stackrel{\,\rm #1}{/_{\raisebox{-2pt}{$\!\scriptstyle\Cyl$}}}\!#3$}}\kern .8pt}
  {#2\raise -1.8pt\vbox{\hbox{$\scriptstyle\kern
  -.8pt\stackrel{\rm #1}{/_{\!\raisebox{-1pt}{$\scriptscriptstyle\Cyl$}}}#3$}}\kern .8pt}
  {TODO}}

\newcommand{\wptrD}[3]{\mathchoice
  {#2\raise -1.8pt\vbox{\hbox{$\kern -.8pt\stackrel{\rm #1}{/_{\raisebox{-1pt}{$\!\scriptstyle\Cylp$}}}\!#3$}}}
  {#2\raise -1.8pt\vbox{\hbox{$\kern -.8pt\stackrel{\,\rm #1}{/_{\raisebox{-2pt}{$\!\scriptstyle\Cylp$}}}\!#3$}}\kern .8pt}
  {#2\raise -1.8pt\vbox{\hbox{$\scriptstyle\kern
  -.8pt\stackrel{\rm #1}{/_{\!\raisebox{-1pt}{$\scriptscriptstyle\Cylp$}}}#3$}}\kern .8pt}
  {TODO}}

\newcommand\wtrto{\wtrD{\dto}}

\newcommand\trs\trsc
\newcommand{\trsa}[2]{\mathchoice
{TODO}
  {#1\mskip2mu\raise-1.8pt\vbox{%
    \hbox{$\kern-.8pt/\kern-5.4pt\raise1pt\vbox{%
\hbox{$\scriptscriptstyle\circlearrowright$}}\mskip1mu#2$}}\kern.8pt}
{TODO}
{TODO}}
\newcommand{\trsab}[2]{\mathchoice
{TODO}
  {#1\mskip2mu\raise-1.8pt\vbox{%
    \hbox{$\kern-.8pt/\kern-5.4pt\raise1pt\vbox{%
\hbox{$\scriptstyle\circlearrowright$}}\mskip1mu#2$}}\kern.8pt}
{TODO}
{TODO}}
\newcommand{\trsac}[2]{\mathchoice
{TODO}
  {#1\mskip2mu\raise-1.8pt\vbox{%
    \hbox{$\kern-.8pt/\kern-4.3pt\raise1pt\vbox{%
\hbox{$\scriptstyle\circlearrowright$}}\mskip1mu#2$}}\kern.8pt}
{TODO}
{TODO}}
\newcommand{\trsb}[2]{\mathchoice
{TODO}
  {#1\mskip2mu\raise-1.8pt\vbox{%
    \hbox{$\kern-.8pt/\kern-5pt\raise1pt\vbox{%
      \hbox{$\scriptscriptstyle=$}}\mskip1mu#2$}}\kern.8pt}
{TODO}
{TODO}}
\newcommand{\trsc}[2]{\mathchoice
  {#1\mskip1mu\raise-1.8pt\vbox{%
    \hbox{$\kern-.8pt/\kern-4.5pt\raise1pt\vbox{%
      \hbox{$\scriptstyle\circ$}}\mskip1mu#2$}}}
  {#1\mskip1mu\raise-1.8pt\vbox{%
    \hbox{$\kern-.8pt/\kern-4.5pt\raise1pt\vbox{%
      \hbox{$\scriptstyle\circ$}}\mskip1mu#2$}}\kern.8pt}
  {#1\mskip1mu\raise-1.8pt\vbox{%
    \hbox{$\scriptstyle\kern-.8pt/\kern-4.0pt\raise-0pt\vbox{%
      \hbox{$\scriptstyle\circ$}}\mskip0mu#2$}}\kern.8pt}
{TODO}}
\newcommand{\cotr}[2]{\mathchoice
  {\raise -1.8pt\vbox{\hbox{$#2\backslash$}}#1}
  {\raise -1.8pt\vbox{\hbox{$#2\backslash$}}#1}
  {\raise -1.8pt\vbox{\hbox{$\scriptstyle#2\backslash$}}#1}
  {\raise -1.8pt\vbox{\hbox{$\scriptscriptstyle#2\backslash$}}#1}}
\newcommand{\trm}[2]{\mathchoice
  {#1\raise -1.8pt\vbox{\hbox{$\kern -.8pt\!\stackrel{\,\,\,\rm co}{/}\!\!#2$}}}
  {#1\raise -1.8pt\vbox{\hbox{$\kern -.8pt\!\stackrel{\,\,\,\rm co}{/}\!\!#2$}}\kern .8pt}
  {#1\raise -1.8pt\vbox{\hbox{$\scriptstyle\kern -.8pt\!\!\stackrel{\,\,\,\rm co}{/}\!\!#2$}}\kern .8pt}
  {#1\raise -1.8pt\vbox{\hbox{$\scriptscriptstyle\kern
  -.8pt\!\!\stackrel{\,\,\,\rm co}{/}\!\!#2$}}\kern .8pt}}
\newcommand{\trD}[3]{\mathchoice
  {#2\raise -1.8pt\vbox{\hbox{$\kern -.8pt\!\stackrel{\,\,\,\rm #1}{/}\!\!#3$}}}
  {#2\raise -1.8pt\vbox{\hbox{$\kern -.8pt\!\stackrel{\,\,\,\rm #1}{/}\!\!#3$}}\kern .8pt}
  {#2\raise -1.8pt\vbox{\hbox{$\scriptstyle\kern -.8pt\!\!\stackrel{\,\,\,\rm #1}{/}\!\!#3$}}\kern .8pt}
  {#2\raise -1.8pt\vbox{\hbox{$\scriptscriptstyle\kern
  -.8pt\!\!\stackrel{\,\,\,\rm #1}{/}\!\!#3$}}\kern .8pt}}

\newcommand{\trto}[2]{\mathchoice
  {#1\raise -1.8pt\vbox{\hbox{$\kern -.8pt\!\stackrel{\,\,\,\rm to}{/}\!\!#2$}}}
  {#1\raise -1.8pt\vbox{\hbox{$\kern -.8pt\!\stackrel{\,\,\,\rm to}{/}\!\!#2$}}\kern .8pt}
  {#1\raise -1.8pt\vbox{\hbox{$\scriptstyle\kern -.8pt\!\!\stackrel{\,\,\,\rm to}{/}\!\!#2$}}\kern .8pt}
  {#1\raise -1.8pt\vbox{\hbox{$\scriptscriptstyle\kern
  -.8pt\!\!\stackrel{\,\,\,\rm to}{/}\!\!#2$}}\kern .8pt}}
  
\newcommand{\cotrD}[3]{\mathchoice
  {\raise -1.8pt\vbox{\hbox{$#3\!\stackrel{\!\!\!\!\rm #1}{\backslash}$}}#2}
  {\raise -1.8pt\vbox{\hbox{$#3\!\stackrel{\!\!\!\!\rm #1}{\backslash}$}}#2}
  {\raise -1.8pt\vbox{\hbox{$\scriptstyle#3\stackrel{\!\!\!\rm
  #1}{\backslash}$}}#2}
  {TODO}}

\renewcommand\o{\mathrm{o}}

\newcommand\A{\mathcal{A}}
\newcommand\B{\mathcal{B}}

\newcommand\V{\mathcal{V}}

\newcommand\limind\varinjlim
\newcommand\limproj\varprojlim
\DeclareMathOperator*{\prodfib}{\times}
\newcommand\comma{\mathop{\downarrow}\nolimits}
\newcommand\commabig{\mathop{\downarrow}}
\newcommand\commalax{\mathop{\downarrow^{\mskip-2mu\prime}\mskip-3mu}\nolimits}

\newcommand\trans\dt 
\newcommand\VCat{\ECat{\V}} 
\newcommand\VbCat{\ECat{\overline{\V}}} 
\newcommand\ECat[1]{\nCat{#1}} 
\newcommand\VpCat{\ECat{\V'\mskip-1mu}}

\newcommand\Mod[1]{\leftindex_{#1}{\mathrm{Mod}}}
\newcommand\RMod[1]{\mathrm{Mod}_{#1}}
\newcommand\ev{\mathrm{ev}}

\newcommand{\Set}{{\mathcal{S}\mspace{-2.mu}\it{et}}}
\newcommand{\Cat}{{\mathcal{C}\mspace{-2.mu}\it{at}}}

\newcommand{\nCat}[1]{{#1}\hbox{\protect\nbd-}\kern1pt\Cat}

\newcommand{\ooCat}{\nCat{\infty}}
\newcommand{\ooCatCart}{\infty\hbox{\protect\nbd-}\Cat^{}_\mathrm{cart}}
\newcommand{\ooCATCart}{\infty\hbox{\protect\nbd-}\CAT^{}_\mathrm{cart}}

\newcommand{\ooGray}{{\omega}\hbox{\protect\nbd-}\kern1pt\mathcal{G}\mspace{-2.mu}\it{ray}}

\newcommand{\CAT}{{\mathcal{C}\mspace{-2.mu}\rm{A\kern-1pt T}}}
\newcommand{\CATbb}{\mathbb{C}\rm AT}
\newcommand{\nCAT}[1]{{#1}\hbox{\protect\nbd-}\kern1pt\CAT}
\newcommand{\ooCAT}{\nCAT{\infty}}

\newcommand\oo{$\infty$\nbd}
\newcommand\comp\ast
\newcommand\Dn[1]{\mathrm{D}_{#1}}
\newcommand\boundary\partial

\newcommand\HomOpLax{\Homi_{\mathrm{oplax}}}
\newcommand\HomLax{\Homi_{\mathrm{lax}}}

\newcommand\HomCyl{\Homi_\Cyl}
\newcommand\op{\mathrm{op}}
\newcommand\dop\op
\newcommand\co{\mathrm{co}}
\newcommand\dco\co

\newcommand\dtop{\mathrm{top}}
\newcommand\dto{\mathrm{to}}
\newcommand\dcot{\mathrm{cot}}
\newcommand\dt{\mathrm{t}}

\newcommand\join\star
\newcommand\ooCatOpLax\ooCatOpLaxGray
\newcommand\ooCATOpLax{\infty\hbox{\protect\nbd-}\CATbb^{}_\mathrm{oplax}}
\newcommand\ooCatLax\ooCatLaxGray
\newcommand\ooCatOpLaxGray{\infty\hbox{\protect\nbd-}\mathbb{C}\mathsf{at}^{}_\mathrm{oplax}}
\newcommand\ooCatLaxGray{\infty\hbox{\protect\nbd-}\mathbb{C}\mathsf{at}^{}_\mathrm{lax}}

\newcommand\Ncn[1][]{\ifthenelse{\isempty{#1}}{\Ncnaux{n}}{\Ncnaux{#1}}}
\newcommand\Ncnaux[1]{N_{\Theta_{#1}}}
\newcommand\Nmsn[1][]{\ifthenelse{\isempty{#1}}{\Nmsnaux{n}}{\Nmsnaux{#1}}}
\newcommand\Nmsnaux[1]{N_{\cDelta^{\mkern-2mu #1}}}

\newcommand{\Homi}{\operatorname{\kern.5truept\underline{\kern-.5truept\mathsf{Hom}\kern-.5truept}\kern1truept}}

\newcommand\cDelta{\mathbf{\Delta}}

\newcommand\hDeltank[2]{\Lambda^k_n}

\newcommand\longto\longrightarrow
\newcommand\ot\leftarrow
\newcommand\longot\longleftarrow
\newcommand\hookto\hookrightarrow
\newcommand\xto\xrightarrow
\newcommand\xot\xleftarrow
\newcommand\tod\Rightarrow
\newcommand\otd\Leftarrow

\newcommand\toiso{\xto{\sim}}

\newcommand\Gro[1]{\textstyle\int_{#1}}

\newcommand\CG{\mathbb{C}}
\DeclareSymbolFont{AMSbb}{U}{msb}{m}{n}
\DeclareMathSymbol{\amsbbI}{\mathalpha}{AMSbb}{"49}
\newcommand\IG{\amsbbI}
\newcommand\HomCG{\Homi_{\CG}}

\newcommand\Cyl\Gamma
\newcommand\Cylp{\Gamma'}
\newcommand\Cylpb[1]{\times_{#1} \Cyl{#1} \times_{#1}}
\newcommand\compc{\ast_c}
\newcommand\compl{\ast_l}
\newcommand\compr{\ast_r}
\newcommand\compG{\ast_0}
\renewcommand\ev{\mathrm{ev}}
\renewcommand\ss{\mathbb{s}}
\renewcommand\tt{\mathbb{t}}
\newcommand\kk{\mathbb{k}}
\newcommand\ee{\mathbb{e}}

\newcommand\otimesdual{\mathop{\mskip1mu\overline{\otimes}\mskip1mu}}
\newcommand\Vdual[1]{\overline{#1}}
\newcommand\compV{\circ}
\newcommand\idV[1]{\id{#1}}
\newcommand\Vcart{\V_{\mathrm{cart}}}

\newcommand\Vl{\V_{\textit{l}}}
\newcommand\Vr{\V_{\textit{r}}}
\newcommand\Homir{\Homi_{\textit{r}}}
\newcommand\Homil{\Homi_{\textit{l}}}
\newcommand\can{\mathrm{can}}

\author{Dimitri Ara}
\address{Aix~Marseille~Univ,~CNRS,~I2M,~Marseille,~France}
\email{dimitri.ara@univ-amu.fr}
\urladdr{\href{http://www.i2m.univ-amu.fr/perso/dimitri.ara/}{http://www.i2m.univ-amu.fr/perso/dimitri.ara/}}

\author{Léonard Guetta}
\address{~Utrecht~University,~Utrecht,~the~Netherlands\hskip1cm}
\email{l.s.guetta@uu.nl\hskip4cm}
\urladdr{\href{https://leoguetta.github.io/}{https://leoguetta.github.io/}}

\title{Lax functorialities of the comma construction for \pdfoo-categories}

\begin{document}

\frontmatter

\begin{abstract}
  Motivated by the Grothendieck construction, we study the functorialities
  of the comma construction for strict \oo-categories. To state the
  most general functorialities, we use the language of Gray \oo-categories, that
  is, categories enriched in the category of strict \oo-categories endowed
  with the oplax Gray tensor product. Our main result is that the comma
  construction of strict \oo-categories defines a Gray \oo-functor, that
  is, a morphism of Gray \oo-categories. To makes sense of this statement,
  we prove that slices of Gray \oo-categories exist. Coming back to the
  Grothendieck construction, we propose a definition in terms of the comma
  construction and, as a consequence, we get that the Grothendieck
  construction of strict \oo-categories defines a Gray \oo-functor.
  Finally, as a by-product, we get a notion of Grothendieck construction for
  Gray \oo-functors, which we plan to investigate in future work.
\end{abstract}

\def\subjclassname
    {\textup{2020} Mathematics Subject Classification}%

\def\keywordsname{Keywords}%

\subjclass{18A25, 18D20, 18N20, 18N30}

\keywords{\oo-categories, comma \oo-categories, Gray \oo-categories,
  Grothendieck construction, slice Gray \oo-categories, strict
\oo-categories}

\maketitle

\mainmatter

\section*{Introduction}

\mysubsection*{The starting point: the Grothendieck construction}

The classical Grothendieck construction defines, for every small category
$I$, a functor
\[ \Gro{I} \colon \Homi(I^{\o}, \Cat) \to \Cat \]
that sends a functor $F \colon I^{\o} \to \Cat$, where $I^\o$ denotes the opposite category
of $I$ and $\Cat$ the category of small categories, to the so-called
\emph{Grothendieck construction} $\Gro{I} F$ of~$F$. Here $\Homi$ denotes the
cartesian internal hom of $\Cat$, whose morphisms are
strict natural transformations. But the functorialities of the
Grothendieck construction are more general. First, if $F, G \colon I^{\o} \to \Cat$
are two functors of this kind and $\alpha \colon F \tod G$ is an \emph{oplax}
transformation (that is, roughly speaking, a transformation where the
naturality squares only commute up to a non-invertible $2$-cell), then one can
still integrate $\alpha$ to obtain a functor $\Gro{I} \alpha \colon \Gro{I} F \to
\Gro{I} G$. Second, the construction is also functorial
in~$I$. Combining these, we get a functoriality
\[
  \raisebox{1.5pc}{
  $\xymatrix@C=1.5pc@R=3pc{
    I^{\o} \ar[rr]^{u^{\o}} \ar[dr]_{F}_{}="f" & & J^{\o}
    \ar[dl]^{G} \\
    & \Cat
    \ar@{}"f";[ur]_(.15){}="ff"
    \ar@{}"f";[ur]_(.55){}="oo"
    \ar@<-0.5ex>@2"ff";"oo"^{\alpha}
    &
  }$}
  \qquad
  \mapsto
  \qquad
  \xymatrix@C=3pc{
    \Gro{I} F \ar[r]^{\Gro{} (u, \alpha)} & \Gro{J} G
    \dpbox,
  }
\]
where $\alpha$ is an oplax transformation.

The purpose of this paper is to study higher generalizations of
these functorialities in the setting of strict
\oo-categories. Our original motivation was to investigate the
homotopical properties of the Grothendieck construction for strict
\oo-categories, and particularly the generalization of a theorem by
Thomason \cite{ThomasonHocolim}, which will be addressed in a separate paper
\cite{AraGagnaGuettaIntegration}.

\mysubsection*{From the Grothendieck construction to comma
\pdfoo-categories}

Let~$\ooCat$ denote the category of strict \oo-categories and strict
\oo-functors. Using its cartesian internal hom, this category can be
promoted to a strict \oo-category~$\ooCatCart$, whose $2$-cells are strict
transformations and whose higher cells are strict higher transformations. If
now $F \colon I^{\o} \to \ooCatCart$ is a strict \oo-functor, where $I$ is a
strict \oo-category and $I^{\o}$ denotes the dual \oo-category obtained by
reversing the orientation of all the cells of~$I$, then a Grothendieck
construction $\Gro{I} F$ was defined by Warren in his work on the model of
strict \oo-groupoids for dependent type theory~\cite{Warren}.

However, the definition given by Warren is unsatisfactory because it relies
on explicit and complicated formulas. We propose to \emph{define} the
Grothendieck construction of $F \colon I^{\o} \to \ooCatCart$ as the
universal \oo-category $\Gro{I} F$ endowed with a $2$-square
\[
  \xymatrix@C=1pc@R=1.5pc{
    & (\Gro{I} F)^{\o} \ar[dl] \ar[dr]^{} \\
    \Dn{0} \ar[dr]_{c^{}_{\Dn{0}}} \ar@{}[rr]_(.35){}="x"_(.65){}="y"
    \ar@2"x";"y"^{\gamma} 
    & & I^{\o} \ar[dl]^F \\
    & \ooCatCart & \dpbox,
  }
\]
where $\Dn{0}$ denotes the terminal \oo-category, $c^{}_{\Dn{0}}$ the constant
\oo-functor of value $\Dn{0}$ and $\gamma$ an oplax transformation.
This type of universal $2$-squares was already studied by the first-named author and
Maltsiniotis \cite{AraMaltsiThmAII} and is a straightforward generalization of the
classical comma construction of two functors $f \colon A \to C$ and $g
\colon B \to C$, usually denoted by~$f \comma g$. More precisely, we have
\[ \Gro{I} F = (c^{}_{\Dn{0}} \comma F)^{\o} \mpbox, \]
where $\comma$ denotes the \emph{oplax} comma construction. Although these
definitions are abstract, explicit formulas can be extracted, and one can recover
Warren's formulas from this abstract point of view.

We are thus led to consider the following more general case. Let $A$, $B$
and $C$ be three strict \oo-categories, and let $f \colon A \to C$ and $g
\colon B \to C$ be two strict \oo-functors. We can form the (oplax) comma
\oo-category $A\comma_C B$, that is, the universal strict \oo-category
 equipped with a $2$\nbd-square
\[
    \xymatrix@C=1pc@R=1.5pc{
    & A\comma_CB \ar[dl] \ar[dr]^{} \\
    A \ar[dr]_{f} \ar@{}[rr]_(.35){}="x"_(.65){}="y"
    \ar@2"x";"y"^{\gamma} 
    & & B \ar[dl]^g \\
    & C & \dpbox,
  }
\]
where $\gamma$ is an oplax transformation.
Our goal is now to study the functorialities of $A\comma_CB$ in
$g\colon B \to C$, with $f \colon A \to C$ fixed, and symmetrically,
of which the functorialities of the Grothendieck construction are a
particular case.
The universal property of the comma construction immediately gives a
functoriality
\[
  \begin{split}
  \raisebox{1.5pc}{
  $\xymatrix@C=1.5pc@R=3pc{
    B \ar[rr] \ar[dr]_{\phantom{'}f}_{}="f" & & B'
    \ar[dl]^(0.5){f'} \\
    & C
    \ar@{}"f";[ur]_(.15){}="ff"
    \ar@{}"f";[ur]_(.55){}="oo"
    \ar@<-0.5ex>@2"ff";"oo"^{}
    &
  }$}
  \qquad
  &
  \mapsto
  \qquad
  \xymatrix{
    A\comma_CB \ar[r] & A\comma_CB' \dpbox,
  }
  \end{split}
\]
where the $2$-cell represents an oplax transformation.
Working a bit harder, the first-named author and Maltsiniotis got a
functoriality
\[
  \begin{split}
  \raisebox{2pc}{
  $\xymatrix@C=1.5pc@R=3pc{
    B \ar@/^2ex/[rr]^(.33){}_{}="1" \ar@/_2ex/[rr]^(.30){}_{}="0"
    \ar[dr]_{}="f"_{\phantom{'}f}
    \ar@2"1";"0"^{}
    & & B' \ar[dl]^{f'} \\
    & C
    \ar@{}"f";[ur]_(.15){}="ff"
    \ar@{}"f";[ur]_(.55){}="oo"
    \ar@<-0.5ex>@/^1ex/@{:>}"ff";"oo"^(.18){}_(.30){}="h'"
    \ar@<-2.0ex>@/^-1ex/@2"ff";"oo"_(.22){}_(.80){}="h"
    \ar@3"h";"h'"_(.20){}
  }$}
  \qquad
  &
  \mapsto
  \qquad
    \xymatrix@C=3pc@R=3pc{
      A\comma_CB \ar@/^2.5ex/[r]_{}="0"
      \ar@/_2.5ex/[r]_{}="1"
      \ar@{}"0";"1"_(.10){}="00"
      \ar@{}"0";"1"_(.90){}="11"
      \ar@2"00";"11"
    & A\comma_CB'} \dpbox,
  \end{split}
\]
where the $2$-cells represent oplax transformations and the $3$-cell
represents an oplax $2$\nobreakdash-transformation (also known as an oplax
modification), and proved that the comma construction defines a
sesquifunctor (see \cite[Theorem~B.2.6]{AraMaltsiThmAII}). And now comes the
question: what is the general functoriality statement?

\mysubsection*{The need for Gray \pdfoo-categories and their slices}

The answer to this question uses the language of Gray \oo-categories, which
the first-named author introduced with Maltsiniotis in their work on the join construction and
the slices~\cite{AraMaltsiJoint}. Indeed, the diagrams above involving $0$-cells,
$1$\nbd-cells, $2$\nbd-cells and $3$\nbd-cells actually live in~$\ooCatOpLax$, in which
$0$-cells are strict \oo-categories, $1$-cells are strict \oo-functors,
$2$\nobreakdash-cells are oplax transformations, $3$-cells are oplax $2$-transformations,
and so on. But $\ooCatOpLax$ is not an \oo-category, \emph{not even a weak
one}! Indeed, if
\[
  \xymatrix@C=3pc{
    A \ar@/^2.3ex/[r]_{}="1"
    \ar@/_2.3ex/[r]_{}="0"
    &
    B \ar@/^2.3ex/[r]_{}="2"
    \ar@/_2.3ex/[r]_{}="3"
    &
    C
    \ar@2"1";"0"_{\alpha}
    \ar@2"2";"3"_{\beta}
  }
\]
are two oplax transformations, then there are a priori two ways of composing
them:
\[
  (t(\beta) \comp_0 \alpha) \comp_1 (\beta \comp_0 s(\alpha))
  \quad\text{and}\quad
  (\beta \comp_0 t(\alpha)) \comp_1 (s(\beta) \comp_0 \alpha)
  \mpbox,
\]
where $s$ and $t$ denote the source and the target. In general, these two
oplax transformations are different! In other words, $\ooCatOpLax$ does not
satisfy the exchange rule. What is true is that there is a \emph{non-invertible}
canonical oplax $2$-transformation
\[
  \xymatrix@C=1.7pc{
    (\beta \comp_0 t(\alpha)) \comp_1 (s(\beta) \comp_0 \alpha)
    \ar@3[r]^-{}
  &
  (t(\beta) \comp_0 \alpha) \comp_1 (\beta \comp_0 s(\alpha))
  \mpbox,
}
\]
which can be pictured as
\[
    \left\{
  \begin{tikzcd}
    A\ar[r,bend left,""{name=C,below}] \ar[r,bend right,""{name=D,above}]&B \ar[d,phantom,"\ast_1" description] \ar[r] & C\\
    A \ar[r] & B \ar[r,bend left,""{name=A,below}] \ar[r,bend
    right,""{name=B,above}] &C
    \ar[from=A,to=B,Rightarrow,"\,\beta"]
    \ar[from=C,to=D,Rightarrow,"\,\alpha"]
  \end{tikzcd}
\right\}
\,
\Rrightarrow
\,
 \left\{ \begin{tikzcd}
    A \ar[r] & B \ar[r,bend left,""{name=A,below}] \ar[r,bend right,""{name=B,above}]\ar[d,phantom,"\ast_1" description] &C \\
    A\ar[r,bend left,""{name=C,below}] \ar[r,bend right,""{name=D,above}]&B \ar[r] & C
    \ar[from=A,to=B,Rightarrow,"\,\beta"]
    \ar[from=C,to=D,Rightarrow,"\,\alpha"]
  \end{tikzcd}
\right\}
  \begin{tikzcd}
    \\\\\mpbox.
  \end{tikzcd}
\]
This means that $\ooCatOpLax$ is some kind of oplax \oo-category. Formally,
$\ooCatOpLax$ is what we call a \emph{Gray \oo-category}, that is, a
category enriched in $\ooCat$ endowed with the oplax Gray tensor product.
Morphisms of Gray \oo-categories are called \emph{Gray \oo-functors}.

Let us come back to the comma construction. It seems reasonable
to expect that if the correspondence
\[
  A\comma_C \var \colon (B, B\to C) \mapsto A\comma_CB
\]
extends to a Gray
\oo-functor, then its target should be~$\ooCatOpLax$. But what would be the source Gray
\oo-category? Or, in other words, in which Gray \oo-category do the
triangles and cones drawn earlier are $1$-cells and
$2$-cells? Obviously, in some kind of slice Gray \oo-category
$\tr{\ooCatOpLax}{C}$ of the Gray \oo-category $\ooCatOpLax$ above the
strict \oo-category $C$ (which is an object of $\ooCatOpLax$). We prove that
slice Gray \oo-categories exist, in full generality:

\begin{theorem*}
Let $\CG$ be a Gray \oo-category and let $c$ be an object of $\CG$. Then
there is a (natural) Gray \oo-category~$\tr{\CG}{c}$ of objects of~$\CG$
over~$c$.
\end{theorem*}

In the case that $\CG$ is a \emph{strict} \oo-category (which we
can consider as a Gray \oo-category where the exchange rule is an
equality), we recover the usual notion of slices of strict \oo-categories
(see for instance \cite[Chapter 9]{AraMaltsiJoint}). Note that the existence
of slices of Gray \oo-categories was first conjectured in
\cite[Appendix C]{AraMaltsiJoint}.

\mysubsection*{Our Gray-functoriality results}

Using slices of Gray \oo-categories, we can finally express the desired
functoriality of the construction $A\comma_CB$, actually both in $A$ and $B$
simultaneously, which is the main result of our paper:

\begin{theorem*}
  The oplax comma construction $\var \comma_C \var$ defines a Gray
  \oo-functor
  \[
    \var \comma_C \var \colon \tr{\ooCatOpLax}{C}\times\trto{\ooCatOpLax}{C} \to \ooCatOpLax
    \mpbox,
  \]
  where the decoration ``$\dto$'' indicates a duality by which the slice has
  to be conjugated.
\end{theorem*}

This theorem gives the full functorialities of the comma construction.
In particular, in low dimensions, we recover the sesquifunctor of
\cite[Theorem~B.2.6]{AraMaltsiThmAII} and the mapping
\[
  \raisebox{2.45pc}{
  $\xymatrix@R=1.5pc@C=3.5pc{
    A \ar@/_2ex/[dd]_(0.62){}_{}="u"
    \ar@/^2ex/[dd]_(0.65){}_{}="u'"
    \ar[dr]^f_{}="f" & &
    B \ar@/_2ex/[dd]^(0.65){}_{}="v"
    \ar@/^2ex/[dd]^(0.65){}_{}="v'"
    \ar[dl]_g_{}="g" \\
                     & C \\
    A' \ar[ur]_{f'} & & B' \ar[ul]^{}
    \ar@2"u";"u'"^{}
    \ar@2"v'";"v"_{}
    \ar@{}[ll];"f"_(0.40){}="sa"_(0.85){}="ta"
    \ar@<-1.5ex>@/_1ex/@2"sa";"ta"_(.70){}_(0.40){}="a'"
    \ar@<0.0ex>@/^1ex/@{:>}"sa";"ta"^(.70){}_(0.60){}="a"
    \ar@3"a'";"a"_(.60){}
    \ar@{}[];"g"_(0.40){}="tb"_(0.85){}="sb"
    \ar@<-1.5ex>@/_1ex/@2"sb";"tb"_(.30){}_(0.40){}="b'"
    \ar@<0.0ex>@/^1ex/@{:>}"sb";"tb"^(.30){}_(0.60){}="b"
    \ar@3"b'";"b"^(.40){}
  }$}
  \qquad
  \mapsto
  \qquad
  \xymatrix@C=3pc@R=3pc{
    A\comma_CB \ar@/^2.5ex/[r]_{}="0"
    \ar@/_2.5ex/[r]_{}="1"
    \ar@{}"0";"1"_(.10){}="00"
    \ar@{}"0";"1"_(.90){}="11"
    \ar@2"00";"11"
    & A'\comma_CB' \dpbox,}
\]
where $2$-cells represent oplax transformations and the $3$-cell
represents an oplax $2$\=/transformation. Concretely, our theorem asserts that
this mapping generalizes to cones of any dimension and higher oplax
transformations, and moreover that these mappings are compatible with all
the compositions of Gray slices.

This theorem is an instance of an idea we would like to put forward:
\emph{every functorial construction on \oo-categories should be promoted to a Gray
\oo-functor.}

Note that, as the duality appearing in the statement of the theorem shows,
dualities of strict and Gray \oo-categories play an important and subtle
role in this paper, and we study them with great care.

As a direct corollary of the theorem, we get:

\begin{corollary*}
  Let $A$ and $B$ be two strict \oo-categories. The oplax comma construction
  restricts to a Gray \oo-functor
  \[
    \var \comma_C \var \colon
    \HomOpLax(A, C)^\o \times \HomOpLax(B, C)^\dto
    \to
    \ooCatOpLax
    \mpbox,
  \]
  where $\HomOpLax(X, Y)$ denotes the strict \oo-category of strict
  \oo-functors from~$X$ to~$Y$, oplax transformations and higher
  oplax transformations between them.
\end{corollary*}

We also study the functorialities of the oplax comma construction restricted
to higher \emph{strict} transformations. The \oo-category $\ooCatCart$ can
be seen as a sub-Gray \oo-category of~$\ooCatOpLax$ by considering higher
strict transformations as particular higher oplax transformations. A priori,
we get by restriction a Gray \oo-functor
\[
  \var \comma_C \var \colon
  \tr{\ooCatCart}{C}\times\trto{\ooCatCart}{C}
  \to
  \ooCatOpLax
  \mpbox,
\]
but we prove that it actually lands into $\ooCatCart$:

\begin{proposition*}
  The oplax comma construction restricts to a strict \oo-functor
  \[
    \var \comma_C \var \colon \tr{\ooCatCart}{C}\times\trto{\ooCatCart}{C}
    \to \ooCatCart \mpbox.
  \]
\end{proposition*}

Finally, we apply all these results to the particular case of the
Grothendieck construction. For that, notice that the \oo-category
$\ooCatCart$ is an object of $\ooCATOpLax$, the (very large) Gray
\oo-category of possibly large strict \oo-categories, strict \oo-functors
and lax (higher) transformations between them. We get:

\begin{corollary*}
  The Grothendieck construction defines a Gray \oo-functor
  \[
    \begin{aligned}
      \Gro{} \colon \big(\trto{\ooCatOpLax}{\{\ooCatCart\}}\big)^{\dto} &\to \ooCatOpLax\\
      (F \colon I^{\o} \to \ooCatCart) &\mapsto \Gro{I} F \qquad\quad,
      \end{aligned}
    \]
    $\smash[t]{\trto{\ooCatOpLax}{\{\ooCatCart\}}}$ being the full sub-Gray
    \oo-category of $\smash[t]{\trto{\ooCATOpLax}{\!\{\ooCatCart\}}}$ spanned
    by those \oo-functors $F \colon I^{\o} \to \ooCatCart$, where $I$ is a
    small \oo-category.
\end{corollary*}

In particular, if we fix a strict \oo-category $I$, the above Gray
\oo-functor restricts to a Gray \oo-functor
\[
  \Gro{I} \colon \HomOpLax(I^{\o},\ooCatCart) \to \ooCatOpLax \mpbox{,}
\]
as well as to a strict \oo-functor
\[
  \Gro{I} \colon \Homi(I^{\o},\ooCatCart) \to \ooCatCart \mpbox,
\]
where $\Homi$ is the cartesian internal hom of $\ooCAT$, the (very large)
category of possibly large strict \oo-categories.

Note that functorialities of the Grothendieck construction in the setting of
\emph{weak} \oo-categories have been studied recently by Loubaton
\cite{Loubaton}. However, our work is not a particular case of his as,
first, Loubaton only deals with the case of a fixed base (weak) \oo-category
$I$ and, second, he only considers a (weak) sub-\oo-category of
$\HomOpLax(I^{\o},\ooCatCart)$, containing oplax transformations but not
general oplax $n$\=/transformations for $n > 1$. In other words, the
generalization of our work to weak \oo-categories would give more general
functorialities than the ones studied by \hbox{Loubaton}\footnote{Since we
made our work publicly available, preliminary progress in this direction
has been made by Gepner and Heine \cite{GepnerHeine}}.

\mysubsection*{A side remark on cylinders in Gray \pdfoo-categories}

It seems important to bring the reader's attention to the fact that there is
a \emph{tour de force} behind the definition of slice Gray \oo-category
$\tr{\CG}{c}$. Intuitively, to define such  Gray \oo-categories, one must
make sense in an arbitrary Gray \oo-category of pasting diagrams shaped like
(higher) cones
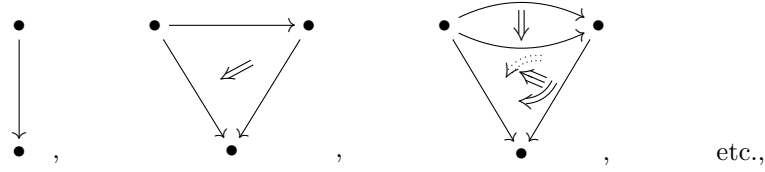
\[
   \xymatrix@R=3pc{
    \bullet \ar[d] \\
    \bullet
    \dpbox{,}
  }
  \qquad\qquad
  \xymatrix@C=1.5pc@R=3pc{
    \bullet \ar[rr] \ar[dr]_{}_{}="f" & & \bullet
    \ar[dl]^(0.42){} \\
    & \bullet
    \ar@{}"f";[ur]_(.15){}="ff"
    \ar@{}"f";[ur]_(.55){}="oo"
    \ar@<0.5ex>@2"oo";"ff"^{}
    &
    \dpbox{,}
  }
  \qquad
  \qquad
  \xymatrix@C=1.5pc@R=3pc{
    \bullet \ar@/^2ex/[rr]^(.33){}_{}="1" \ar@/_2ex/[rr]^(.30){}_{}="0"
    \ar[dr]_{}="f"_{}
    \ar@2"1";"0"^{}
    & & \bullet \ar[dl]^{} \\
    & \bullet
    \ar@{}"f";[ur]_(.15){}="ff"
    \ar@{}"f";[ur]_(.55){}="oo"
    \ar@<-0.5ex>@/^1ex/@{<:}"ff";"oo"^(.18){}_(.30){}="h'"
    \ar@<-2.0ex>@/^-1ex/@{<=}"ff";"oo"_(.22){}_(.80){}="h"
    \ar@3"h";"h'"_(.20){}
    & \mpbox{,\qquad\qquad\text{etc.},}
  }
  \qquad
  \qquad
\]
as well as defining compositions between those. For example, defining the
whiskering operation of a $2$\nbd-cell with a $1$\nbd-cell in slice Gray
\oo-categories amounts to define a ``total composite'' of the
following pasting diagram:
  
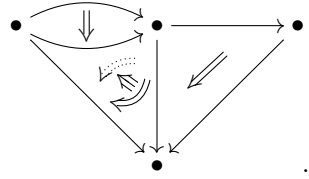
\[
    \xymatrix@C=3.5pc@R=3.5pc{
      \bullet
    \ar@/^2ex/[r]^(.31){}_{}="1"
    \ar@/_2ex/[r]^(.30){}_{}="0"_(.70){}="f"
    \ar[dr]_{}="g"_{}
    \ar@2{<-}"0";"1"_{}
    &
    \bullet \ar[r]^{}_(.75){}="fp"
       \ar[d]_(.70){}="gp2"_(.20){}="gp"^(0.62){}
    &
    \bullet \ar[dl]^{}
    \\
    &
    \bullet
    \ar@{}"g";"gp"_(.15){}="ff1"
    \ar@{}"g";"gp"_(.80){}="oo1"
    \ar@<-0.0ex>@/^1ex/@{<:}"ff1";"oo1"^(.35){}_(.40){}="h'"
    \ar@<-1.0ex>@/^-1.5ex/@2{<-}"ff1";"oo1"_(.46){}_(.70){}="h"
    \ar@3"h";"h'"_(.20){}
    \ar@{}"gp2";"fp"_(.25){}="x2"
    \ar@{}"gp2";"fp"_(.75){}="y2"
    \ar@<0.4ex>@2{<-}"x2";"y2"^{}
    &
    \mpbox.
    }
  \]
In the case of \emph{strict} \oo-categories, the cone-shaped pasting diagrams are
considered as degenerate cases of cylinder-shaped pasting diagrams
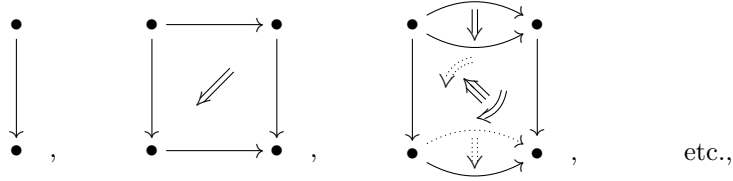
\[
   \xymatrix@R=3pc{
   \bullet \ar[d]_{} \\
   \bullet \dpbox{,}
   }
   \qquad
   \qquad
    \xymatrix@C=3pc@R=3pc{
      \bullet \ar[r]^{} \ar[d]_{} &
      \bullet \ar[d]^{} \\
      \bullet \ar[r]_{} & \bullet
      \ar@{}[u];[l]_(.30){}="s"
      \ar@{}[u];[l]_(.70){}="t"
      \ar@2"s";"t"_{}
      \dpbox{,}
    }
   \qquad
   \qquad
    \xymatrix@C=3pc@R=3pc{
      \bullet
      \ar@/^2ex/[r]^(0.70){}_{}="0"
      \ar@/_2ex/[r]_(0.70){}_{}="1"
      \ar[d]_{}="f"_{}
      \ar@2"0";"1"_{}
      &
      \bullet
      \ar[d]^{} \\
      \bullet
      \ar@{.>}@/^2ex/[r]^(0.30){}_{}="0"
      \ar@/_2ex/[r]_(0.30){}_{}="1"
      \ar@{:>}"0";"1"_{}
      &
      \bullet
      \ar@{}[u];[l]_(.36){}="x"
      \ar@{}[u];[l]_(.64){}="y"
      \ar@<-1.5ex>@/_1ex/@{:>}"x";"y"_(0.60){}_{}="0"
      \ar@<1.5ex>@/^1ex/@2"x";"y"^(0.40){}_{}="1"
      \ar@{}"0";"1"_(.05){}="z"
      \ar@{}"0";"1"_(.95){}="t"
      \ar@3{>}"t";"z"_{}
      \dpbox{,\qquad\qquad\text{etc.},}
    }
    \qquad\qquad
  \]
  and the compositions between cone-shaped diagrams are induced by the
  ones at the level of cylinders. \emph{However}, for Gray \oo-categories,
  as we will explain in an appendix to this paper, it is not possible to
  define all the cylinder compositions that would be required to then deduce
  the ones for cones as a particular case. Hence, the fact that we can
  nonetheless define cone-shaped diagrams and their compositions within a
  Gray \oo-category suggests the existence of an interesting but not yet
  fully understood theory of pasting diagrams in Gray \oo-categories. 
  
\mysubsection*{Towards a Grothendieck construction for Gray \pdfoo-functors}

Let us end this introduction with a noteworthy by-product result that we
obtain in this paper. As noted before, the Grothendieck construction of an
\oo-functor $F \colon I^{\o} \to \ooCat$ can be defined as (the dual of) a
comma \oo-category. In fact, this comma \oo-category is (the dual of) a
relative slice, meaning it is obtained by pulling back a slice \oo-category
as follows:
\[
  \begin{tikzcd}
    (\int_I F)^{\o} \ar[d] \ar[r] & \cotr{\ooCat}{\Dn{0}}\ar[d]\\
    I^{\o} \ar[r,"F"'] & \ooCat \dpbox.
    \ar[from=1-1,to=2-2,phantom,"\lrcorner" description, very near start]
  \end{tikzcd}
\]
The advantage of this description is that it can be adapted
straightforwardly in the context of Gray \oo-categories, up to some
subtleties on dualities. If $\CG$ is a (small) Gray \oo-category, then its
dual obtained by reversing the orientation of all the cells is not a Gray
\oo-category but what we call an \ndef{anti} Gray \oo-category. By that, we
mean a category enriched in $\ooCat$, endowed with the oplax Gray tensor
product, but with the composition morphisms of type
\[
  \Homi(X,Y)\otimes\Homi(Y,Z)\to \Homi(X,Z)
\]
(as opposed to $\Homi(Y,Z)\otimes\Homi(X,Y) \to \Homi(X,Z)$, which is a
different notion as the oplax Gray tensor product is not
symmetrical and not even braided). As an example of an anti Gray \oo-category,
we have $\ooCatLax$, whose $0$\nbd-cells are strict \oo-categories,
$1$\nbd-cells are strict \oo-functors and higher cells are higher \emph{lax}
transformations.
Now, given a (small) Gray \oo-category $\IG$ and $F \colon
\IG^{\o} \to \ooCatLax$ an anti Gray \oo-functor, we define the
(dual of the) Grothendieck construction $\int_{\IG}F$ of $F$ as the
following pullback:
\[
  \begin{tikzcd}
    (\int_{\IG}F)^{\o} \ar[d] \ar[r] &
    \cotr{\ooCatLax}{\Dn{0}}\ar[d] \\
    \IG^{\o} \ar[r,"F"']& \ooCatLax \dpbox{.}
    \ar[from=1-1,to=2-2,phantom,"\lrcorner" description, very near start]
  \end{tikzcd}
\]
Note that $\int_{\IG} F$ is indeed a Gray \oo-category (and not an anti Gray
\oo-category). As already said, the dualities for strict and Gray
\oo-categories are subtle and we study them thoroughly in this paper.

We plan to work more extensively on this Grothendieck construction
in the context of Gray \oo-categories in future work.

\mysubsection*{Acknowledgments}

The authors would like to thank Andrea Gagna, Tommy-Lee Klein, Félix Loubaton, Maxime Lucas
and Georges Maltsiniotis for helpful conversations on the subject of this
paper, as well as the referee for insightful suggestions that improved the paper.

\section{Preliminaries on enriched categories}

We begin with some preliminaries on categories enriched in a monoidal
category. Our case of interest is the category of strict \oo-categories
endowed with the Gray tensor product, which will be introduced in the next
section. This tensor product is \emph{not} symmetric (and not even braided).

\begin{paragr}\label{paragr:skew}
  Let $\V = (\V,\otimes,I)$ be a monoidal category. Since we do not assume
  $\V$ to be symmetric, we need to distinguish between the notion of
  a $\V$\nbd-category (or a category enriched in $\V$) and that of an
  \emph{anti} $\V$\nbd-category. 

  A \ndef{$\V$-category} $\A$ is given by
  \begin{itemize}
    \item a set of \ndef{objects} $\Ob(\A)$,
    \item for all objects $X$ and $Y$ of $\A$, an \ndef{object of
      morphisms} $\Homi_\A(X,Y)$ in $\V$,
    \item for all objects $X$, $Y$ and $Z$ of $\A$, a \ndef{composition
      morphism}
      \[
        \compV \colon
        \Homi_\A(Y,Z)\otimes \Homi_\A(X,Y) \to \Homi_\A(X,Z)
      \]
      in $\V$,
    \item for every object $X$ of $\A$, an \ndef{identity morphism}
      \[ \idV{X} \colon I \to \Homi_\A(X, X) \]
      in $\V$,
  \end{itemize}
  satisfying well-known axioms.

  The notion of an \ndef{anti} $\V$\nbd-category is obtained likewise but
  with composition morphisms of type
  \[
      \compV \colon
      \Homi_\A(X,Y) \otimes \Homi_\A(Y,Z)  \to \Homi_\A(X,Z)
      \mpbox.
  \]
  More formally, an anti $\V$\nbd-category is a $\overline{\V}$\nbd-category,
  where $\Vdual{\V}$ denotes the monoidal category $(\V,\otimesdual,I)$, the
  monoidal product $\otimesdual$ being defined by $X \otimesdual Y = Y
  \otimes X$.

  We denote by $\VCat$ the (large) $2$\nbd-category whose objects are
  $\V$\nbd-categories, whose morphisms are $\V$\nbd-functors and whose
  $2$\nbd-morphisms are $\V$\nbd-natural transformations. Consequently, the
  $2$-category $\VbCat$ is the $2$-category of anti $\V$-categories. Its
  morphisms are called \ndef{anti $\V$-functors} and its $2$-morphisms
  \ndef{anti $\V$-natural transformations}.
\end{paragr}

\begin{paragr}\label{paragr:transpose}
  Let $\V$ be a monoidal category. Given a $\V$\nbd-category $\A$, we define
  its \ndef{transpose} to be the obvious \emph{anti}
  $\V$-category~$\A^{\trans}$ which has the same set of objects as
  $\A$ and whose homs are defined by
  \[
    \Homi_{\A^{\trans}}(X,Y)=\Homi_{\A}(Y,X)
    \mpbox.
  \]
  In particular, applying this to $\Vdual{\V}$, we get that the transpose of an
  anti $\V$\nbd-category is a $\V$\nbd-category.
  
  The correspondence that sends a $\V$\nbd-category to its transpose
  canonically extends to a $2$\nbd-functor
  \[
    \begin{aligned}
      (-)^{\trans} \colon (\VCat)^{\co} &\to \VbCat\\
      \A &\mapsto \A^{\trans}
       \mpbox,
    \end{aligned}
  \]
  where the decoration ``$\co$'' indicates that the orientation of the
  $2$-cells of $\VCat$ is reversed.
\end{paragr}

\begin{paragraph}\label{paragr:enriched_product}
  Let $\V$ be a monoidal category. Recall that if $\V$ admits limits indexed
  by a given category, then so does $\VCat$. In particular, let us describe
  the case of binary products explicitly. If $\A$ and $\B$ are two
  $\V$-categories, their product $\V$-category $\A \times \B$ can be
  described in the following way:
  \begin{itemize}
    \item $\Ob(\A \times \B) = \Ob(\A) \times \Ob(\B)$,
    \item if $X, X'$ are objects of $\A$ and $Y, Y'$ are objects of $\B$,
      then
      \[
        \Homi_{\A \times \B}((X, Y), (X', Y'))
         = \Homi_\A(X,X') \times \Homi_\B(Y,Y')
         \mpbox,
       \]
    \item if $X, X', X''$ are objects of $\A$ and $Y, Y', Y''$ are objects
      of $\B$, the composition of~$\A \times \B$ is given by
      \[
        \begin{tikzcd}
         \big(\Homi_\A(X',X'') \times \Homi_\B(Y',Y'')\big)
         \otimes
         \big(\Homi_\A(X,X') \times \Homi_\B(Y,Y')\big)
         \ar[d, "{(p_1 \otimes p_1, p_2 \otimes p_2)}"']
         \\
          \big(\Homi_\A(X',X'') \otimes \Homi_\A(X,X')\big)
         \times
         \big(\Homi_\B(Y',Y'') \otimes \Homi_\B(Y,Y')\big)
         \ar[d, "\compV \times \compV"']
         \\
          \Homi_\A(X,X'')
           \times
          \Homi_\B(Y,Y'')
          \dpbox,
        \end{tikzcd}
      \]
      where $p_1$ and $p_2$ denote the first and second projections of the
      cartesian product,
    \item if $X$ is an object of $\A$ and $Y$ is an object of $\B$, the
      identity morphism of $(X, Y)$ is
      \[
        \begin{tikzcd}[column sep=3pc]
        I \ar[r, "{(\idV{X}, \idV{Y})}"] & \Homi_\A(X, X) \times \Homi_\B(Y, Y) \mpbox.
        \end{tikzcd}
      \]
  \end{itemize}
\end{paragraph}

\begin{paragraph}\label{paragr:monoidal_push-forward}
  Recall that if $F \colon \V \to \V'$ is a lax monoidal functor between
  monoidal categories, then $F$ induces a $2$-functor
  \[ F_\ast \colon \VCat \to \VpCat \]
  sending a $\V$-category $\A$ to the obvious $\V'$-category $F_\ast(\A)$
  with the same set of objects as $\A$ and
  \[ \Homi_{F_\ast(\A)}(X, Y) = F(\Homi_\A(X, Y)) \mpbox. \]

  We call \ndef{monoidal functor} a lax monoidal functor whose
  structural natural transformations are isomorphisms. We say that a
  functor $F \colon \V \to \V'$ between two monoidal categories is
  \ndef{anti-monoidal} if it is monoidal considered as a functor $F \colon
  \Vdual{\V} \to \V'$, with the notation of~\ref{paragr:skew}.
\end{paragraph}

\begin{paragr}\label{paragr:def_biclosed}
  We say that a monoidal category $\V$ is \ndef{closed on the
  right} if, for every object~$Y$ of $\V$, the functor $-\otimes Y$
  admits a right adjoint, then denoted by $\Homir(Y,\var)$. In this case,
  for $X$, $Y$ and~$Z$ three objects of $\V$, we have a natural
  isomorphism
  \[
    \Hom_{\V}(X\otimes Y, Z) \simeq \Hom_{\V}(X,\Homir(Y,Z))
    \mpbox.
  \]
  If $\V$ is closed on the right, there is an obvious $\V$-category, denoted
  by $\Vr$, whose objects are the same as those of $\V$ and whose homs
  are given by the $\Homir(X,Y)$.

  Dually we say that $\V$ is \ndef{closed on the left} if, for every object
  $X$ of $\V$, the functor $X\otimes \var$ admits a right adjoint,
  denoted by $\Homil(X,\var)$. We then have a natural isomorphism
  \[
       \Hom_{\V}(X\otimes Y, Z) \simeq \Hom_{\V}(Y,\Homil(X,Z))
       \mpbox.
  \]
  If $\V$ is closed on the left, there is an obvious anti $\V$-category,
  denoted by $\Vl$, whose objects are the same as those of $\V$ and whose
  homs are given by the $\Homil(X,Y)$.

  We say that $\V$ is \ndef{biclosed} if it is closed both on the left and
  on the right. In this case, we have a canonical isomorphism
  \[
    \Homil(X,\Homir(Y,Z))\simeq  \Homir(Y,\Homil(X,Z)) \mpbox,
  \]
  natural in $X$, $Y$ and $Z$ in $\V$.
\end{paragr}

\begin{paragraph}\label{paragr:times_enriched}
  Let $\V$ be a monoidal category closed on the right. Suppose that $\V$
  admits binary products. Then the product defines a $\V$-functor
  \[ \times \colon \Vr \times \Vr \to \Vr \mpbox. \]
  This $\V$-functor is given on objects by
  \[ (X, Y) \mapsto X \times Y \mpbox, \]
  and, if $X,X',Y,Y'$ are four objects of $\V$, on homs by the morphism
  \[
      \Homir(X, X') \times \Homir(Y, Y') 
      \to
      \Homir(X \times Y, X' \times Y')
  \]
  obtained by adjunction from the composite
  \[
    \begin{tikzcd}[column sep=3pc]
      \big(\Homir(X, X') \times \Homir(Y, Y')\big)
      \otimes
      \big(X \times Y\big)
      \ar[d, "{(p_1 \otimes p_1, p_2 \otimes p_2)}"']
      \\
      \big(\Homir(X, X') \otimes X\big)
      \times
      \big(\Homir(Y, Y') \otimes Y\big)
      \ar[r, "\ev \times \ev"]
      &
      X' \times Y' \mpbox,
    \end{tikzcd}
  \]
  where $p_1$ and $p_2$ denote the two projections of the product,
  and $\ev$ the evaluation morphism of the right internal hom.

  Similarly, if $\V$ is closed on the left, we have a canonical anti
  $\V$-functor
  \[ \times \colon \Vl \times \Vl \to \Vl \mpbox. \]
\end{paragraph}

\begin{paragr}
  Suppose now that $\V$ is a biclosed monoidal category and let $\A$ be a
  $\V$\nbd-category. A \ndef{copresheaf} over $\A$ is a $\V$\nbd-functor
  \[
    F \colon \A \to \Vr \mpbox.
  \]
  We denote by $\Homi(\A,\Vr)$ the category of copresheaves over $\A$ and
  $\V$\nbd-natural transformations between them.

  A \ndef{presheaf} over $\A$ is an anti $\V$\nbd-functor
  \[
    F \colon \A^\trans \to \Vl \mpbox.
  \]
  We denote by $\Homi(\A^{\trans},\Vl)$ the category of presheaves over $\A$
  and anti $\V$\nbd-natural transformations between them.
\end{paragr}

\begin{example}
  Let $a$ be an object of a $\V$\nbd-category $\A$. Then
  $\Homi_{\A}(a,\var)$ is a copresheaf over $\A$, and
  $\Homi_{\A}(\var,a)$ is a presheaf over $\A$.
\end{example}

As we shall now see, copresheaves and presheaves admit a useful description
in terms of left and right modules.

\begin{paragr}\label{paragr:def_module}
  Let $\V$ be a biclosed monoidal category and let $\A$ be a $\V$\nbd-category.
  A \ndef{right $\A$\nbd-module}
  consists of
  \begin{itemize}
    \item a family $(F_a)_{a \in \Ob(\A)}$ of objects of
      $\V$, 
    \item for all objects $a$ and $a'$ of $\A$, a morphism of $\V$
      \[
        \rho_{a,a'} \colon F_{a'} \otimes \Homi_{\A}(a,a') \to F_{a}
        \mpbox,
      \]
  \end{itemize}
  such that
  \begin{itemize}                 
    \item for all objects $a$, $a'$ and $a''$ of $\A$, the diagram
      \[
        \begin{tikzcd}[column sep=3pc]
          F_{a''} \otimes \Homi_{\A}(a',a'')\otimes \Homi_{\A}(a,a')
          \ar[r,"F_{a''} \otimes \compV"]
          \ar[d,"{\rho_{a',a''} \otimes \Homi_\A(a, a')}"']
            &
            F_{a''} \otimes \Homi_{\A}(a,a'')
            \ar[d,"{\rho_{a,a''}}"]
            \\
            F_{a'} \otimes \Homi_{\A}(a,a')
            \ar[r,"{\rho_{a,a'}}"']
            &
            F_{a}
        \end{tikzcd}
      \]
      commutes,
    \item for every object $a$ of $\A$, the diagram
      \[
        \begin{tikzcd}[column sep=3pc]
          F_a \simeq F_a \otimes I
          \ar[r,"F_a \otimes 1_a"]
          \ar[dr,"="']
      &
      F_a \otimes \Homi_{\A}(a,a)
      \ar[d,"\rho_{a,a}"]
      \\
      &
      F_a
        \end{tikzcd}
      \]
      commutes.
  \end{itemize}

  If $F$ and $F'$ are two right $\A$-modules, a \ndef{morphism of right
  $\A$\nbd-modules} $u\colon F \to F'$ consists of a family $(u_a \colon
  F^{}_a \to F'_a)_{a \in \Ob(\A)}$ of morphisms of $\V$ such that for all
  objects~$a$ and $a'$ of $\A$, the square
  \[
    \begin{tikzcd}
      F_{a'} \otimes \Homi_{\A}(a,a')
      \ar[r,"{\rho_{a,a'}}"]
      \ar[d,"{u_{a'} \otimes \Homi_\A(a, a')}"']
      &
      F_{a}
      \ar[d,"u_{a}"]
      \\
      F'_{a'} \otimes \Homi_{\A}(a,a')
      \ar[r,"{\rho'_{a,a'}}"']
      &
      F'_{a}
    \end{tikzcd}
  \]
  commutes. Right $\A$-modules and their morphisms form a category that we
  will denote by $\RMod{\A}$.

  The notions of \ndef{left $\A$\nbd-module} and of \ndef{morphism of left
  $\A$-modules} are defined analogously but with an action morphism
  of type
  \[
    \lambda_{a,a'} \colon F_a \otimes \Homi_{\A}(a,a') \to F_{a'}
    \mpbox.
  \]
 We denote the category of left $\A$\nbd-modules by $\Mod{\A}$. 
\end{paragr}

\begin{proposition}\label{prop:modules_as_functors}
  Let $\V$ be a biclosed monoidal category and let $\A$ be a
  $\V$\nbd-category. We have isomorphisms of categories
  \[
    \Mod{\A} \overset{\sim}{\to} \Homi(\A,\Vr) \mpbox,
  \]
  and
  \[
    \RMod{\A} \overset{\sim}{\to} \Homi(\A^{\trans},\Vl) \mpbox.
  \]
\end{proposition}

\begin{proof}
  This follows from the adjunctions between the tensor product and $\Homir$
  and $\Homil$. We leave the details to the reader.
\end{proof}

\begin{remark}
  In practice, in this paper, we will produce presheaves and then use the
  previous proposition to obtain right modules and do computations using the
  laws of right modules.
\end{remark}

As we shall see in the next section, one specific property of the Gray
tensor product is that its monoidal unit is the terminal object. We now
develop some enriched category theory with this additional hypothesis.

\begin{paragraph}\label{paragr:semicart_def}
  A monoidal category $(\V, \otimes, I)$ is said to \ndef{have projections}
  if $\V$ admits finite products and the tensor unit of $\V$ is a terminal
  object. In this case, if $X$ and $Y$ are two objects of $\V$, we get
  ``projections''
  \[
    \begin{tikzcd}[column sep=2.2pc]
      X 
      &
      X \otimes I
      \ar[l, "\sim"']
      &
      X \otimes Y
      \ar[l, "{X \otimes p^{}_Y}"']
      \ar[r, "{p^{}_X \otimes Y}"]
      &
      I \otimes Y
      \ar[r, "\sim"]
      &
      Y
      \dpbox,
    \end{tikzcd}
  \]
  where $p^{}_Z \colon Z \to I$ denotes the unique morphism to the terminal object.
  In particular, we get a morphism
  \[ \pi = (\pi_1, \pi_2) \colon X \otimes Y \to X \times Y  \mpbox, \]
  natural in $X$ and $Y$ in $\V$.

  The cartesian product defines a monoidal structure on $\V$ and we will
  denote by~$\V^\times$ the resulting monoidal category. By default, $\V$
  will be endowed with the monoidal product $\otimes$ but sometimes, to
  emphasize this, we will denote this monoidal category by $\V^\otimes$.

  With this notation, the morphism $\pi$ shows that the identity functor of
  $\V$ is a lax monoidal functor from $\V^\times$ to $\V^\otimes$. It is
  also lax monoidal considered with values in~$\overline{\V^\otimes}$. In
  particular, we get $2$-functors
  \[
    \ECat{\V^\times} \to \ECat{\V^\otimes}
    \quadand
    \ECat{\V^\times} \to \ECat{\overline{\V^\otimes}}
    \mpbox.
  \]

  If the morphism $\pi \colon X \otimes Y \to X \times Y$ is an epimorphism for
  all objects $X$ and~$Y$ of~$\V$, we will say that $\V$ has
  \ndef{jointly surjective projections}. In this case, the two $2$-functors
  above are injective on objects and fully faithful.

  Suppose moreover that $\V$ is cartesian closed and that $\V^\otimes$ is
  monoidal biclosed. We will denote by $\Vcart$ the $\V^\times$-category
  obtained by enriching $\V$ over itself using the cartesian internal
  hom. In other words, if $X$ and~$Y$ are two objects of $\V$,
  then $\Homi_{\Vcart}(X, Y) = \Homi(X, Y)$, where $\Homi$ denotes the
  cartesian internal hom of~$\V$. Based on the above, $\Vcart$ can also
  be considered as either a $\V^\otimes$-category or a
  $\overline{\V^\otimes}$\nbd-category. By the Yoneda lemma, using the
  morphism~$\pi$, we get canonical morphisms
  \[
    \Homi_{\Vcart}(X, Y) \to \Homir(X,Y) \quadand \Homi_{\Vcart}(X, Y) \to
    \Homil(X, Y) \mpbox.
  \]
  These morphisms are monomorphisms if $\V$ has jointly surjective
  projections. In any case, they induce a $\V^\otimes$\nbd-functor and a
  $\overline{\V^\otimes}$\nbd-functor
  \[ 
    \Vcart \to \Vr
    \quadand
    \Vcart \to \Vl \mpbox.
  \]
\end{paragraph}

\begin{paragraph}\label{paragr:semicart_strength}
  Let $\V$ be a monoidal category with projections. If $X$, $Y$ and $Z$
  are three objects of $\V$, we have a canonical natural morphism
  \[
    \phi \colon X \otimes (Y \times Z) \to (X \otimes Y) \times Z \mpbox,
  \]
  given on components by
  \[
    X \otimes (Y \times Z) \xto{X \otimes p_1} X \otimes Y
    \quadand
    X \otimes (Y \times Z) \xto{\pi_2} Y \times Z \xto{p_2} Z
    \mpbox,
  \]
  where $p_1$ and $p_2$ denote the projections of the cartesian product, and                                                                                           
  $\pi_1$ and $\pi_2$ the ``projections'' of the tensor product.
  We will not need this but one can actually show that $\phi$ is a tensorial
  strength on the functor $\var \times Z$.

  Suppose moreover that $\V$ is cartesian closed and monoidal biclosed. Then
  the morphism $\phi$ induces a natural morphism
  \[
    \lambda \colon \Homil(X, \Homi(Y, Z)) \to \Homi(Y, \Homil(X, Z))
  \]
  that makes the following square commutative
  \[
    \begin{tikzcd}
      \Homil(X, \Homi(Y, Z))
      \ar[r, "\lambda"]
      \ar[d]
      &
      \Homi(Y, \Homil(X, Z))
      \ar[d]
      \\
      \Homil(X, \Homir(Y, Z))
      \ar[r, "\sim"]
      &
      \Homir(Y, \Homil(X, Z))
      \dpbox,
    \end{tikzcd}
  \]
  where the vertical arrows are induced by the canonical morphism from
  $\Homi$ to $\Homir$. Explicitly, the morphism $\lambda$ is obtained by
  adjunction from
  \[
    \begin{tikzcd}[column sep=2.5pc]
       X \otimes \big(\Homil(X, \Homi(Y, Z)) \times Y\big)
       \ar[d, "\phi"']
       \\
       \big(X \otimes \Homil(X, \Homi(Y, Z))\big) \times Y
       \ar[r, "{\ev \times Y}"]
       &
       \Homi(Y, Z) \times Y
       \ar[r, "{\ev}"]
       &
       Z
       \dpbox,
    \end{tikzcd}
  \]
where $\ev$ denotes the evaluation morphism. It follows from the
commutative square above that if $\V$ has jointly surjective
projections, then $\lambda$ is a monomorphism.
\end{paragraph}

\begin{remark}
  The previous paragraph can be dualized to $\Homir$. In
  particular, there is a canonical natural morphism
  \[
    \Homir(X, \Homi(Y, Z)) \to \Homi(Y, \Homir(X, Z))
    \mpbox.
  \]
\end{remark}

\section{Preliminaries on strict \pdfoo-categories and Gray
\pdfoo-categories}

\begin{paragraph}
  For any $n \ge 1$, we denote by $\nCat{n}$ the category of (small) strict
  $n$\nbd-categories and strict $n$-functors, that is, the category
  of categories enriched in $\nCat{(n-1)}$ (with the cartesian monoidal
  structure), the category $\nCat{0}$ being the category $\Set$ of sets.

  We have a canonical inclusion
  \[
    \nCat{(n-1)} \hookrightarrow \nCat{n}
  \]
  sending a strict $(n-1)$\nbd-category to the strict $n$\nbd-category
  obtained by adding only trivial $n$-cells (that is, identities). This
  inclusion admits a right adjoint
  \[ \tau_{n-1} \colon \nCat{n} \to \nCat{(n-1)} \mpbox, \]
  and the category $\ooCat$ of strict \oo-categories and
  strict \oo-functors is obtained as the limit of the diagram
  \[
    \begin{tikzcd}
      \cdots \ar[r]&\nCat{2} \ar[r,"\tau_1"] & \nCat{1}
      \ar[r,"\tau_0"] & \nCat{0} \dpbox.
    \end{tikzcd}
  \]
  For any $n\ge 0$, we have a canonical fully faithful functor
  \[\nCat{n} \hookto \ooCat \mpbox,\]
  admitting both a left and a right adjoint, and whose image consists
  exactly of those strict \oo-categories with only trivial $k$\nbd-cells for
  $k > n$. We shall always consider the previous fully faithful functor as
  an inclusion.

  From now on, we will drop the adjective ``strict'' and simply refer to
  ``strict \oo-categories'' and ``strict \oo-functors'' as ``\oo-categories''
  and ``\oo-functors'' (and similarly for ``$n$-categories'' and
  ``$n$-functors'').
\end{paragraph}

Let us introduce some notation.

\begin{paragraph}
Let $C$ be an \oo-category. If $x$ is a $k$\nbd-cell of $C$, for $0\le i 
\le k$, we denote by
  \[
    s_i(x) \quadand t_i(x)
  \]
the $i$\nbd-dimensional source and target of $x$, also called $i$-source and
$i$-target of $x$. In the case where $i=k-1$, we also write $s(x)$ and
$t(x)$. For $i \ge k$, we denote by
\[
  \id{x}^i
\]
the $i$\nbd-dimensional unit of $x$. When $i=k+1$, we also write
$\id{x}$.

  Given two $k$\nbd-cells $x$ and $y$ such that $s_i(x)=t_i(y)$, with $i <
  k$, we denote by
  \[
    x\comp_i y
  \]
  their $i$\nbd-composition. More generally, for $k > i$ and $l > i$, $x$ a
  $k$\nbd-cell and $y$ an $l$\nbd-cell such that $s_i(x)=t_i(y)$, then if $k
  \le l$, we set
  \[
    x\comp_iy =\id{x}^l\comp_iy\mpbox,
  \]
  and if $k \ge l$, we set
  \[
    x\comp_iy =x\comp_i\id{y}^l\mpbox.
  \]
\end{paragraph}

\begin{paragraph}\label{paragr:def_Dn}
  For any $n \ge 0$, we define $\Dn{n}$ to be the $n$\nbd-category freely
  generated by a generic $n$\nbd-cell. Explicitly, $\Dn{n}$ has exactly one
  non-trivial $n$\nbd-cell $e_n$, and exactly two non-trivial $k$\nbd-cells
  for $0 \le k < n$ given by $s_k(e_n)$ and $t_k(e_n)$.
  \[
      \Dn{0}= \begin{tikzcd} \bullet \end{tikzcd}, \quad
      \Dn{1}= \begin{tikzcd}[column sep=2pc] \bullet \ar[r] & \bullet \end{tikzcd}, \quad
      \Dn{2}= \begin{tikzcd}[column sep=2.5pc]\bullet \ar[r,bend
        left,""{name=A,below}]\ar[r,bend
        right,""{name=B}]&\bullet\ar[from=A,to=B,Rightarrow]\end{tikzcd},\quad \cdots
  \]
  In particular, $\Dn{n}$ represents the functor $\ooCat
  \to \Set$ that sends an \oo-category to its set of
  $n$\nbd-cells. If $x$ is an $n$\nbd-cell of an \oo-category $C$, we
  will denote by $\tilde{x} \colon \Dn{n} \to C$ the corresponding
  \oo-functor.
\end{paragraph}

\begin{paragraph}
  The category $\ooCat$ has several interesting monoidal
  structures. First, it is cartesian closed and the associated
  internal hom will be denoted by
  \[
    \Homi(A,B)\mpbox,
  \]
  for $A$ and $B$ two \oo-categories.
  The $n$\nbd-cells of this \oo-category are in bijection with the
  \oo-functors
  \[
    \Dn{n}\times A \to B\mpbox.
  \]
  In particular, the $0$\nbd-cells are simply the \oo-functors $A \to B$. For $n\ge 1$, the
  $n$\nbd-cells are referred to as \ndef{strict $n$\nbd-transformations}.
  Explicitly, given two \oo-functors $u,v \colon A \to B$, a strict
  $n$\nbd-transformation $\alpha$, with $0$\nbd-source $u$ and
  $0$\nbd-target $v$, is a family $(\alpha_{x})$ of $n$\nbd-cells of $B$,
  indexed by the $0$\nbd-cells $x$ of $A$, such
  that
  \begin{itemize}
    \item for every $0$-cell $x$ of $A$, we have
      \[ s_0(\alpha_x)=u(x) \quadand t_0(\alpha_x)=v(x) \mpbox, \]
    \item for every $k$\nbd-cell $x$ of $A$, with $k\ge 1$, we have
      \[
        \alpha_{t_0(x)} \ast_0 u(x) = v(x) \ast_0 \alpha_{s_0(x)} \mpbox.
      \]
\end{itemize}
The category $\ooCat$ is enriched over itself via the
cartesian product, and we have a ``fixed point'' property: the
category of categories enriched in $(\ooCat,\times,\Dn{0})$ is canonically
isomorphic to $\ooCat$ itself. In particular, we have a (large)
\oo\nbd-category
\[
  \ooCatCart \mpbox,
\]
whose $0$\nbd-cells are (small)
\oo\nbd-categories, whose $1$\nbd-cells are \oo\nbd-functors and whose
$n$\nbd-cells, with $n> 1$, are strict $(n-1)$\nbd-transformations. 
\end{paragraph}

\begin{paragraph}
  Another fundamental monoidal structure on $\ooCat$ comes from the
  so-called \ndef{(oplax) Gray tensor product} (see for example
  \cite[Appendix A]{AraMaltsiJoint}), denoted by~$\otimes$. To give an
  intuition, the tensor product $\Dn{1}\otimes\Dn{1}$ is a square with a
  non-trivial $2$\nbd-cell
  \[
    \begin{tikzcd}[column sep=2pc, row sep=2pc]
      \bullet \ar[r] \ar[d] & \bullet \ar[d]\\
      \bullet \ar[r] & \bullet
      \dpbox,
      \ar[to=2-1,from=1-2,Rightarrow,shorten=3mm]
    \end{tikzcd}
  \]
  whereas the cartesian product $\Dn{1}\times \Dn{1}$ is a commutative square.
  The monoidal  unit is the terminal \oo-category $\Dn{0}$ (as for the cartesian
  structure) and the Gray tensor product thus defines a monoidal structure
  with projections in the sense of~\ref{paragr:semicart_def}.
  This monoidal structure is \emph{not} symmetrical (and not even braided).
  For example, we have
  \[
    \Dn{1}\otimes\Dn{2} =
    \begin{tikzcd}[column sep=2.5pc, row sep=2pc]
      \bullet \ar[d] \ar[r,bend left,""{name=A,below}]
      \ar[r,bend right,""{name=B}] &\bullet\ar[d]
      \\
      \bullet \ar[r] &
      \bullet \ar[from=A,to=B,Rightarrow]
      \ar[from=1-2,to=2-1,Rightarrow,shorten=4mm]
    \end{tikzcd}
      \Rrightarrow
    \begin{tikzcd}[column sep=2.5pc, row sep=2pc]
      \bullet \ar[d] \ar[r] & \bullet \ar[d] \\
      \bullet \ar[r,bend left,""{name=C,below}] \ar[r,bend right,""{name=D}] &
      \bullet
      \ar[from=C,to=D,Rightarrow]\ar[from=1-2,to=2-1,Rightarrow,shorten=4mm]
    \end{tikzcd}
  \]
  and
  \[
    \Dn{2}\otimes\Dn{1} =
    \begin{tikzcd}[column sep=2.5pc, row sep=2pc]
      \bullet \ar[r] \ar[d,bend left,""{name=A,left}]
      \ar[d,bend right,""{name=B,left}] & \bullet \ar[d] \\
      \bullet \ar[r] & \bullet
      \ar[from=1-2,to=2-1,Rightarrow,shorten=4mm]
      \ar[from=A,to=B,Rightarrow,end anchor={[xshift=0.7mm]},start
      anchor={[xshift=0.5mm]}]
    \end{tikzcd}
      \Rrightarrow
    \begin{tikzcd}[column sep=2.5pc, row sep=2pc]
      \bullet \ar[r] \ar[d] & \bullet
      \ar[d,bend left,""{name=C,left}] \ar[d,bend right,""{name=D,left}] \\
      \bullet \ar[r] & \bullet
      \ar[from=1-2,to=2-1,Rightarrow,shorten=4mm]\ar[from=C,to=D,Rightarrow,end
      anchor={[xshift=0.7mm]},start anchor={[xshift=0.5mm]}]
      \dpbox.
    \end{tikzcd}
  \]
\end{paragraph}

\begin{paragraph}\label{paragr:def_principal_cell}
  Let $A$ and $B$ be two \oo-categories. If $x$ is a $k$-cell of $A$ and $y$
  an $l$-cell of $B$, then there is an associated $(k+l)$-cell $x \otimes y$
  in $A \otimes B$. Explicitly, with the notation of \ref{paragr:def_Dn},
  this cell corresponds to the \oo-functor
  \[
    \begin{tikzcd}
      \Dn{k+l}
      \ar[r, "\tilde{c}"]
    &
    \Dn{k} \otimes \Dn{l}
    \ar[r, "\tilde{x} \otimes \tilde{y}"]
    &
    A \otimes B
    \mpbox,
    \end{tikzcd}
  \]
  where $c$ denotes the \ndef{principal cell} of $\Dn{k} \otimes \Dn{l}$,
  that is, its unique non-trivial $(k+l)$\=/cell (see for instance
  \cite[paragraph B.1.5]{AraMaltsiThmAII}).

  We will not need this but one can show that cells of the form $x \otimes
  y$ generate $A \otimes B$ under composition.
\end{paragraph}

\begin{paragraph}\label{paragr:adjunction_lax_oplax}
  The Gray tensor product is biclosed, with right and left internal homs
  denoted by $\HomOpLax$ and $\HomLax$, respectively, so that we have
  \[
    \begin{aligned}
      \Hom(A\otimes B, C) &\simeq \Hom(A, \HomOpLax(B,C))\\
      &\simeq \Hom(B,\HomLax(A,C))\mpbox,
    \end{aligned}
  \]
  for $A$, $B$ and $C$ three \oo-categories. Moreover, by
  \ref{paragr:def_biclosed}, this last bijection can be promoted to a
  natural isomorphism
  \[
    \HomLax(A,\HomOpLax(B,C)) \simeq \HomOpLax(B,\HomLax(A,C)) \mpbox.
  \]

  The $n$\nbd-cells of $\HomOpLax(A,B)$ (resp.~of $\HomLax(A,B)$)
  are in bijection with the \oo-functors
  \[
    \Dn{n}\otimes A \to B \quad\text{(resp.\ } A\otimes \Dn{n}\to B)\mpbox.
  \]
  In particular, the $0$\nbd-cells of both $\HomOpLax(A,B)$ and
  $\HomLax(A,B)$ are simply the \oo-functors $A \to B$. For $n\ge 1$, an
  $n$\nbd-cell of $\HomOpLax(A,B)$ (resp.~of $\HomLax(A,B)$) is called an
  \ndef{oplax $n$\nbd-transformation} (resp.~a \ndef{lax
  $n$\nbd-transformation}). For $n=1$, we will simply say \ndef{oplax
  transformations} (resp.~\ndef{lax transformations}). For an explicit
  description of oplax transformations, see for example \cite[paragraph
  1.9]{AraMaltsiJoint}.
\end{paragraph}

\begin{paragraph}\label{paragr:def_cyl}
  By definition, if $u, v \colon A \to B$ are two \oo-functors, an oplax
  transformation~$\alpha$ from $u$ to $v$ corresponds to an \oo-functor $h
  \colon \Dn{1} \otimes A \to B$ making the following diagram commutative
  \[
    \begin{tikzcd}
      \Dn{0} \otimes A
      \ar[d, "\sigma \otimes A"']
      \ar[r, phantom, "\simeq"]
      &[-1.8pc]
      A
      \ar[dr, "u", bend left=20]
      \\
      \Dn{1} \otimes A \ar[rr, "h"] && A
      \\
      \Dn{0} \otimes A
      \ar[u, "\tau \otimes A"]
      \ar[r, phantom, "\simeq"]
                                    &
      A
      \ar[ur, "v"', bend right=20]
                                    & \dpbox,
    \end{tikzcd}
  \]
  where $\sigma, \tau \colon \Dn{0} \to \Dn{1}$ correspond to the source and the
  target of the non-trivial $1$-cell of $\Dn{1}$. By adjunction, it also
  corresponds to an \oo-functor $k \colon A \to \HomLax(\Dn{1}, B)$ making the
  diagram
  \[
    \begin{tikzcd}
      &
      B
      \ar[r, phantom, "\simeq"]
      &[-1.7pc]
      \HomLax(\Dn{0}, B) 
      \\
      A
      \ar[rr, "k"]
      \ar[ur, "u", bend left=20]
      \ar[dr, "v"', bend right=20]
      & &
      \HomLax(\Dn{1}, B)
      \ar[u, "{\HomLax(\sigma, B)}"']
      \ar[d, "{\HomLax(\tau, B)}"]
      \\
      &
      B
      \ar[r, phantom, "\simeq"]
      &
      \HomLax(\Dn{0}, B) 
    \end{tikzcd}
  \]
  commute. This leads us to set
  \[ \Cyl B = \HomLax(\Dn{1}, B) \mpbox. \]
  This \oo-category is the \ndef{\oo-category of cylinders in $B$}. A
  $k$-cell in this \oo-category is called a \ndef{$k$-cylinder in $B$}. In
  other words, a $k$-cylinder in $B$ is an \oo-functor $\Dn{1} \otimes
  \Dn{k} \to B$. If $\beta \colon \Dn{1} \otimes \Dn{k} \to B$ is such a
  $k$-cylinder, the image of the principal cell of $\Dn{1} \otimes \Dn{k}$
  (see~\ref{paragr:def_principal_cell}) in $B$ will be called the \ndef{principal
  cell} of $\beta$. We will denote it by $\beta_k$.

  As a particular case of the compatibilities between $\HomLax$ and
  $\HomOpLax$ (see \ref{paragr:adjunction_lax_oplax}), we get that
  if $A$ and $B$ are two \oo-categories, then we have a natural isomorphism
  \[
    \Cyl \HomOpLax(A, B) \simeq \HomOpLax(A, \Cyl{B})
    \mpbox.
  \]
  This isomorphism will play an important role in this paper.
\end{paragraph}

In \cite[Section 4]{LafMetPolRes} (see also \cite[Appendix A]{LMW}), the
authors describe the \oo-category~$\Cyl{C}$ inductively. The two following
paragraphs provide preliminaries to phrase their description.

\begin{paragraph}\label{paragr:compc}
  If $C$ is an \oo-category, then the \oo-category $\Cyl{C}$ is naturally
  the object of morphisms of a category internal to $\ooCat$. Indeed, the
  functors
  \[
    \sigma, \tau \colon \Dn{0} \to \Dn{1}, \quad
    \kappa \colon \Dn{1} \to \Dn{0}
    \quadand
    \nabla \colon \Dn{1} \to \Dn{1} \amalg_{\Dn{0}} \Dn{1}
    \mpbox,
  \]
  corresponding respectively to the source and the target of the unique
  non-trivial $1$-cell of $\Dn{1}$, the unit of the unique object of
  $\Dn{0}$ and the total composition of $\Dn{1} \amalg_{\Dn{0}} \Dn{1}$,
  define a cocategory internal to categories, and hence internal to
  \oo-categories. By applying the functor $\HomLax(\var, C)$ which
  sends colimits to limits, we get \oo-functors
  \[
    \ss, \tt \colon \Cyl{C} \to C,
    \quad
    \kk \colon C \to \Cyl{C}
    \quadand
    \compc \colon \Cyl{C} \times_C \Cyl{C} \to \Cyl{C}
  \]
  defining a structure of category internal to \oo-categories. If $x$ is a
  cell of $C$, we will denote $\kk(x)$ by $\idd{x}$.
\end{paragraph}

\begin{paragraph}\label{paragr:compr_str}
  Let $C$ be an \oo-category and let $c$ and $d$ be two objects of $C$.
  For every $1$-cell $u \colon c' \to c$, we have an \oo-functor
  \[ \Cyl\Homi_C(u, d) \colon \Cyl\Homi_C(c, d) \to \Cyl\Homi_C(c', d)
  \mpbox. \]
  If $\alpha$ is a cell in $\Cyl\Homi_C(c,d)$, we will denote its image
  under this \oo-functor by $\alpha \compr u$.

  More generally, as the \oo-functor $\Gamma = \HomLax(\Dn{1}, \var)$
  commutes with limits, the composition \oo-functor
  \[
    \Homi_C(c, d) \times \Homi_C(c', c) \to \Homi_C(c', d)
  \]
  induces an \oo-functor
  \[
    \Cyl\Homi_C(c, d) \times \Cyl\Homi_C(c', c) \to \Cyl\Homi_C(c', d)
    \mpbox,
  \]
  and if $\alpha$ is a $k$-cell of $\Cyl\Homi_C(c,d)$ and $u$ a $k$-cell of
  $\Homi_C(c', c)$, we define $\alpha \compr u$ to be the image of the pair
  $(\alpha, \idd{u})$ by this \oo-functor.

  Similarly, if $v \colon d \to d'$ is a $1$-cell of $C$, we have an
  \oo-functor
  \[ \Cyl\Homi_C(c, v) \colon \Cyl\Homi_C(c, d) \to \Cyl\Homi_C(c, d')
  \mpbox, \]
  and if $\alpha$ is a cell in $\Cyl\Homi_C(c,d)$, its image will be denoted
  by $v \compl \alpha$. More generally, if $v$ is a $k$-cell of
  $\Homi_C(d, d')$ and $\alpha$ is a $k$-cell of $\Cyl\Homi_C(c,d)$, we define
  $v \compl \alpha$ to be the image of the pair $(\idd{v}, \alpha)$ by the
  \oo-functor
  \[
    \Cyl\Homi_C(d, d') \times \Cyl\Homi_C(c, d) \to \Cyl\Homi_C(c, d') 
    \mpbox.
  \]
\end{paragraph}

\begin{remark}
  We will come back to these operations in terms of modules
  in~\ref{paragr:compr} and Remark~\ref{rem:compl}.
\end{remark}

\begin{paragraph}\label{paragr:enriched_cyl}
  Let $C$ be an \oo-category. By \cite[Section 4]{LafMetPolRes}, the
  \oo-category $\Cyl{C}$ can be described (up to isomorphism) as a category
  enriched in \oo-categories in the following way:
  \begin{itemize}
    \item The objects of $\Cyl{C}$ are the $1$-cells of $C$.
    \item If $f \colon c \to d$ and $f' \colon c' \to d'$ are two objects of
      $\Cyl{C}$, we have
      \[ 
        \begin{split}
          \MoveEqLeft
        \Homi_{\Cyl{C}}(f, f')
        \\
        & =
        \Homi_C(c, c')
        \times_{\Homi_C(c, d')} \Cyl{\Homi_C(c, d')} \times_{\Homi_C(c, d')}
        \Homi_C(d, d')
        \mpbox,
        \end{split}
      \]
      where this iterated fiber product denotes the limit of the diagram
      \[
        \begin{tikzcd}[column sep=0pc]
          \Homi_C(c, c') \ar[dr, "{f' \comp_0 \var}"'] & & \Cyl{\Homi_C(c, d')}
          \ar[dl, "\ss"] \ar[dr, "\tt"']
                                       & &
                                       \Homi_C(d, d') \ar[dl, "{\var \comp_0
                                       f}"]
                                       \\
                                       & \Homi_C(c, d') & & \Homi_C(c, d') &
                                       & \mpbox.
        \end{tikzcd}
      \]
      Concretely, a $k$-cell in this hom is a triple
      $(u, \alpha, v)$ in
      \[
        \Homi_C(c, c')_k \times \Cyl{\Homi_C(c, d')}_k
        \times \Homi_C(d, d')_k
      \]
      such that
      \begin{equation}\label{eq:cyl}
        \ss(\alpha) = f' \comp_0 u \quadand \tt(\alpha) = v \comp_0 f
        \mpbox.
        \tag{$\ast$}
      \end{equation}
      This formula is an \oo-categorification of the formula for a
      $1$-cylinder, i.e., a $2$\=/square:
      \[
        \xymatrix@C=3pc@R=3pc{
          c \ar[r]^u \ar[d]_f &
          c' \ar[d]^{f'} \\
          d \ar[r]_v & d'
          \ar@{}[u];[l]_(.30){}="s"
          \ar@{}[u];[l]_(.70){}="t"
          \ar@2"s";"t"_{\alpha}
          \dpbox{.}
        }
      \]
      In other words, a $(k+1)$-cylinder in $C$ is given by its
      $0$-source $f \colon c \to d$ and its $0$-target $f' \colon c' \to
      d'$, a $(k+1)$-cell $u$ in $C$ of $0$-source $c$ and $0$-target $c'$, a
      $(k+1)$-cell $v$ in $C$ of $0$-source $d$ and $0$-target $d'$, and a
      $k$-cylinder $\alpha$ in $\Homi_C(c, d')$ satisfying the relations
      \eqref{eq:cyl}.
    \item If $f, f', f''$ are three objects of $\Cyl C$, the
      composition \oo-functor
      \[
        \Homi_{\Cyl C}(f', f'') \times \Homi_{\Cyl C}(f, f')
        \to
        \Homi_{\Cyl C}(f, f'')
      \]
      is given by
      \[
        ((u', \alpha', v'), (u, \alpha, v))
        \mapsto
        (u' \comp_0 u,
        (v' \compl \alpha) \compc (\alpha' \compr u),
        v' \comp_0 v)
        \mpbox.
      \]
      This formula is an \oo-categorification of the formula for composing
      the diagram
      \[
        \xymatrix@C=3pc@R=3pc{
          c \ar[r]^u \ar[d]_f &
          c' \ar[r]^{u'} \ar[d]_{f'} &
          c'' \ar[d]^{f''} 
          \\
          d \ar[r]_{v\phantom{'}} &
          \ar@{}[u];[l]_(.30){}="s"
          \ar@{}[u];[l]_(.70){}="t"
          \ar@2"s";"t"_{\alpha\phantom'}
          d' \ar[r]_{v'} &
          d''
          \ar@{}[u];[l]_(.30){}="s"
          \ar@{}[u];[l]_(.70){}="t"
          \ar@2"s";"t"_{\alpha'}
          \dpbox{.}
        }
      \]
    \item If $f \colon c \to d$ is an object of $\Cyl C$, its unit is the
      $1$-cylinder $(\id{c}, \idd{f}, \id{d})$, corresponding to the
      commutative square
      \[
        \xymatrix@C=3pc@R=3pc{
          c \ar[r]^{\id c} \ar[d]_f &
          c \ar[d]^{f} \\
          d \ar[r]_{\id d} & d
          \ar@{}[u];[l]_(.30){}="s"
          \ar@{}[u];[l]_(.70){}="t"
          \ar@2"s";"t"_{\id{f}}
          \dpbox{.}
        }
      \]
  \end{itemize}
\end{paragraph}

We now come to one of the central notions of this paper, the notion of a
Gray \oo-category, introduced in~\cite[Appendix C]{AraMaltsiJoint}.

\begin{paragraph}
A \ndef{Gray \oo-category} is a $\V$-category for $\V=
(\ooCat,\otimes,\Dn{0})$ and an \ndef{anti Gray \oo-category} an anti
$\V$-category, again for $\V = (\ooCat,\otimes,\Dn{0})$.

If $\CG$ is a Gray \oo-category, its objects will also be called
\ndef{$0$-cells}. If $c$ and $d$ are two $0$-cells of $\CG$, the
$k$-cells of the \oo-category $\Homi_\CG(c, d)$ will be called
\ndef{$(k+1)$-cells} of~$\CG$. By definition, if $x$ is such a cell, the
\ndef{$0$-source} $s_0(x)$ of $x$ is $c$ and its \ndef{$0$-target}~$t_0(x)$
is~$d$. The composition $x \comp_i y$ of two $k$-cells of $\Homi_\CG(c,
d)$ will be denoted by $x \comp_{i+1} y$ if the cells $x$ and $y$ are
considered as $(k+1)$-cells of $\CG$.

If $a$, $b$ and $c$ are three $0$-cells of $\CG$, the composition
\oo-functor will be denoted by
\[
  \compG \colon \Homi_\CG(b,c)\otimes \Homi_\CG(a,b) \to \Homi_\CG(a,c)
  \mpbox.
\]
If $x$ is a $k$-cell of $0$-source $b$ and $0$-target $c$ and $y$ is an
$l$-cell of $0$-source $a$ and $0$-target~$b$, we will denote by $x \comp_0
y$ the cell obtained by applying the composition \oo-functor to the $(k + l -
2)$-cell $x \otimes y$ of $\Homi_\CG(b,c)\otimes \Homi_\CG(a,b)$. This cell
is a $(k + l - 2)$-cell of~$\Homi_\CG(a,c)$. In other words, $x \comp_0 y$
is a $(k + l - 1)$-cell of $\CG$. Its $0$-source is $a$ and its $0$-target
$c$. Note that the composition \oo-functor is uniquely determined by all
the~$x \comp_0 y$.

In particular, the composition $\comp_0$ of a $k$-cell and a $1$-cell
of $\CG$ is a $k$-cell. The composition $x \comp_0 y$ of two
$2$-cells of $\CG$ is a $3$\=/cell. In general, the exchange rule in $\CG$
for two such $2$-cells $x$ and $y$ does not hold on the nose but up to the
non-invertible $3$-cell $x \comp_0 y$:
\[
  \xymatrix@C=1.7pc{
  (x \comp_0 t(y)) \comp_1 (s(x) \comp_0 y)
  \ar@3[r]^-{x \comp_0 y}
  &
  (t(x) \comp_0 y) \comp_1 (\beta \comp_0 s(y))
}
\]
(see~\cite[Proposition~B.1.14]{AraMaltsiThmAII}). We refer the reader to
\cite[Section B.1]{AraMaltsiThmAII} for more details on the concrete
description of the structure of a Gray \oo-category.

Similar definitions and notation apply to anti Gray \oo-categories.
\end{paragraph}

\begin{paragraph}\label{paragr:def_trans_Gray}
  Enriched functors between Gray \oo-categories will be called \ndef{Gray
  \oo-functors}. Similarly, enriched functors between anti Gray
  \oo-categories will be called \ndef{anti Gray \oo-functors}. A Gray or
  anti Gray \oo-functor is uniquely determined by its action on cells.
  More precisely, morphisms between Gray or anti Gray \oo-categories can be
  described as functions on $k$-cells for every $k \ge 0$ that are
  compatible with sources, targets, compositions and units.

  Enriched natural transformations between Gray or anti Gray \oo-functors
  will be called \ndef{strict transformations}. Explicitly, if $F, G \colon
  \CG \to \CG'$ are two Gray \oo-functors (or two anti Gray \oo-functors), a
  strict transformation $\alpha \colon F \tod G$ consists of the data of a
  $1$-cell 
  \[ \alpha_c \colon F(c) \to G(c) \]
  for every $0$-cell~$c$ of $\CG$, such that, for every $k$-cell $x$ of
  $\CG$, with $k \ge 1$, we have
  \[ \alpha_{t_0(x)} \comp_0 F(x) = G(x) \comp_0 \alpha_{s_0(x)} \mpbox. \]
\end{paragraph}

\begin{paragraph}
We will denote by
\[
  \ooCatOpLax \quad \text{(resp.~$\ooCatLax$)}
\]
the Gray \oo-category (resp.~the anti Gray \oo-category) whose objects are
\oo-categories and whose homs are
\[
  \HomOpLax(A,B) \quad \text{ (resp.\ } \HomLax(A,B)) \mpbox,
\]
for $A$ and $B$ are two \oo-categories.
By definition, the $0$-cells of
$\ooCatOpLax$ (resp.~of~$\ooCatLax$) are \oo-categories, its $1$-cells are
\oo-functors, its $2$-cells are oplax (resp.~lax) transformations and, for
$n > 2$, its $n$-cells are oplax (resp.~lax) $(n-1)$\=/transformations.
\end{paragraph}

\begin{paragraph}
  Since the monoidal unit of the Gray tensor product is the terminal
  \oo-category, by \ref{paragr:semicart_def}, for all \oo-categories $A$
  and $B$, we have a canonical \oo-functor
  \[
  \pi = (\pi_1, \pi_2) \colon A\otimes B \to A\times B \mpbox.
  \]
\end{paragraph}

\begin{proposition}
  If $A$ and $B$ are two \oo-categories, the \oo-functor
  \[
    \pi \colon A\otimes B \to A\times B
  \]
  is an epimorphism.
\end{proposition}

\begin{proof}
  The \oo-category $A \times B$ is generated under composition by cells of
  the form~$(x, \id{b}^k)$, where $x$ is a $k$-cell of $A$, with $k \ge 0$,
  and $b$ a $0$-cell of $B$, and of the form $(\id{a}^l, y)$, where $a$ is a
  $0$-cell of $A$ and $y$ an $l$-cell of $B$, with $l \ge 0$. It thus
  suffices to show that these cells are in the image of the \oo-functor
  $\pi$. But if $b$ is a $0$-cell of $B$, considering the commutative diagram
  \[
    \begin{tikzcd}
      A \otimes \Dn{0}
      \ar[r, "A \otimes b"]
      \ar[d, "\pi", "\rotatebox{90}{\(\sim\)}"']
      &
      A \otimes B
      \ar[d, "\pi"]
      \\
      A \times \Dn{0}
      \ar[r, "A \times b"']
      &
      A \times B
      \dpbox,
    \end{tikzcd}
  \]
  we get that for every $k$-cell $x$ of $A$, the cell $(x, \id{b}^k)$ is in
  the image of $\pi \colon A \otimes B \to A \times B$. A similar argument
  shows that cells of the form $(\id{a}^l, y)$ are in the image of $\pi$,
  thereby proving the result.
\end{proof}

\begin{remark}
  We will not need this but one can actually prove that the \oo-functor~$\pi$
  of the proposition is surjective on cells. More precisely, if $x$ is a
  $k$-cell of~$A$ and $y$ is an $l$-cell of $B$, then we have $\pi(x \otimes y) =
  (\id{x}^{k+l}, \id{y}^{k+l})$.
\end{remark}

\begin{paragraph}
  In the language of~\ref{paragr:semicart_def}, the previous proposition
  states that the monoidal category $(\ooCat, \otimes, \Dn{0})$ has jointly
  surjective projections. We can thus apply the considerations of
  \ref{paragr:semicart_def} and \ref{paragr:semicart_strength}. Let us start
  by \ref{paragr:semicart_def}.

  We get that the category of \oo-categories embeds both in the category of
  Gray \oo-categories and in the category of anti Gray \oo-categories.
  Moreover, we have canonical monomorphisms
  \[
    \Homi(A,B) \hookto \HomOpLax(A,B)
    \quadand
    \Homi(A,B) \hookto \HomLax(A,B)
    \mpbox,
  \]
  which we will treat as inclusions.

  We thus have canonical Gray and anti
  Gray \oo-functors
  \[
    \ooCatCart \hookto \ooCatOpLax \quadand \ooCatCart \hookto \ooCatLax
    \mpbox,
  \]
  which are the identity on objects and faithful.
\end{paragraph}

\begin{paragraph}\label{paragr:intermed_cyl}
  Let us now apply~\ref{paragr:semicart_strength} to $(\ooCat, \otimes,
  \Dn{0})$. If $D, A, B$ are three \oo-categories, we have a canonical
  monomorphism
  \[
    \lambda \colon \HomLax(D, \Homi(A, B)) \hookto \Homi(A, \HomLax(D, B))
  \]
  that makes the square
  \[
    \begin{tikzcd}
      \HomLax(D, \Homi(A, B))
      \ar[r, "\lambda", hook]
      \ar[d, hook]
      &
      \Homi(A, \HomLax(D, B))
      \ar[d, hook]
      \\
      \HomLax(D, \HomOpLax(A, B))
      \ar[r, "\sim"]
      &
      \HomOpLax(A, \HomLax(D, B))
    \end{tikzcd}
  \]
  commute. We will treat $\lambda$ as an inclusion. We thus get a
  factorization
  \[
      \HomLax(D, \Homi(A, B))
      \hookto
      \Homi(A, \HomLax(D, B))
      \hookto
      \HomLax(D, \HomOpLax(A, B))
  \]
  of the canonical inclusion. This applies in particular to the case where
  $D = \Dn{1}$ in which we get inclusions
  \[
    \Cyl\Homi(A, B) \hookto \Homi(A, \Cyl B) \hookto \Cyl\HomOpLax(A, B)
    \mpbox.
  \]
\end{paragraph}

\begin{remark}
  By adjunction, the $k$-cells of
  \[
    \Cyl\Homi(A, B), \quad \Homi(A, \Cyl B) \quadand \Cyl\HomOpLax(A, B)
  \]
  correspond to \oo-functors from
  \[
   (\Dn{1} \otimes \Dn{k}) \times A,
   \quad
   \Dn{1} \otimes (\Dn{k} \times A)
   \quadand
    \Dn{1} \otimes \Dn{k} \otimes A
   \mpbox,
 \]
 respectively, to $B$. In particular, for $A = \Dn{1}$, the $1$-cells
 correspond to cubes in $B$ of shapes
  \[
   (\Dn{1} \otimes \Dn{1}) \times \Dn{1},
   \quad
   \Dn{1} \otimes (\Dn{1} \times \Dn{1})
   \quadand
    \Dn{1} \otimes \Dn{1} \otimes \Dn{1}
   \mpbox,
 \]
 respectively. This means that the $1$-cells of $\Cyl\HomOpLax(\Dn{1}, B)$
 are fully lax cubes in $B$, those of $\Homi(\Dn{1}, \Cyl B)$ are
 commutative cubes with only four lax faces (the two commutative faces being
 opposite to each other) and those of $\Cyl\Homi(\Dn{1}, B)$ are
 commutative cubes with only two (opposite) lax faces.
\end{remark}

We end the section by some considerations on the dualities of $\ooCat$.

\begin{paragraph}\label{paragr:def_dual}
  If $S \subset \N^\ast$ is a subset of the set of positive integers, then we
  will denote by
  \[ D_S \colon \ooCat \to \ooCat \]
  the \oo-functor sending an \oo-category $C$ to the \oo-category obtained
  from $C$ by reversing the orientation of all the cells whose dimension
  belongs to $S$. It is immediate that $D_S$ is an involutive endofunctor of
  $\ooCat$. Actually, up to isomorphism, all the autoequivalences of
  $\ooCat$ are of the form $D_S$. We will sometimes refer to
  these autoequivalences as \ndef{dualities}.

  Several special cases play an important role in the theory of $\ooCat$:
  \begin{itemize}
    \item If $S = \N^\ast$, then $D_{\N^\ast}$ is denoted by $D_\o$ and is
      called the \ndef{total dual}. We simply write $C^\o$ for the total
      dual of an \oo-category $C$.
    \item If $S = 2\N+1$ is the set of odd integers, then $D_{2\N+1}$ is
      denoted by $D_\op$ and is called the \ndef{odd dual}. We simply write
      $C^\op$ for the odd dual of an \oo-category $C$.
    \item If $S = 2\N^\ast$ is the set of positive even integers, then
      $D_{2\N^\ast}$ is denoted by $D_\co$ and is called the \ndef{even
      dual}. We simply write $C^\co$ for the odd dual of an \oo-category~$C$.
    \item If $S = \{1\}$, then $D_{\{1\}}$ is denoted by $D_\dt$ and is called the
      \ndef{transpose}. We simply write $C^\dt$ for the transpose of $C$.
      This coincides with the transpose of $C$ in the sense of
      \ref{paragr:transpose} when $C$ is considered as a category
      enriched over $\ooCat$ endowed with the cartesian product.
  \end{itemize}
  By composing all these special dualities, we get a group of eight dualities. In
  particular, if $C$ is an \oo-category, we get seven other \oo-categories
  \[
    C^\o,
    \quad
    C^\op,
    \quad
    C^\co,
    \quad
    C^\dt,
    \quad
    C^\dto = {(C^\o)}^\dt,
    \quad
    C^\dtop = {(C^\op)}^\dt,
    \quad
    C^\dcot = {(C^\co)}^\dt
    \mpbox.
  \]
  Note that this group of dualities is isomorphic to ${(\Z/2\Z)}^3$, a
  natural basis (as a module over $\Z/2\Z$) being given by $D_\op$, $D_\co$
  and $D_\dt$.
\end{paragraph}

We now recall the compatibilities of the dualities of $\ooCat$ with the Gray
tensor product.

\begin{proposition}
  Let $A$ and $B$ be two \oo-categories. There are canonical isomorphisms
  \[
    (A \otimes B)^\op \simeq B^\op \otimes A^\op,
    \quad
    (A \otimes B)^\co \simeq B^\co \otimes A^\co,
    \quad
    (A \otimes B)^\o \simeq A^\o \otimes B^\o \mpbox,
  \]
  natural in $A$ and $B$. In other words, the functors
  \[ D_\co, D_\op \colon \ooCat \to \ooCat \]
  are anti-monoidal and the functor
  \[ D_\o \colon \ooCat \to \ooCat \]
  is monoidal, $\ooCat$ being endowed with the Gray tensor product.
\end{proposition}

\begin{proof}
  See for instance \cite[Proposition A.22]{AraMaltsiJoint}.
\end{proof}

\begin{remark}
  Besides the trivial duality, $D_\op$, $D_\co$ and $D_\o$ are the
  only dualities of $\ooCat$ that are either monoidal or anti-monoidal
  (see the proof of \cite[Proposition~A.20]{AraMaltsiJoint}).
\end{remark}

\begin{remark}\label{rem:dual_hom}
  It follows from the previous proposition that if $A$ and $B$ are two
  \oo-categories, then we have natural isomorphisms
  \[
    \begin{split}
    \HomOpLax(A, B)^\op
    & \simeq
    \HomLax(A^\op, B^\op)
    \mpbox,
    \\
    \HomOpLax(A, B)^\co
    & \simeq
    \HomLax(A^\co, B^\co)
    \mpbox,
    \\
    \HomOpLax(A, B)^\o
    & \simeq
    \HomOpLax(A^\o, B^\o)
    \mpbox.
    \end{split}
  \]
\end{remark}

\begin{paragraph}\label{paragr:Gray_dual}
  Let $\CG$ be a Gray \oo-category. Using the previous proposition, from $\CG$,
  we can get two anti Gray \oo-categories $(D_\op)_\ast(\CG)$ and
  $(D_\co)_\ast(\CG)$, and one Gray \oo-category~$(D_\o)_\ast(\CG)$, obtained
  by applying these dualities hom-wise
  (see~\ref{paragr:monoidal_push-forward} for the notation). Note that in
  these new Gray or anti Gray \oo-categories, the $1$-cells are never
  reversed and in some sense there is a shift of the reversed dimensions by
  $1$. With this in mind, we set, with the notation
  of~\ref{paragr:transpose},
  \[
    \CG^\op = \big((D_\co)_\ast(\CG)\big)^\trans,
    \quad
    \CG^\co = (D_\op)_\ast(\CG),
    \quad
    \CG^\o = \big((D_\o)_\ast(\CG)\big)^\trans \mpbox.
  \]
  Then $\CG^\op$ is a Gray \oo-category, and $\CG^\co$ and $\CG^\o$ are anti
  Gray \oo-categories. We also set
  \[
    \CG^\dtop = {(\CG^\op)}^\trans,
    \quad
    \CG^\dcot = {(\CG^\co)}^\trans,
    \quad
    \CG^\dto = {(\CG^\o)}^\trans
    \mpbox.
  \]
  To sum up, from a Gray \oo-category $\CG$, we get three other Gray
  \oo-categories
  \[ \CG^\op,\quad \CG^\dcot,\quad \CG^\dto \]
  and four anti Gray \oo-categories
  \[ \CG^\dt,\quad \CG^\dtop,\quad \CG^\co,\quad \CG^\o \mpbox. \]
  No other duality of $\ooCat$ produces a Gray or an anti Gray
  \oo-category.
\end{paragraph}

\begin{paragraph}\label{paragr:duality_Gray}
  Using the previous paragraph, one can interpret the dualities $D_\dco$,
  $D_\op$ and~$D_\o$ of \oo-categories as Gray \oo-functors.
  More precisely, one can check using \ref{rem:dual_hom} that these
  dualities induce isomorphisms of Gray \oo-categories
  \[
    \begin{split}
      D_\op & \colon (\ooCatLax)^\dco \to \ooCatOpLax
      \mpbox,
      \\
      D_\dco & \colon (\ooCatLax)^\dtop \to \ooCatOpLax
      \mpbox,
      \\
      D_\o & \colon (\ooCatOpLax)^\dto \to \ooCatOpLax
      \mpbox.
    \end{split}
  \]
  Dually, we have isomorphisms of anti Gray \oo-categories
  \[
    \begin{split}
      D'_\op & \colon (\ooCatOpLax)^\dco \to \ooCatLax
      \mpbox,
      \\
      D'_\dco & \colon (\ooCatOpLax)^\dtop \to \ooCatLax
      \mpbox,
      \\
      D'_\o & \colon (\ooCatLax)^\dto \to \ooCatLax
      \mpbox.
    \end{split}
  \]
\end{paragraph}

\section{Comma \pdfoo-categories}

\begin{paragraph}
  Consider a diagram
  \[
    \begin{tikzcd}
      A \ar[r, "f"] & C & \ar[l, "g"'] B
    \end{tikzcd}
  \]
  in $\ooCat$. The \ndef{comma \oo-category} $A \comma_C B$, also
  denoted by $f \comma g$, is the universal \oo-category endowed with a
  $2$-square
  \[
    \xymatrix@C=1pc@R=1.5pc{
      & A \comma_C B \ar[dl]_{p_1} \ar[dr]^{p_2} \\
      A \ar[dr]_f \ar@{}[rr]_(.35){}="x"_(.65){}="y"
      \ar@2"x";"y"^{\gamma} 
      & & B \ar[dl]^g \\
      & C
    }
  \]
  in $\ooCatOpLax$, that is, where $p_1$ and $p_2$ are \oo-functors and
  $\gamma$ is an \emph{oplax} transformation. This means that if
  $T$ is an \oo-category endowed with a similar $2$-square
  \[
    \xymatrix@C=1.5pc@R=1.5pc{
      & T \ar[dl]_{a} \ar[dr]^{b} \\
      A \ar[dr]_f \ar@{}[rr]_(.35){}="x"_(.65){}="y"
      \ar@2"x";"y"^{\lambda} 
      & & B \ar[dl]^g \\
      & C & \dpbox{,}
    }
  \]
  then there exists a unique \oo-functor $h \colon T \to A \comma_C B$ that
  factors the diagram in the sense that
  \[
  p_1h = a,
  \quad
  p_2h = b \quadand \gamma \comp_0 h = \lambda
  \mpbox.
  \]
  By \ref{paragr:def_cyl}, and with its notation, the data of such a
  $2$-square is equivalent to the data of an \oo-functor from $T$ to the
  limit of the diagram
    \[
      \begin{tikzcd}[column sep=1pc]
        A \ar[dr, "f"'] & & \Cyl C \ar[dl, "{\ss\phantom{\tt}}"] \ar[dr, "\tt"']
                        & & B \ar[dl,"g"] \\
                        & C & & C & \dpbox.
      \end{tikzcd}
    \]
  This shows that we have
  \[ A \comma_C B = A \times_C \Cyl C \times_C B \mpbox{.} \]
  The canonical projections $p_1$ and $p_2$ are the obvious projections on
  $A$ and $B$, and the $2$\nbd-cell $\gamma\colon fp_1 \tod gp_2$
  corresponds to the projection
  \[
  \gamma \colon A \comma_C B \to \Cyl C
  \mpbox.
  \]
\end{paragraph}

\begin{example}\label{ex:slices_as_commas}
  The slice \oo-categories are particular cases of comma \oo-categories.
  Indeed, if $C$ is an \oo-category and $c$ is an object of $C$, then the
  comma construction~$c \comma C$ of the diagram
  \[
    \begin{tikzcd}
      \Dn{0} \ar[r, "c"] & C & C \ar[l, "{\id{C}}"']
    \end{tikzcd}
  \]
  is canonically isomorphic to the slice \oo-category $\cotr{C}{c}$
  described in \cite[Chapter 9]{AraMaltsiJoint} (see~\cite[Proposition
  7.1]{AraMaltsiThmAII} for a proof).
  More generally, if $v \colon B \to C$ is an \oo-functor, then we have
  \[ \cotr{B}{c} = c \comma v \mpbox, \]
  where $\cotr{B}{c}$ is the relative slice
  defined by the pullback
  \[
    \begin{tikzcd}
    \cotr{B}{c}
     \ar[d] \ar[r]
     &
     \cotr{C}{c}
     \ar[d, "U"]
     \\
     B
     \ar[r, "v"']
     &
     C
     \dpbox{,}
     \ar[from=1-1,to=2-2,phantom,"\lrcorner" description, very near start]
    \end{tikzcd}
  \]
  with $U$ denoting the forgetful \oo-functor.

  Similarly, we have
  \[ \tr{C}{c} = C \comma c \]
  and, more generally, if $u \colon A \to C$ is an \oo-functor,
  \[ \tr{A}{c} = u \comma c \mpbox. \]
\end{example}

\begin{paragraph}
  The comma construction $A\comma_C B$ is functorial in $A$ and~$B$. Indeed,
  if
  \[
    \xymatrix@R=1pc@C=3pc{
      A \ar[dd]_u \ar[dr]^f_{}="f" & & B \ar[dl]_g_{}="g" \ar[dd]^v \\
                                   & C \\
      A' \ar[ur]_{f'} & & B' \ar[ul]^{g'}
      \ar@{}[ll];"f"_(0.35){}="sa"_(0.85){}="ta"
      \ar@2"sa";"ta"^{\alpha}
      \ar@{}[];"g"_(0.35){}="tb"_(0.85){}="sb"
      \ar@2"sb";"tb"^{\beta}
    }
  \]
  is a diagram in $\ooCatOpLax$, then we get an \oo-functor
      \[
      (u, \alpha) \comma (\beta, v) \colon A \comma_C B \to A' \comma_C B'
      \]
  by applying the universal property of $A' \comma_C B'$ to the $2$-square
  obtained by composing the diagram
      \[
        \xymatrix@R=1pc@C=2pc{
          & A \comma_C B \ar[dl]_{p_1} \ar[dr]^{p_2} \\
          A \ar[dd]_u \ar[dr]^f_{}="f" & & B \ar[dl]_g_{}="g" \ar[dd]^v 
      \ar@{}[ll];[]_(0.40){}="x"_(0.60){}="y"
      \ar@2"x";"y"^{\gamma}
          \\
            & C \\
          A' \ar[ur]_{f'} & & B' \ar[ul]^{g'}
          \ar@{}[ll];"f"_(0.35){}="sa"_(0.85){}="ta"
          \ar@2"sa";"ta"^{\alpha}
          \ar@{}[];"g"_(0.35){}="tb"_(0.85){}="sb"
          \ar@2"sb";"tb"^{\beta}
          \dpbox{.}
        }
        \]

  Therefore, the comma construction defines a functor from
  the obvious category whose objects are the diagrams
  \[
    \begin{tikzcd}
      A \ar[r, "f"] & C & \ar[l, "g"'] B
    \end{tikzcd}
  \]
  and whose morphisms are the diagrams
  \[
    \xymatrix@R=1pc@C=3pc{
      A \ar[dd]_u \ar[dr]^f_{}="f" & & B \ar[dl]_g_{}="g" \ar[dd]^v \\
                                   & C \\
      A' \ar[ur]_{f'} & & B' \ar[ul]^{g'}
      \ar@{}[ll];"f"_(0.35){}="sa"_(0.85){}="ta"
      \ar@2"sa";"ta"^{\alpha}
      \ar@{}[];"g"_(0.35){}="tb"_(0.85){}="sb"
      \ar@2"sb";"tb"^{\beta}
    }
  \]
  in $\ooCatOpLax$ to the category $\ooCat$.
\end{paragraph}

In \cite{AraMaltsiThmAII}, the first-named author and Maltsiniotis proved
that the comma construction can be promoted to a sesquifunctor. The main
goal of the present paper is to express and prove the full functorialities
of the comma construction, with respect to the higher structure of the Gray
\oo-category $\ooCatOpLax$. To do so, one ingredient will be the
\oo-categorical universal property of the comma construction that we will
now describe.

\begin{paragraph}
  Let
  \[
    \begin{tikzcd}
      A \ar[r, "f"] & C & \ar[l, "g"'] B
    \end{tikzcd}
  \]
  be as before and consider the universal $2$-square
  \[
    \xymatrix@C=1pc@R=1.5pc{
      & A \comma_C B \ar[dl]_{p_1} \ar[dr]^{p_2} \\
      A \ar[dr]_f \ar@{}[rr]_(.35){}="x"_(.65){}="y"
      \ar@2"x";"y"^{\gamma} 
      & & B \ar[dl]^g \\
      & C & \dpbox.
    }
  \]
    For every \oo-category $T$, we have a canonical \oo-functor
    \[
    \begin{tikzcd}
    \displaystyle\HomOpLax(T,A\comma_C B)
    \ar[d]\\
    \HomOpLax(T,A)\times_{\HomOpLax(T,C)}\HomOpLax(T,\Cyl C)\times_{\HomOpLax(T,C)}\HomOpLax(T,B)
    \dpbox,
    \end{tikzcd}
    \]
    induced by the projections $p_1$, $\gamma$ and $p_2$,
    which we can identify with an \oo-functor
    \[
    \HomOpLax(T,A\comma_CB) \longrightarrow \HomOpLax(T,A)\commabig_{\HomOpLax(T,C)}\HomOpLax(T,B)
    \]
    using the canonical isomorphism $\HomOpLax(T,\Cyl C)\simeq \Cyl
    \HomOpLax(T,C)$ of \ref{paragr:def_cyl}. Since the functor
    $\HomOpLax(T,-) \colon \ooCat \to \ooCat$ is a right adjoint, it
    commutes with fiber products and we get the following result:
\end{paragraph}

{%
\let\thmnewline\relax
  \begin{proposition}[Higher universal property of the comma
    construction]\label{prop:univ}
    If $\begin{tikzcd}[column sep=1pc] A \ar[r, "f"] & C & \ar[l,"g"'] B
    \end{tikzcd}$ is a diagram in $\ooCat$, then for any \oo-category $T$
    the canonical morphism
    \[
      \smash{
        \HomOpLax(T,A \comma_C B) \overset{\sim}{\longrightarrow}
        \HomOpLax(T,A)\commabig_{\HomOpLax(T,C)} \HomOpLax(T,B)
      }
    \]
    is an isomorphism.
  \end{proposition}
}%

The enriched description of the \oo-category of cylinders
(see~\ref{paragr:enriched_cyl}) leads to an analogous description for the
comma \oo-category:

\begin{paragraph}\label{paragr:enriched_comma}
  If $\begin{tikzcd}[column sep=1pc] A \ar[r, "f"] & C & \ar[l,"g"'] B
  \end{tikzcd}$ is a diagram in $\ooCat$, then the \oo-category $A \comma_C
  B$ can be described (up to isomorphism) as a category enriched in
  \oo-categories in the following way:
  \begin{itemize}
    \item The objects of $A \comma_C B$ are triples $(a, l \colon fa \to
      gb, b)$, where $a$ is an object of~$A$, $b$ an object of $B$ and $l$ a
      $1$-cell of $C$.
    \item If $(a, l \colon fa \to gb, b)$ and $(a', l' \colon fa' \to gb',
      b')$ are two objects of $A \comma_C B$, then
      \[ 
        \begin{split}
          \MoveEqLeft
        \mskip10mu
        \Homi_{A \comma_C B}((a, l, b), (a', l', b'))
        \\
        & =
        \Homi_A(a, a')
        \times_{\Homi_C(fa, gb')} \Cyl{\Homi_C(fa, gb')}
        \times_{\Homi_C(fa, gb')}
        \Homi_B(b, b')
        \mpbox,
        \end{split}
      \]
      where this iterated fiber product denotes the limit of the diagram
      \[
        \begin{tikzcd}[column sep=-0.5pc]
          \Homi_A(a, a') \ar[dr, "{l' \comp_0 f(\var)}"']
          & &
          \Cyl{\Homi_C(fa, gb')}
          \ar[dl, "{\ss\phantom{\tt}}"] \ar[dr, "\tt"']
          & &
          \Homi_B(b, b') \ar[dl, "{g(\var) \comp_0 l}"]
          \\
          & \Homi_C(fa, gb') & & \Homi_C(fa, gb') &
          & \mpbox.
        \end{tikzcd}
      \]
      This \oo-category is actually itself a comma \oo-category, namely
      \[ \Homi_A(a, a') \comma_{\Homi_C(fa, gb')} \Homi_B(b, b') \mpbox. \]
      Concretely, a $k$-cell in this hom is a triple
      $(u, \alpha, v)$ in
      \[
        \Homi_A(a, a')_k \times \Cyl{\Homi_C(fa, gb')}_k
        \times \Homi_B(b, b')_k
      \]
      such that
      \[
        \ss(\alpha) = l' \comp_0 f(u) \quadand \tt(\alpha) = g(v) \comp_0 l
        \mpbox.
      \]
    \item If $(a, l, b), (a', l', b'), (a'', l'', b'')$ are three objects of
      $A \comma_C B$, the composition \oo-functor
      \[
        \begin{tikzcd}
        \Homi_{A \comma_C B}((a', l', b'), (a'', l'', b''))
        \times
        \Homi_{A \comma_C B}((a, l, b), (a', l', b'))
        \ar[d]
        \\
        \Homi_{A \comma_C B}((a, l, b), (a'', l'', b''))
        \end{tikzcd}
      \]
      is given by
      \[
        ((u', \alpha', v'), (u, \alpha, v))
        \mapsto
        (u' \comp_0 u,
        (g(v') \compl \alpha) \compc (\alpha' \compr f(u)),
        v' \comp_0 v)
        \mpbox.
      \]
    \item If $(a, l, b)$ is an object of $A \comma_C B$, its unit is the
      triple $(\id{a}, \idd{l}, \id{b})$.
  \end{itemize}
\end{paragraph}

\begin{paragraph}\label{paragr:dual_comma}
  The comma construction we studied in this section is the \emph{oplax}
  comma construction. Similarly, one can define a \ndef{lax comma
  construction} by replacing the oplax transformation in the $2$-square of
  the universal property of the oplax comma construction by a lax
  transformation. If
  \[
    \begin{tikzcd}
      A \ar[r, "f"] & C & \ar[l,"g"'] B
    \end{tikzcd}
  \]
  is a diagram in $\ooCat$, we will denote by $A \commalax_C B$,
  or $f \commalax g$, the lax comma construction of $f$ and $g$. Explicitly,
  we have
  \[
    A \commalax_C B = A \times_C \Cylp C \times_C B \mpbox{,}
  \]
  where $\Cylp C = \HomOpLax(\Dn{1}, C)$.

  The lax comma construction can also be defined by duality from the oplax
  version. Indeed, there are natural isomorphisms
  \[
    (A \comma_C B)^\op \simeq B^\op \commalax_{C^\op} A^\op
    \quadand
    (A \comma_C B)^\co \simeq A^\co \commalax_{C^\co} B^\co
    \mpbox.
  \]
  In particular, we have
  \[ 
    (A \comma_C B)^\o \simeq B^\o \comma_{C^\o} A^\o
    \quadand
    (A \commalax_C B)^\o \simeq B^\o \commalax_{C^\o} A^\o
    \mpbox.
  \]

  In this text, we will mainly deal with the oplax version of the comma
  construction and therefore drop the adjective ``oplax''.
\end{paragraph}

\section{Slices of Gray \pdfoo-categories}

The purpose of this section is to define, for $\CG$ a Gray \oo-category and
$c$ an object of $\CG$, a slice Gray \oo-category $\tr{\CG}{c}$.

\medbreak

The description of the comma construction of~\ref{paragr:enriched_comma}
gives in particular an enriched description for the slices of \oo-categories.
We will see that this description can be adapted to Gray \oo-categories.
The definition will involve the \oo-category
\[ \Cyl(\HomCG(d, c)) = \HomLax(\Dn{1}, \HomCG(d, c))\mpbox, \]
where $d$ is another object of $\CG$, and we start the section by an analysis
of the structure of this \oo-category.

\begin{paragraph}\label{paragr:slice_module}
  Let $\CG$ be a Gray \oo-category. For every object $c$ of $\CG$,
  by composing the two anti Gray \oo-functors
  \[
    \begin{tikzcd}[column sep=3.5pc]
      \CG^\trans \ar[r, "{\HomCG(\var, c)}"]
      &
      \ooCatLaxGray \ar[r, "{\Cyl{}}"]
      &
      \ooCatLaxGray
      \mpbox,
    \end{tikzcd}
  \]
  where
  \[
      \Cyl{} = \HomLax(\Dn{1}, \var) \mpbox,
  \]
  we get an anti Gray \oo-functor
  \[
    \Cyl(\HomCG(\var, c)) \colon \CG^\trans \to \ooCatLaxGray
    \mpbox.
  \]
  By Proposition~\ref{prop:modules_as_functors}, this means that
  $\Cyl(\HomCG(\var, c))$ is a right $\CG$-module in the sense
  of~\ref{paragr:def_module}.
\end{paragraph}

\begin{paragraph}\label{paragr:compr}
  Let $\CG$ be a Gray \oo-category and let $a$, $b$ and $c$ be three objects
  of $\CG$. The structure of right $\CG$-module of the previous paragraph
  defines an \oo-functor
  \[
    \compr \colon
    \Cyl{\HomCG(b, c)} \otimes \HomCG(a, b)
    \to
    \Cyl{\HomCG(a, c)}
    \mpbox.
  \]
  Moreover, the axioms of modules give that, if $a$, $b$, $c$ and $d$ are
  four objects of~$\CG$, then the diagrams
  \[
    \xymatrix@C=5pc{
      \Cyl{\HomCG(c, d)} \otimes \HomCG(b, c) \otimes \HomCG(a, b)
      \ar[r]^-{\Cyl{\HomCG(c, d)} \otimes \compG}
      \ar[d]_{\compr \otimes \HomCG(a,b)}
      &
      \Cyl{\HomCG(c, d)} \otimes \HomCG(a, c)
      \ar[d]^{\compr}
      \\
      \Cyl{\HomCG(b, d)} \otimes \HomCG(a, b)
      \ar[r]_-{\compr}
      &
      \Cyl{\HomCG(a, d)}
    }
  \]
  and
  \[
    \xymatrix@C=6pc{
      \Cyl\HomCG(a, b)
      \ar[r]^-{\Cyl\HomCG(a, b) \otimes \id{a}}
      \ar[rd]_{=}
      &
      \Cyl\HomCG(a, b) \otimes \HomCG(a, a)
      \ar[d]^{\compr}
      \\
      &
      \Cyl\HomCG(a, b)
    }
  \]
  commute.
\end{paragraph}

\begin{remark}
  In the case where the Gray \oo-category comes from a strict \oo-category,
  we have already defined an operation $\compr$ in
  \ref{paragr:compr_str}. Proposition~\ref{prop:comp_compr} will show
  that it is compatible with the operation $\compr$ introduced in the
  previous paragraph.
\end{remark}

\begin{paragraph}\label{paragr:compc_natural}
  We saw in~\ref{paragr:compc} that, if $C$ is an \oo-category, then we have
  \oo-functors
  \[
    \ss, \tt \colon \Cyl{C} \to C,
    \quad
    \kk \colon C \to \Cyl{C}
    \quadand
    \compc \colon \Cyl{C} \times_C \Cyl{C} \to \Cyl{C}
  \]
  that define a structure of category internal to \oo-categories.
  All the operations of this structure are natural in $C$. This means that
  this structure of internal category to $\ooCat$ extends to a structure of
  internal category to the category of anti Gray \oo-functors from
  $\ooCatLaxGray$ to itself and strict transformations between them. More
  precisely, the anti Gray \oo-functor~$\Cyl \colon \ooCatLaxGray \to
  \ooCatLaxGray$ is the object of morphisms of a category internal to the
  category of anti Gray \oo-functors from $\ooCatLaxGray$ to itself, the
  object of objects being the identity anti Gray \oo-functor.
\end{paragraph}

\begin{paragraph}\label{paragr:cyl_map_modules}
  Let $\CG$ be a Gray \oo-category and let $c$ be an object of $\CG$.
  By precomposing the internal category of the previous paragraph
  with the anti Gray \oo-functor
  \[
    \HomCG(\var, c) \colon \CG^\trans \to \ooCatLaxGray \mpbox,
  \]
  we get that
  \[
    \Cyl(\HomCG(\var, c)) \colon \CG^\trans \to \ooCatLaxGray
  \]
  is the object of morphisms of a category internal to the category of anti
  Gray \oo-functors from $\CG^\trans$ to $\ooCatLaxGray$, with object
  of objects $\HomCG(\var, c)$ and structure maps
  \[
    \ss, \tt \colon \Cyl{\HomCG(\var,c)} \tod \HomCG(\var,c) \mpbox,
    \quad
    \kk \colon \HomCG(\var,c) \tod \Cyl{\HomCG(\var,c)}
  \]
  and
  \[
    \compc \colon \Cyl{\HomCG(\var,c)} \times_{\HomCG(\var,c)}
    \Cyl{\HomCG(\var,c)} \tod \Cyl{\HomCG(\var,c)}
    \mpbox.
  \]

  By Proposition~\ref{prop:modules_as_functors}, this means that
  $\Cyl(\HomCG(\var, c))$ is the object of morphisms of a category internal
  to the category of right $\CG$-modules and the four above structure maps
  correspond to morphisms of right $\CG$-modules.
\end{paragraph}

\begin{paragraph}\label{paragr:cyl_square_map_modules}
  Let $\CG$ be a Gray \oo-category and let $a$, $b$ and $c$ be three objects
  of $\CG$. The fact that the structure maps $\ss$, $\tt$, $\kk$ and
  $\compc$ of the previous paragraph correspond to morphisms of right
  $\CG$-modules precisely means that the squares
  \[
    \xymatrix{
    \Cyl{\HomCG(b, c)} \otimes \HomCG(a, b)
    \ar[d]_{\ee \otimes \HomCG(a, b)}
    \ar[r]^-\compr
    &
    \Cyl{\HomCG(a, c)}
    \ar[d]^{\ee}
    \\
    \HomCG(b, c) \otimes \HomCG(a, b)
    \ar[r]_-{\compG}
    &
    \HomCG(a, c)
    \dpbox,
    }
  \]
  for $\ee$ being $\ss$ or $\tt$,
  \[
    \xymatrix{
    \Cyl{\HomCG(b, c)} \otimes \HomCG(a, b)
    \ar[r]^-\compr
    &
    \Cyl{\HomCG(a, c)}
    \\
    \HomCG(b, c) \otimes \HomCG(a, b)
    \ar[u]^{\kk \otimes \HomCG(a, b)}
    \ar[r]_-{\compG}
    &
    \HomCG(a, c)
    \ar[u]_{\kk}
    }
  \]
  and
  \[
    \small
    \xymatrix{
      \big(
      \Cyl{\HomCG(b,c)} \times_{\HomCG(b,c)} \Cyl{\HomCG(b,c)}
      \big)
      \otimes
      \HomCG(a, b)
      \ar[r]^-{\compc \otimes \id{}}
      \ar[d]
      &
      \Cyl{\HomCG(b,c)}
      \otimes
      \HomCG(a, b)
      \ar[d]^{\compr}
      \\
      \Cyl{\HomCG(a,c)} \times_{\HomCG(a,c)} \Cyl{\HomCG(a,c)}
      \ar[r]_-{\compc}
      &
      \Cyl{\HomCG(a,c)}
      \zbox{\qquad\qquad,}
    }
  \]
  where the left vertical arrow is the composite
  \[
    \small
    \xymatrix{
      \big(
      \Cyl{\HomCG(b,c)} \times_{\HomCG(b,c)} \Cyl{\HomCG(b,c)}
      \big)
      \otimes
      \HomCG(a, b)
      \ar[d]_{\can}
      \\
      \big(
      \Cyl{\HomCG(b,c)} \otimes \HomCG(a, b)
      \big)
      \times_{\HomCG(b,c) \otimes \HomCG(a, b)}
      \big(
      \Cyl{\HomCG(b,c)} \otimes \HomCG(a, b)
      \big)
      \ar[d]_{\compr \times_{\comp_0} \compr}
      \\
      \Cyl{\HomCG(a,c)} \times_{\HomCG(a,c)} \Cyl{\HomCG(a,c)}
      \dpbox,
    }
  \]
  commute.
\end{paragraph}

\begin{remark}\label{rem:compl}
  Note that if $a$, $b$ and $c$ are three objects of a Gray \oo-category
  $\CG$, there is no natural \oo-functor
  \[
    \HomCG(b, c)
    \otimes
    \Cyl{\HomCG(a, b)}
    \to
    \Cyl{\HomCG(a, c)}
    \mpbox,
  \]
  and in particular, there is no natural structure of left $\CG$-module on
  $\Cyl{\HomCG(a, \var)}$. What is true is that
  $\Cylp{\HomCG(a, \var)}$, where $\Cylp(C) = \HomOpLax(\Dn{1},
  C)$, is naturally a left $\CG$-module.
\end{remark}

We can now define slice Gray \oo-categories.

\begin{paragraph}\label{paragr:def_Gray_slice}
  Let $\CG$ be a Gray \oo-category and let $c$ be an object of $\CG$. We define
  the \ndef{slice Gray \oo-category} $\tr{\CG}{c}$ in the following way:
  \begin{itemize}[wide]
    \item The objects of $\tr{\CG}{c}$ are pairs $(d, f \colon d \to c)$, where
      $d$ is an object of $\CG$ and $f$ a $1$-cell.
    \item If $(d, f \colon d \to c)$ and $(d', f' \colon d' \to c)$ are two objects of
      $\CG$, we set
      \[
        \begin{split}
          \Homi_{\tr{\CG}{c}}((d, f), (d', f')) 
          & = \HomCG(d, d') \Cylpb{\HomCG(d, c)} \{f\}
          \\
          & = \HomCG(d, d') \comma_{\HomCG(d, c)} \{f\}
          \mpbox.
        \end{split}
      \]
      By definition, a $k$-cell in this hom consists of a pair
      $(u, \alpha)$, with $u$ a $k$-cell of~$\HomCG(d, d')$ and $\alpha$
      a $k$-cell of $\Cyl{\HomCG(d, c)}$ such that
      \[
        \ss(\alpha) = f' \compG u
        \quadand
        \tt(\alpha) = \id{f}^k
        \mpbox.
      \]
      In particular, an object of this hom corresponds to a $2$-triangle
      \[
        \xymatrix@C=1.5pc{
          d \ar[rr] \ar[dr]_{f}_(.6){}="f" & & d' \ar[dl]^(0.48){f'} \\
          & c
          \ar@{}"f";[ur]_(.15){}="ff"
          \ar@{}"f";[ur]_(.55){}="oo"
          \ar@<-0.5ex>@2"oo";"ff"
          & \dpbox.
        }
      \]
      We will denote by $U$ and $\gamma$ the projections
      \[
        \begin{split}
        U \colon  \Homi_{\tr{\CG}{c}}((d, f), (d', f')) \to \Homi_{\CG}(d, d') 
        \\
        \gamma \colon \Homi_{\tr{\CG}{c}}((d, f), (d', f')) \to \Cyl{\HomCG(d,
        c)} \mpbox,
        \end{split}
      \]
      so that
      \[
        U(u, \alpha) = u \quadand \gamma(u, \alpha) = \alpha
        \mpbox.
      \]

    \item If $(d, f \colon d \to c)$ is an object of $\tr{\CG}{c}$,
      the associated unit
      \[ \Dn{0} \to \Homi_{\tr{\CG}{c}}((d, f), (d, f)) \]
      is given by the pair
      \[ 
        \begin{tikzcd}[column sep=1.5pc]
          \Dn{0} \ar[r, "{\id{d}}"] & \HomCG(d, d) \mpbox,
        \end{tikzcd}
        \quad
        \begin{tikzcd}[column sep=1.5pc]
          \Dn{0} \ar[r, "f"] & \HomCG(d, c) \ar[r, "{\kk}"] &
          \Cyl(\HomCG(d, c)) \mpbox.
        \end{tikzcd}
      \]
      Concretely, it corresponds to the $2$-triangle
      \[
        \xymatrix@C=1.5pc{
          d \ar[rr]^{\id{d}} \ar[dr]_{f}_(.6){}="f" & & d
          \ar[dl]^{f} \\
          & c
          \ar@{}"f";[ur]_(.15){}="ff"
          \ar@{}"f";[ur]_(.55){}="oo"
          \ar@<-0.5ex>@2"oo";"ff"_{\id{f}}
          & \dpbox.
        }
      \]
      In symbols, we have
      \[
        \id{(d, f)} = (\id{d}, \idd{f})
      \]
      (remember that we denote $\kk(f)$ by $\idd{f}$).
    \item Let $(d, f)$, $(d', f')$ and $(d'', f'')$ be three objects of
      $\tr{\CG}{c}$. We now define the composition \oo-functor
      \[
          \Homi_{\tr{\CG}{c}}((d', f'), (d'', f'')) 
          \otimes
          \Homi_{\tr{\CG}{c}}((d, f), (d', f')) 
          \to
          \Homi_{\tr{\CG}{c}}((d, f), (d'', f''))
          \mpbox.
      \]
      To define such an \oo-functor we need to define two \oo-functors
      \[
        \begin{split}
          &
          \Homi_{\tr{\CG}{c}}((d', f'), (d'', f'')) 
          \otimes
          \Homi_{\tr{\CG}{c}}((d, f), (d', f')) 
          \to
          \HomCG(d, d'')
          \\
          &
          \Homi_{\tr{\CG}{c}}((d', f'), (d'', f'')) 
          \otimes
          \Homi_{\tr{\CG}{c}}((d, f), (d', f')) 
          \to
          \Cyl(\HomCG(d, c))
        \end{split}
      \]
      compatible with the pullback defining $\Homi_{\tr{\CG}{c}}((d, f),
      (d'', f''))$.
      The first one is defined by composing
      \[
        \xymatrix{
          \Homi_{\tr{\CG}{c}}((d', f'), (d'', f'')) 
          \otimes
          \Homi_{\tr{\CG}{c}}((d, f), (d', f'))
          \ar[d]_{U \otimes U}
          \\
          \HomCG(d', d'') 
          \otimes
          \HomCG(d, d') 
          \ar[r]^-{\compG}
          &
          \HomCG(d, d'')
          \dpbox.
        }
      \]
      The second one is defined by composing
      \[
        \xymatrix{
          \Homi_{\tr{\CG}{c}}((d', f'), (d'', f'')) 
          \otimes
          \Homi_{\tr{\CG}{c}}((d, f), (d', f'))
          \ar[d]
          \\
          \Cyl(\HomCG(d, c)) 
          \times_{\HomCG(d, c)}
          \Cyl(\HomCG(d, c))
          \ar[r]^-{\compc}
          &
          \Cyl(\HomCG(d, c))
          \dpbox,
        }
      \]
      where the vertical morphism is induced by the following
      hexagon
      \[
        \xymatrix@C=-4pc{
          & 
          \Homi_{\tr{\CG}{c}}((d', f'), (d'', f'')) 
          \otimes
          \Homi_{\tr{\CG}{c}}((d, f), (d', f'))
          \ar[dl]_{\pi_2}
          \ar[dr]^{\gamma \otimes U}
          \\
          \Homi_{\tr{\CG}{c}}((d, f), (d', f'))
          \ar[d]_\gamma
          & &
          \Cyl(\HomCG(d', c)) 
          \otimes
          \HomCG(d, d')
          \ar[d]^-{\compr}
          \\
          \Cyl(\HomCG(d, c)) 
          \ar[dr]_{\ss\phantom{\tt}}
          & &
          \Cyl(\HomCG(d, c))
          \ar[dl]^\tt
          \\
          &
          \HomCG(d, c)
          &
          \zbox{\qquad\qquad\qquad.}
        }
      \]
      The fact that this hexagon commutes and that we have indeed
      defined a composition \oo-functor will be verified within the proof of
      the next theorem. In symbols, we have
      \[
        (u', \alpha') \compG (u, \alpha)
        =
        \big(u' \compG u,
        \alpha \compc (\alpha' \compr u)\big) \mpbox.
      \]
      Note that this is an \oo-categorification of the composition of
      triangles
      \[
        \raisebox{2pc}
        {\xymatrix@C=2.5pc@R=2.5pc{
          d \ar[r]^u \ar[dr]_{}="g"_(.40){f}
      & d' \ar[r]^{u'}_(.60){}="fp" \ar[d]_(.60){}="gp"_(.54){f'} & d''
      \ar[dl]_{}="gpp"^(.38){f''} \\
      & c
      \ar@{}"g";[u]_(0.10){}="x"
      \ar@{}"g";[u]_(.75){}="y"
      \ar@<-0.1ex>@2"y";"x"_(.50)\alpha
      \ar@{}"gp";"fp"_(.32){}="x2"
      \ar@{}"gp";"fp"_(.84){}="y2"
      \ar@<0.4ex>@2"y2";"x2"_(0.50){\alpha'\!}
    }}
    \quad
    \mapsto
    \quad
    \raisebox{2pc}
    {\xymatrix@C=1.5pc{
    d \ar[rr]^{u''} \ar[dr]_{f}_(.6){}="f" & & d'' \ar[dl]^(0.48){f''} \\
    & c
    \ar@{}"f";[ur]_(.15){}="ff"
    \ar@{}"f";[ur]_(.55){}="oo"
    \ar@<-0.5ex>@2"oo";"ff"_{\alpha''}
    & \dpbox{,}
    }}
  \]
  with
  \[ u'' = u'u \quadand \alpha'' = \alpha \comp_1 (\alpha' \comp_0 u) \mpbox. \]
  \end{itemize}
\end{paragraph}

\begin{theorem}\label{thm:Gray_slice}
  If $\CG$ is a Gray \oo-category and $c$ is an object of $\CG$, then
  $\tr{\CG}{c}$ as described above is indeed a Gray \oo-category.
\end{theorem}

\begin{proof}
  We start by proving that the composition of $\tr{\CG}{c}$ described in the
  previous paragraph is well defined. We first have to prove that the
  hexagon claimed to be commutative is indeed commutative.
  Fix $(u, \alpha)$ a $k$-cell of $\Homi_{\tr{\CG}{c}}((d, f), (d', f'))$
  and $(u', \alpha')$ a $k$-cell of $\Homi_{\tr{\CG}{c}}((d', f'), (d'',
  f''))$. If we evaluate the left part of the hexagon on $(u', \alpha')
  \otimes (u, \alpha)$, we get $\ss(\alpha)$. Evaluating the right part, we
  get $\tt(\alpha' \compr u)$. But
  \[
    \tt(\alpha' \compr u) = \tt(\alpha') \compG u = f' \compG u = \ss(\alpha)
    \mpbox,
  \]
  where the first equality follows from the fact that $\tt$ is a morphism of
  right $\CG$-modules (see the first square of
  \ref{paragr:cyl_square_map_modules} for $\ee = \tt$). A priori, we have
  shown that the hexagon commutes on ``pure tensors'' $(u', \alpha')
  \otimes (u, \alpha)$. Nevertheless, since each of the equalities we are
  using comes from commutative diagrams, this algebraic proof can be
  transformed into a diagrammatic proof showing that the hexagon commutes,
  without any restriction\footnote{A less elegant argument to
    conclude that the diagram commutes is to use the fact that the ``pure tensors'' $a
    \otimes b$ in a Gray tensor product $A \otimes B$ form a generating
  set under composition.}. From now on, we will freely use this technique to show
  commutativity of diagrams starting from a tensor product.

  To prove that the composition of $\tr{\CG}{c}$ is well defined, we now
  have to show the commutativity of the diagram
  \[
    \begin{tikzcd}[column sep=1.5pc]
      & &[-6pc]
          \Homi_{\tr{\CG}{c}}((d', f'), (d'', f'')) 
          \otimes
          \Homi_{\tr{\CG}{c}}((d, f), (d', f')) 
          \ar[lld]
          \ar[d]
          \ar[rrd]
      &[-6pc]
          \\
          \HomCG(d, d'')
          \ar[dr, "{f'' \comp_0 \var}"']
      &
      &
          \Cyl(\HomCG(d, c))
          \ar[dl, "{\ss\phantom{\tt}}"]
          \ar[dr, "\tt"']
      &
      &
      \Dn{0}
      \ar[dl, "f"]
      \\
      &
          \HomCG(d, c)
      &
      &
          \HomCG(d, c)
      &
          \dpbox.
    \end{tikzcd}
  \]
  For the left square of the diagram, we have
  \[
    \begin{split}
      \ss(\alpha \compc (\alpha' \compr u))
      & =
      \ss(\alpha' \compr u)
      =
      \ss(\alpha') \comp_0 u
      =
      (f'' \comp_0 u') \comp_0 u
      =
      f'' \comp_0 (u' \comp_0 u)
      \mpbox,
    \end{split}
  \]
  using first the internal category structure of the \oo-category of
  cylinders (see \ref{paragr:compc}), then the fact that $\ss$ is a morphism
  of right $\CG$-modules (see the first commutative square
  of~\ref{paragr:cyl_square_map_modules} for $\ee = \ss$) and finally the
  associativity of the composition of the Gray \oo-category $\CG$. This
  proves that the left square commutes. As for the right square, we have
  \[
    \begin{split}
      \tt(\alpha \compc (\alpha' \compr u))
      & =
      \tt(\alpha)
      =
      \id{f}^k
      \mpbox,
    \end{split}
  \]
  using again the internal category structure of the \oo-category of
  cylinders.
  This ends the proof that the composition of $\tr{\CG}{c}$ is well defined.

  We now have to check the axioms of Gray \oo-categories. Let us first prove
  the associativity. Fix a $k$-cell $(u'', \alpha'')$ of
  $\Homi_{\tr{\CG}{c}}((d'', f''), (d''', f'''))$. We have to prove that
  \[
    \big(u'', \alpha''\big) \compG \big((u', \alpha') \compG (u, \alpha)\big)
    =
    \Big(u'' \compG (u' \compG u),\,\,
    \big(\alpha \compc (\alpha' \compr u)\big)
    \compc
    \big(\alpha'' \compr (u' \comp_0 u)\big)
    \Big)
  \]
  equals
  \[
    \big((u'', \alpha'') \compG (u', \alpha')\big) \compG \big(u, \alpha\big)
    =
    \Big((u'' \compG u') \compG u,\,\,
    \alpha
    \compc
    \big[
      \big(\alpha' \compc (\alpha'' \compr u')\big)
      \compr
      u
    \big]
    \Big)
    \mpbox.
  \]
  The equality of the first components follows from the associativity of the
  composition of $\CG$. As for the second components, we have
  \[
    \begin{split}
    \alpha
    \compc
    \big[
      \big(\alpha' \compc (\alpha'' \compr u')\big)
      \compr
      u
    \big]
    & =
    \alpha
    \compc
    \big[
      \big(\alpha' \compr u\big) \compc \big((\alpha'' \compr u') \compr
      u\big)
    \big]
    \\
    & =
    \alpha
    \compc
    \big[
      \big(\alpha' \compr u\big) \compc \big(\alpha'' \compr (u' \comp_0
      u)\big)
    \big]
    \\
    & =
    \big(\alpha \compc (\alpha' \compr u)\big)
    \compc
    \big(\alpha'' \compr (u' \comp_0 u)\big)
    \mpbox,
    \end{split}
  \]
  where the first equality follows from the fact that $\compc$ is a morphism
  of right $\CG$-modules (see the last commutative square of
  \ref{paragr:cyl_square_map_modules}), the second from the fact that
  $\compr$ is a right action (see the commutative square of
  \ref{paragr:compr}) and the last from the associativity of the operation
  $\compc$ (see~\ref{paragr:compc}). This ends the proof that the
  composition of $\tr{\CG}{c}$ is associative.

  Finally, we prove the axioms involving units. For the right unit axiom, we
  have
  \[
    (u, \alpha) \compG \id{(d, f)}
    =
    (u, \alpha) \compG (\id{d}, \idd{f})
    =
    (u \compG \id{d}, \idd{f} \compc (\alpha \compr \id{d}))
    =
    (u, \alpha)
    \mpbox,
  \]
  where the last equality uses the axiom of units in $\CG$,
  the structure of category of \ref{paragr:compc} and the fact that
  $\compr$ is a right action (see the commutative triangle
  of~\ref{paragr:compr}). Finally, for the left unit axiom, we have
  \[
    \id{(d', f')} \compG (u, \alpha)
    =
    (\id{d'}, \idd{f'}) \compG (u, \alpha)
    =
    (\id{d'} \compG u, \alpha \compc (\idd{f'} \compr u))
    =
    (u, \alpha)
    \mpbox,
  \]
  where the last equality uses the axiom of units in $\CG$, the fact that
  $\idd{}$ is a morphism of right $\CG$\=/modules (see the second
  commutative square of \ref{paragr:cyl_square_map_modules}) and
  the structure of internal category of \ref{paragr:compc}.
\end{proof}

\begin{remark}
  The existence of slice Gray \oo-categories was first conjectured by the
  first-named author and Maltsiniotis \cite[conjecture C.24]{AraMaltsiJoint}.
\end{remark}

\begin{remark}
  The definition of the composition \oo-functor of the slice Gray
  \oo-category involves the right action $\compr$. We saw in
  \ref{paragr:enriched_cyl} that for $C$ a strict \oo-category, we need both
  the right action $\compr$ and the left action $\compl$ to express the
  composition \oo-functor of the \oo-category of cylinders $\Cyl C$. But as
  noted in Remark~\ref{rem:compl}, there is no left action~$\compl$ for a
  Gray \oo-category. This seems to indicate that, if $\CG$ is a Gray
  \oo-category, one cannot define a Gray \oo-category of cylinders $\Cyl
  \CG$. We will explain in Appendix~\ref{app} that this is indeed the
  case.
\end{remark}

\begin{paragraph}\label{paragr:def_U}
  Let $\CG$ be a Gray \oo-category and let $c$ be an object of $\CG$. We
  have a canonical Gray \oo-functor
  \[ U \colon \tr{\CG}{c} \to \CG \mpbox, \]
  called the \ndef{forgetful Gray \oo-functor}. It is defined on objects by
  \[ (d, f) \mapsto d \mpbox, \]
  and, if $(d, f), (d', f')$ are two objects, on morphisms by the
  projection
  \[
    U \colon \Homi_{\tr{\CG}{c}}((d, f), (d', f')) \to \Homi_{\CG}(d, d') 
  \]
  (see \ref{paragr:def_Gray_slice}).
\end{paragraph}

We will now prove that our Gray slices are compatible with the slices of
strict \oo-categories. This requires comparing the operation $\compr$ in the
Gray setting (see~\ref{paragr:compr}) with the one defined in the strict
setting (see~\ref{paragr:compr_str}). To carry out this comparison, we
introduce an alternative description of the operation $\compr$ in the Gray
context.

\begin{paragraph}
  If $C$ and $D$ are two \oo-categories, we will denote by
  \[
    \nu \colon
    \Cyl{C} \otimes D \to \Cyl{(C \otimes D)}
  \]
  the \oo-functor obtained as the transpose of the \oo-functor
  \[
    \Dn{1} \otimes \Cyl{C} \otimes D \xto{\ev \otimes D} C \otimes D
    \mpbox,
  \]
  where $\ev$ is the evaluation \oo-functor
  \[
    \Dn{1} \otimes \HomLax(\Dn{1}, C) \to C
    \mpbox.
  \]
\end{paragraph}

\begin{proposition}\label{prop:alt_compr}
  Let $\CG$ be a Gray \oo-category and let $a$, $b$ and $c$ be three objects
  of~$\CG$. Then the \oo-functor $\compr$ of~\ref{paragr:compr} can be
  described as the composite
  \[
    \xymatrix@C=2.2pc{
      \Cyl\HomCG(b,c) \otimes \HomCG(a,b)
      \ar[r]^-{\nu}
      &
      \Cyl(\HomCG(b,c) \otimes \HomCG(a,b))
      \ar[r]^-{\Cyl(\compG)}
      &
      \Cyl\HomCG(a,c)
      \mdpbox.
    }
  \]
\end{proposition}

\begin{proof}
  By adjunction, we have to show the commutativity of the square
  \[
      \small
    \xymatrix@C=1pc{
      \HomCG(a, b)
      \ar[r]^-{\nu^\sharp}
      \ar[d]_{\HomCG(\var, c)}
      &
      \HomLax(\Cyl\HomCG(b, c), \Cyl(\HomCG(b,c) \otimes \HomCG(a,b)))
      \ar[d]^{\HomLax(\Cyl\HomCG(b, c), \Cyl(\compG))}
      \\
      \HomLax(\HomCG(b,c), \HomCG(a, c))
      \ar[r]_{\Cyl{}}
      &
      \HomLax(\Cyl\HomCG(b,c), \Cyl\HomCG(a, c))
      \dpbox,
    }
  \]
  where $\nu^\sharp$ denotes the transpose of $\nu$. Note that the
  \oo-functor $\HomCG(\var, c)$ is the transpose of the composition
  \oo-functor $\compG \colon \HomCG(b,c) \otimes \HomCG(a,b) \to
  \HomCG(a,c)$. More generally, we claim that if $X$, $Y$ and $Z$ are three
  \oo-categories and $m \colon X \otimes Y \to Z$ is any \oo-functor, then
  the square
  \[
    \xymatrix{
      Y
      \ar[r]^-{\nu^\sharp}
      \ar[d]_{m^\sharp}
      &
      \HomLax(\Cyl X, \Cyl(X \otimes Y))
      \ar[d]^{\HomLax(\Cyl X, \Cyl(m))}
      \\
      \HomLax(X, Z)
      \ar[r]_{\Cyl{}}
      &
      \HomLax(\Cyl X, \Cyl Z)
      \dpbox,
    }
  \]
  where $m^\sharp$ denotes the transpose of $m$, commutes. By
  adjunction, this comes down to the commutativity of the square
  \[
    \xymatrix@C=5pc{
      \Dn{1} \otimes \Cyl X \otimes Y
      \ar[rr]^{\ev \otimes Y}
      \ar[d]_{\Dn{1} \otimes \Cyl X \otimes m^\sharp}
      &
      &
      X \otimes Y
      \ar[d]^m
      \\
      \Dn{1} \otimes \Cyl X \otimes \HomLax(X,Z)
      \ar[r]_-{\ev \otimes \HomLax(X, Z)}
      &
      X \otimes \HomLax(X,Z)
      \ar[r]_-{\ev}
      &
      Z
      \dpbox,
    }
  \]
  where $\ev$ denotes the evaluation \oo-functor, which is readily checked.
\end{proof}

\begin{proposition}\label{prop:comp_compr}
  Let $C$ be an \oo-category and let $a$, $b$ and $c$ be three objects
  of~$C$. Then the triangle
  \[
    \xymatrix{
      \Cyl\Homi_C(b,c) \otimes \Homi_C(a,b)
      \ar[r]^-{\compr}
      \ar[d]_\pi
      &
      \Cyl\Homi(a,c)
      \\
      \Cyl\Homi_C(b,c) \times \Homi_C(a,b)
      \ar[ur]_{\compr}
      &
      \phantom{\Cyl\Homi(a,c)}
      \dpbox,
    }
  \]
  where the horizontal arrow is the \oo-functor of~\ref{paragr:compr} for
  the Gray \oo-category associated to $C$ and the diagonal arrow is the
  \oo-functor of~\ref{paragr:compr_str}, commutes.
\end{proposition}

\begin{proof}
  Using the previous proposition, this boils down to showing the
  commutativity of the diagram
  \[
    \xymatrix@C=2.2pc{
      \Cyl\Homi_C(b,c) \otimes \Homi_C(a,b)
      \ar[r]^-{\nu}
      \ar[d]_\pi
      &
      \Cyl(\Homi_C(b,c) \otimes \Homi_C(a,b))
      \ar[r]^-{\Cyl(\compG)}
      \ar[d]_{\Cyl(\pi)}
      &
      \Cyl\Homi_C(a,c)
      \\
      \Cyl\Homi_C(b,c) \times \Homi_C(a,b)
      \ar[r]_{\id{} \times \kk}
      &
      \Cyl(\Homi_C(b,c) \times \Homi_C(a,b))
      \ar[ur]_{\Cyl(\comp_0)}
      &
      \dpbox,
    }
  \]
  where the definition of the bottom-horizontal arrow uses the
  identification
  \[
      \Cyl(\Homi_C(b,c) \times \Homi_C(a,b))
      \simeq
      \Cyl\Homi_C(b,c) \times \Cyl\Homi_C(a,b)
      \mpbox.
  \]
  The triangle of the diagram obviously commutes and it suffices to show
  that the square commutes. More generally, we claim that if $X$ and $Y$ are
  two \oo-categories, then the diagrams
  \[
    \xymatrix{
      \Cyl X  \otimes Y
      \ar[r]^\nu
      \ar[dr]_{\pi_1}
      &
      \Cyl(X \otimes Y)
      \ar[d]^{\Cyl(\pi_1)}
      \\
      &
      \Cyl X
    }
    \qquad
    \raisebox{-1.5pc}{and}
    \qquad
    \xymatrix{
      \Cyl X  \otimes Y
      \ar[r]^\nu
      \ar[d]_{\pi_2}
      &
      \Cyl(X \otimes Y)
      \ar[d]^{\Cyl(\pi_2)}
      \\
      Y
      \ar[r]_\kk
      &
      \Cyl Y
    }
  \]
  commute. This follows from the naturality of $\nu$ applied to the
  \oo-functors $Y \to \Dn{0}$ and $X \to \Dn{0}$, respectively.
\end{proof}

\begin{proposition}
  Let $C$ be an \oo-category and let $c$ be an object of $C$. Then
  we have a canonical natural isomorphism
  \[ \tr{\iota(C)}{c} \simeq \iota(\tr{C}{c}) \mpbox, \]
  commuting with the forgetful morphisms, where $\iota$ denotes the inclusion
  functor from \oo-categories to Gray \oo-categories.
\end{proposition}

\begin{proof}
  This is true by design of slice Gray \oo-categories, and more precisely,
  by the description of slice \oo-categories that follows from the
  description of comma \oo-categories given in~\ref{paragr:enriched_comma}
  and Proposition~\ref{prop:comp_compr}.
\end{proof}

\begin{proposition}\label{prop:fib_forgetful}
  Let $\CG$ be a Gray \oo-category and let $c$ and $d$ be two objects of
  $\CG$. Then the fiber of the forgetful Gray \oo-functor $\tr{\CG}{c} \to
  \CG$ at $d$ is canonically isomorphic to $\Homi_\CG(d,c)^\o$.
\end{proposition}

\begin{proof}
  Denote by $U_d$ this fiber. By definition, its objects are $1$-cells $d
  \to c$ of $\CG$, and if $f$ and $f'$ are two such objects, we have
  \[ 
    \Homi_{U_d}(f, f') = \{f'\} \Cylpb{\HomCG(d, c)} \{f\}
    = \{f'\} \comma_{\HomCG(d, c)} \{f\}
    \mpbox.
  \]
  This means that a $k$-cell of this \oo-category is a $k$-cylinder $\alpha$
  in $\HomCG(d, c)$ such that $\ss(\alpha) = \id{f'}$ and~$\tt(\alpha) =
  \id{f}$. Moreover, the composition \oo-functor simplifies to
  \[
    \begin{split}
    \Homi_{U_d}(f', f'')
    \otimes
    \Homi_{U_d}(f, f')
    &
    \to
    \Homi_{U_d}(f, f'')
    \\
    (\alpha', \alpha)
    \quad
    \qquad
    \qquad
    &
    \mapsto
    \quad
    \alpha' \compc \alpha
    \end{split}
  \]
  (as $\alpha' \compc (\alpha \compr \id{d}) = \alpha' \compc \alpha$).

  But in general, if $a$ and $a'$ are two objects of an \oo-category $A$, we
  have a canonical isomorphism
  \[ \{a\} \comma_A \{a'\} \toiso \Homi_A(a, a')^\o \]
  (see \cite[Proposition B.6.2]{AraThmB}), sending a $k$-cylinder $\alpha$
  in $A$ to its principal cell $\alpha_k$ (see~\ref{paragr:def_cyl}). We
  thus have
  \[
    \Homi_{U_d}(f, f')
    \simeq
    \Homi_{\HomCG(d, c)}(f', f)^\o
    =
    \Homi_{\HomCG(d, c)^\o}(f, f')
    \mpbox,
  \]
  this isomorphism sending a $k$-cylinder $\alpha$ in $\HomCG(d, c)$ to its
  principal cell $\alpha_k$. Now if $(\alpha', \alpha)$ is in
  $\Homi_{U_d}(f', f'')_k \times \Homi_{U_d}(f, f')_k$, then
  \[
  (\alpha' \compc \alpha)_k
  = \alpha'_k \comp_0 \alpha^{}_k
  \mpbox.
  \]
  This shows that the isomorphism
  \[
    \Homi_{U_d}(f, f')
    \simeq
    \Homi_{\HomCG(d, c)^\o}(f, f')
  \]
  is compatible with compositions, thereby proving the result.
\end{proof}

\begin{paragraph}\label{paragr:dual_slice}
  Let $C$ be a strict \oo-category and let $c$ be an object of $C$. If $D$
  is any duality of $\ooCat$ (see~\ref{paragr:def_dual}), then by
  ``conjugating'' the slice construction $\tr{C}{c}$ by $D$, one gets
  another slice construction. If $D$ does not reverse the $1$-cells, we set
  \[ \trD{D}{C}{c} = D\big(\tr{D(C)}{c}\big) \mpbox. \]
  In the case where $D$ reverses the $1$-cells, this notation would be
  misleading as $D\big(\tr{D(C)}{c}\big)$ is actually a slice below $c$ (and
  not above $c$). In particular, if $D$ is the total dual, following the
  notation of~\cite{AraMaltsiJoint}, we set
  \[ \cotr{C}{c} = {(\tr{C^\o}{c})}^\o \mpbox. \]
  Now if $D$ is a general duality reversing the $1$-cells, then 
  denoting by $D'$ the unique duality such that $D' \circ D$ is the total
  dual (in particular, $D'$ does not reverse the $1$-cells), one gets
  \[
    D\big(\tr{D(C)}{c}\big)
    =
    D'\Big({\big(\tr{{D'(C)}^\o}{c}\big)}^\o\Big)
    =
    D'\big(\cotr{D'(C)}{c}\big)
    \mpbox.
  \]
  Therefore, if $D$ reverses the $1$-cells, we set
  \[
    \cotrD{D'}{C}{c}
    =
    D\big(\tr{D(C)}{c}\big)
    \mpbox.
  \]
  This means that we only decorate (over or under) slices by dualities that do not reverse
  the $1$-cells.

  Let us now apply this to Gray \oo-categories. If $\CG$ is a Gray
  \oo-category and $c$ is an object of $\CG$, we can conjugate our Gray
  slice construction by the three non-trivial dualities of Gray \oo-categories
  (see~\ref{paragr:Gray_dual}) and we set
  \[
      \cotrD{\dco}{\CG}{c} = {(\tr{\CG^\dop}{c})}^\dop\mpbox,
      \qquad
      \cotrD{\dtop}{\CG}{c} = {(\tr{\CG^\dcot}{c})}^\dcot\mpbox,
      \qquad
      \trD{\dto}{\CG}{c} = {(\tr{\CG^\dto}{c})}^\dto\mpbox.
  \]
  Note that each of these Gray \oo-categories admits a forgetful Gray
  \oo-functor to $\CG$.

  Similarly, if $\CG$ is an anti Gray \oo-category and $c$ is an object of
  $\CG$, we set
  \[
  \cotr{\CG}{c} = (\tr{\CG^\o}{c})^\o\mpbox,
  \qquad
  \trD{\dco}{\CG}{c} = {(\cotr{\CG^\dop}{c})}^\dop\mpbox,
  \qquad
  \trD{\dtop}{\CG}{c} = {(\cotr{\CG^\dcot}{c})}^\dcot \mpbox,
  \qquad
  \cotrD{\dto}{\CG}{c} = {(\cotr{\CG^\dto}{c})}^\dto\mpbox.
  \]
  Each of these anti Gray \oo-categories admits a forgetful anti Gray
  \oo-functor to $\CG$.
\end{paragraph}

\section{Gray functorialities of the comma construction}

The purpose of this section is to extend the comma construction
\[
  \begin{tikzcd}[column sep=1pc]
      A \ar[r] & C & \ar[l] B
    \end{tikzcd}
    \quad
    \mapsto
    \quad
    A \comma_C B
\]
to a Gray \oo-functor
\[
  \var \comma_C \var
  \colon
  \tr{\ooCatOpLax}{C} \times \trto{\ooCatOpLax}{C}
  \to \ooCatOpLax
  \mpbox,
\]
where $\tr{\ooCatOpLax}{C}$ is the slice Gray \oo-category obtained by
applying Theorem~\ref{thm:Gray_slice} to the Gray \oo-category
$\ooCatOpLax$, the slice
\smash{$\trD{\raisebox{-3pt}{\hbox{$\scriptstyle\dto$}}}{\ooCatOpLax}{C}$}
is obtained by duality (see~\ref{paragr:dual_slice}) and the product is the
categorical product of Gray \oo-categories
(see~\ref{paragr:enriched_product}).

\renewcommand{\baselinestretch}{0.99}

\begin{paragraph}
  Let us fix \oo-functors
  \[
    \begin{tikzcd}[column sep=1.5pc]
      A \ar[r,"f"] & C & \ar[l,"g"'] B
    \end{tikzcd}
    \quadand
    \begin{tikzcd}[column sep=1.5pc]
      A' \ar[r,"f'"] & C & \ar[l,"g'"'] B'
    \end{tikzcd}
    \mpbox.
  \]
  To any $2$-square
  \[
    \xymatrix@C=1.5pc@R=1.5pc{
      & T \ar[dl]_{a} \ar[dr]^{b} \\
        A \ar[dr]_f \ar@{}[rr]_(.35){}="x"_(.65){}="y"
        \ar@2"x";"y"^{\lambda} 
        & & B \ar[dl]^g \\
        & C &
      }
    \]
  in $\ooCatOpLax$, we are going to associate an \oo-functor
  \[
    \begin{tikzcd}
    K \colon
    \Homi_{\tr{\ooCatOpLax}{C}}((A, f), (A', f'))
      \times
    \Homi_{\trto{\ooCatOpLax}{C}}((B, g), (B', g'))
    \ar[d]
    \\
    \displaystyle\HomOpLax(T,A') \commabig_{\HomOpLax(T,C)} \HomOpLax(T, B')
    \dpbox.
    \end{tikzcd}
  \]
  By definition, we have
  \[
    \begin{split}
      \MoveEqLeft
    \Homi_{\tr{\ooCatOpLax}{C}}((A, f), (A', f'))
    \\
    & =
    \HomOpLax(A,A') \times_{\HomOpLax(A,C)} \Cyl\HomOpLax(A,C)
    \times_{\HomOpLax(A,C)} \{f\}
    \end{split}
  \]
  and, as the duality $D_\dto$ consists in applying the total dual
  hom-wise, we get
  \[
    \begin{split}
      \MoveEqLeft
    \Homi_{\trto{\ooCatOpLax}{C}}((B, g), (B', g'))
    \\
    & \! =
    \Big(\HomOpLax(B,B')^\o \times_{\HomOpLax(B,C)^\o}
      \Cyl\big(\HomOpLax(B,C)^\o\big) \times_{\HomOpLax(B,C)^\o} \{g\}\Big)^\o
    \\
    & \! \simeq
     \{g\} \times_{\HomOpLax(B,C)} \Cyl\HomOpLax(B,C)
     \times_{\HomOpLax(B,C)} \HomOpLax(B,B') 
     \mpbox,
    \end{split}
  \]
  as \[ \big(\Cyl(X^\o)\big)^\o \simeq \HomLax(\Dn{1}^\o, X) \mpbox, \]
  which is isomorphic to $\Cyl X$ but in a way that exchanges $\ss$ and $\tt$.

  This means that a $k$-cell in
  \[
    \Homi_{\tr{\ooCatOpLax}{C}}((A, f), (A', f'))
      \times
    \Homi_{\trto{\ooCatOpLax}{C}}((B, g), (B', g'))
  \]
  consists of a $4$-tuple
  \[ (u, \alpha, \beta, v) \]
  in
  \[
    \begin{tikzcd}
    \HomOpLax(A, A')_k \times (\Cyl{\HomOpLax(A,C)})_k \times
    (\Cyl{\HomOpLax(B,C)})_k 
    \times \HomOpLax(B,B')_k
    \end{tikzcd}
  \]
  satisfying
  \[
    \ss(\alpha) = f' \compG u, \quad \tt(\alpha) = f,
    \quad \ss(\beta) = g, \quad \tt(\beta) = g' \compG v
    \mpbox.
  \]
  In particular, for $k = 0$, we get a diagram
      \[
        \xymatrix@R=1pc@C=3pc{
          A \ar[dd]_u \ar[dr]^f_{}="f" & & B \ar[dl]_g_{}="g" \ar[dd]^v \\
            & C \\
          A' \ar[ur]_{f'} & & B' \ar[ul]^{g'}
          \ar@{}[ll];"f"_(0.35){}="sa"_(0.85){}="ta"
          \ar@2"sa";"ta"^{\alpha}
          \ar@{}[];"g"_(0.35){}="tb"_(0.85){}="sb"
          \ar@2"sb";"tb"^{\beta}
          \dpbox.
        }
      \]
  Similarly, a $k$-cell of
  \[
    \HomOpLax(T,A') \commabig_{\HomOpLax(T,C)} \HomOpLax(T, B')
  \]
  consists of a triple
  \[ (a', \lambda', b') \]
  in
  \[
    \begin{tikzcd}
    \HomOpLax(T,A')_k \times  (\Cyl{\HomOpLax(T,C)})_k \times \HomOpLax(T,
    B')_k
    \end{tikzcd}
  \]
  such that
  \[ \ss(\lambda') = f' \compG a' \quadand \tt(\lambda') = g' \compG b' \mpbox. \]
  For $k = 0$, we get a $2$-square
  \[
    \xymatrix@C=1.5pc@R=1.5pc{
      & T \ar[dl]_{a'} \ar[dr]^{b'} \\
        A' \ar[dr]_{f'} \ar@{}[rr]_(.35){}="x"_(.65){}="y"
        \ar@2"x";"y"^{\lambda'} 
        & & B' \ar[dl]^{g'} \\
        & C & \dpbox{.}
      }
    \]

  The \oo-functor $K$ is defined by \oo-categorification of the formula
  giving the total composite of the $2$-diagram
      \[
        \xymatrix@R=1pc@C=3pc{
          & T \ar[dl]_{a} \ar[dr]^{b} \\
          A \ar[dd]_u \ar[dr]^f_{}="f" & & B \ar[dl]_g_{}="g" \ar[dd]^v
      \ar@{}[ll];[]_(0.40){}="x"_(0.60){}="y"
      \ar@2"x";"y"^{\lambda}
          \\
            & C \\
          A' \ar[ur]_{f'} & & B' \ar[ul]^{g'}
          \ar@{}[ll];"f"_(0.35){}="sa"_(0.85){}="ta"
          \ar@2"sa";"ta"^{\alpha}
          \ar@{}[];"g"_(0.35){}="tb"_(0.85){}="sb"
          \ar@2"sb";"tb"^{\beta}
          \dpbox{,}
        }
        \]
  that is, by the formula
  \[
    (u, \alpha, \beta, v)
    \mapsto
    (u \compG a, (\beta \compr b) \compc \lambda \compc (\alpha \compr a), v
    \compG b)
    \mpbox,
  \]
  where $\compc$ denotes the internal composition of cylinders of
  \ref{paragr:compc} and $\compr$ the right action of \ref{paragr:compr}.
  This formula is well defined as, first, by~\ref{paragr:compr},
  \[ 
    \ss(\beta \compr b) = \ss(\beta) \compG b = g \compG b
    = \tt(\lambda)
    \quadand
    \tt(\alpha \compr a) = \tt(\alpha) \compG a = f \compG a
    = \ss(\lambda)
    \mpbox,
  \]
  so that $\lambda' = (\beta \compr b) \compc \lambda \compc (\alpha \compr
  a)$ makes sense, and, second, using~\ref{paragr:compc} and
  again~\ref{paragr:compr},
  \[
    \ss(\lambda')
    =
    \ss(\alpha \compr a)
    = \ss(\alpha) \compG a
    = (f' \compG u) \compG a
    = f' \compG (u \compG a)
    \mpbox,
  \]
  and similarly
  \[
    \tt(\lambda') 
    = g' \compG (v \compG b)
    \mpbox.
  \]
  Moreover, this formula is \oo-functorial in $(u, \alpha, \beta, v)$ as it
  only involves the operations~$\compG$, $\compc$ and $\compr$, which are all
  \oo-functorial. (We could also have defined the \oo-functor~$K$ in a more
  categorical way, as we did for the composition $\compG$ of Gray slices in
  \ref{paragr:def_Gray_slice}.)
\end{paragraph}

\renewcommand{\baselinestretch}{0.95}

\begin{paragraph}\label{paragr:formula_higher_comma}
  If we apply the construction of the previous paragraph to the universal
  $2$\=/square
  \[
    \xymatrix@C=1pc@R=1.5pc{
      & A \comma_C B \ar[dl]_{p_1} \ar[dr]^{p_2} \\
      A \ar[dr]_f \ar@{}[rr]_(.35){}="x"_(.65){}="y"
      \ar@2"x";"y"^{\gamma} 
      & & B \ar[dl]^g \\
      & C & \mpbox,
    }
  \]
  we get an \oo-functor
  \[
    \begin{tikzcd}
    \var \comma_C \var \colon
    \Homi_{\tr{\ooCatOpLax}{C}}((A, f), (A', f'))
      \times
    \Homi_{\trto{\ooCatOpLax}{C}}((B, g), (B', g'))
    \ar[d]
    \\
    \displaystyle\HomOpLax(A \comma_C B ,A')
    \commabig_{\HomOpLax(A \comma_C B,C)} \HomOpLax(A \comma_C B, B')
    \ar[d, phantom, "\simeq"]
    \\
    \HomOpLax(A \comma_C B, A' \comma_C B')
    \dpbox,
    \end{tikzcd}
  \]
  where the isomorphism is the higher universal property of the comma
  construction (see Proposition~\ref{prop:univ}). Explicitly, the \oo-functor
  $\var \comma_C \var$ is given by the formula
  \[
    (u, \alpha, \beta, v)
    \mapsto
    (u \compG p_1, (\beta \compr p_2) \compc \gamma \compc (\alpha \compr
    p_1), v \compG p_2)
    \mpbox.
  \]

  In particular, if $g \colon B \to C$ is an \oo-functor, we have
  \[
    \begin{split}
      (\var \comma_C B)(u, \alpha)
      & =
      (\var \comma_C \var)(u, \alpha, \idd{g}, \id{B})
      \\
      & =
      (u \compG p_1,
      (\idd{g} \compr p_2) \compc \gamma \compc (\alpha \compr p_1),
      \id{B} \compG p_2)
      \\
      & =
      (u \compG p_1, \gamma \compc (\alpha \compr p_1), p_2)
      \mpbox,
    \end{split}
  \]
  where the last equality uses the fact that $\idd{}$ is a morphism of right
  $\ooCatOpLax$\=/modules (see the second commutative square of
  \ref{paragr:cyl_square_map_modules}) and the structure of internal
  category of~\ref{paragr:compc}, and, similarly, if $f \colon A \to C$ is
  an \oo-functor, we have
  \[
    \begin{split}
      (A \comma_C \var)(\beta, v)
      & =
      (p_1, (\beta \compr p_2) \compc \gamma, v \compG p_2)
      \mpbox.
    \end{split}
  \]
\end{paragraph}

\renewcommand{\baselinestretch}{0.96}

\begin{theorem}\label{thm:comma_Gray}
  Let $C$ be an \oo-category. The comma construction extends, via the
  construction of the previous paragraph, to a Gray \oo-functor
  \[
    \var \comma_C \var
    \colon
    \tr{\ooCatOpLax}{C} \times \trto{\ooCatOpLax}{C}
    \to \ooCatOpLax
    \mpbox.
  \]
\end{theorem}

\begin{proof}
  \newcommand\HomC{\Homi_{\tr{}C}}
  \newcommand\HomtC{\Homi_{\trto{}C}}
  \newcommand\HomOl{\Homi_{\mathrm{ol}}}
  In this proof, to make our formula more compact, we set
  \[ 
    \HomOl = \HomOpLax\mpbox,
    \,\,
    \,\,\,\,
    \HomC = \Homi_{\tr{\ooCatOpLax}{C}}
    \,
    \,\,\,\,\text{and}\,\,\,\,
    \,
    \HomtC =\Homi_{\trto{\ooCatOpLax}{C}}
    \mpbox.
  \]
  For the same reason, if $(T, T \to C)$ is an \oo-category over $C$, we
  will denote it simply by $T$.

  Let us now prove the result. We have to check the compatibility with
  the unit and the compatibility with the composition. To do so, we will use
  the same technique as in the proof of Theorem \ref{thm:Gray_slice}.

  For the unit, we have
  \[
    \begin{split}
      (\var \comma_C \var)\big(\id{(A, f)}, \id{(B, g)}\big)
      & =
      (\var \comma_C \var)(\id{A}, \idd{f}, \idd{g}, \id{B})
      \\
      & =
      (\id{A} \compG p_1,
      (\idd{g} \compr p_2) \compc \gamma \compc (\idd{f} \compr p_1),
      \id{B} \compG p_2)
      \\
      & =
      (p_1, \gamma, p_2)
      \\
      & =
      \id{A \comma_C B}
      \mpbox,
    \end{split}
  \]
  using the fact that $\idd{}$ is a morphism of right
  $\ooCatOpLax$\=/modules (see the second commutative square of
  \ref{paragr:cyl_square_map_modules}) and the structure of internal
  category of \ref{paragr:compc}.
  
  Let us now check the compatibility with the composition.
  Consider $A$, $A'$, $A''$, $B$, $B'$, $B''$ six \oo-categories over $C$.
  We have to prove that the two canonical \oo-functors
  \[
    \begin{tikzcd}
      \big(
        \HomC(A', A'') \times  \HomtC(B', B'')
      \big)
      \otimes
      \big(
        \HomC(A, A') \times  \HomtC(B, B')
      \big)
    \ar[d]
    \\
    \HomOl(A\comma_C B, A'' \comma_C B'') \dpbox,
    \end{tikzcd}
  \]
  which we will describe below, are equal.
  Consider
  \[ (u', \alpha', \beta', v') \quadand (u, \alpha, \beta, v) \]
  cells of
  \[
        \HomC(A', A'') \times  \HomtC(B', B'')
        \quadand
        \HomC(A, A') \times  \HomtC(B, B')
  \]
  respectively. When these cells are $0$-cells, we get a
  diagram
  \[
    \xymatrix@R=2pc@C=3pc{
        A \ar[d]_u \ar[dr]^(0.50)f_{}="f" & & B \ar[dl]_(0.60)g_{}="g"
        \ar[d]^v \\
        A' \ar[r]^(0.60){f'}_(.70){}="f'" \ar[d]_{u'\!}_(.70){}="u'"
        \ar@2[];"f"^(.68){\alpha\phantom{'}}
         & C & B' \ar[l]_(0.65){g'}_(.70){}="g'" \ar[d]^{v'}_(.70){}="v'"
        \ar@2"g";[]^(0.30){\beta} 
         \\
        A'' \ar[ur]_(0.51){f''}
        \ar@{}"u'";"f'"_(.20){}="sa'"^(.80){}="ta'"
        \ar@2"sa'";"ta'"^{\alpha'}
        & & B'' \ar[ul]^(0.52){g''}_{}="g''"
        \ar@{}"g'";"v'"_(.20){}="sb'"^(.80){}="tb'"
        \ar@2"sb'";"tb'"^{\beta'}
        \dpbox.
      }
  \]
  
  Consider first the \oo-functor $M$ defined as the composite
  \[
    \begin{tikzcd}
      \big(
        \HomC(A', A'') \times  \HomtC(B', B'')
      \big)
      \otimes
      \big(
        \HomC(A, A') \times  \HomtC(B, B')
      \big)
    \ar[d, "\can"']
    \\
    \big(\HomC(A', A'') \otimes \HomC(A, A')\big) \times \big(\HomtC(B', B'')
    \otimes \HomtC(B, B')\big)
    \ar[d, "\compG \,\otimes\, \compG"']
    \\
    \HomC(A, A'') \times \HomtC(B, B'')
    \ar[d, "{\var \comma_C \var}"']
    \\
    \HomOl(A\comma_C B, A'' \comma_C B'')
    \dpbox.
    \end{tikzcd}
  \]
  We have
  \[
    \begin{split}
      \MoveEqLeft
      M\big((u', \alpha', \beta', v') \otimes (u, \alpha, \beta, v)\big)
        \\
        & =
        \big(\var \comma_C \var\big)
        \big((u', \alpha') \compG (u, \alpha), (\beta', v') \compG (\beta,
        v)\big)
      \\
        & =
        \big(\var \comma_C \var\big)
        \big(u' \compG u, \alpha \compc (\alpha' \compr u), (\beta' \compr
          v) \compc \beta, v' \compG v\big)
      \\
        & =
      \Big(
        (u' \compG u) \compG p_1,
        \\
        & \qquad
        \big[((\beta' \compr v) \compc \beta) \compr p_2\big]
        \compc \gamma \compc
        \big[(\alpha \compc (\alpha' \compr u)) \compr p_1\big],
        \\
        & \qquad
        (v' \compG v) \compG p_2
      \Big) \mpbox.
    \end{split}
  \]
  Note that this last formula is an \oo-categorification of the formula for
  the total composite of the $2$\=/diagram
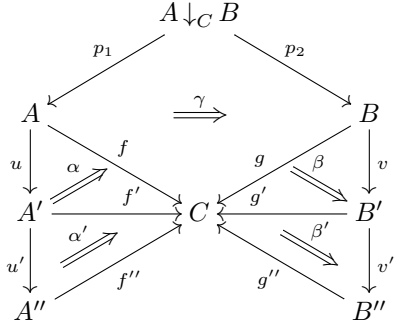
\[
    \xymatrix@R=2pc@C=3pc{
      & A \comma_C B \ar[dl]_{p_1} \ar[dr]^{p_2} &
      \\
      A \ar[d]_u \ar[dr]^(0.50)f_{}="f" & & B \ar[dl]_(0.60)g_{}="g"
      \ar[d]^v
      \ar@{}[ll];[]_(0.40){}="x"_(0.60){}="y"
      \ar@2"x";"y"^{\gamma}
      \\
      A' \ar[r]^(0.60){f'}_(.70){}="f'" \ar[d]_{u'\!}_(.70){}="u'"
      \ar@2[];"f"^(.68){\alpha\phantom{'}}
       & C & B' \ar[l]_(0.65){g'}_(.70){}="g'" \ar[d]^{v'}_(.70){}="v'"
      \ar@2"g";[]^(0.30){\beta} 
       \\
      A'' \ar[ur]_(0.51){f''}
      \ar@{}"u'";"f'"_(.20){}="sa'"^(.80){}="ta'"
      \ar@2"sa'";"ta'"^{\alpha'}
      & & B'' \ar[ul]^(0.52){g''}_{}="g''"
      \ar@{}"g'";"v'"_(.20){}="sb'"^(.80){}="tb'"
      \ar@2"sb'";"tb'"^{\beta'}
      \dpbox.
    }
  \]

  Consider now the \oo-functor $N$ obtained as the composite
  \[
    \begin{tikzcd}
      \big(
        \HomC(A', A'') \times  \HomtC(B', B'')
      \big)
      \otimes
      \big(
        \HomC(A, A') \times  \HomtC(B, B')
      \big)
      \ar[d, "{(\var \comma_C \var) \otimes (\var \comma_C \var)}"']
      \\
      \HomOl(A' \comma_C B', A'' \comma_C B'')
      \otimes
      \HomOl(A \comma_C B, A' \comma_C B')
      \ar[d, "\compV"']
      \\
      \HomOl(A\comma_C B, A'' \comma_C B'')
      \dpbox.
    \end{tikzcd}
  \]
  Let us compute $N\big((u', \alpha', \beta', v') \otimes (u, \alpha, \beta,
  v)\big)$. Denote by
  \[
    \xymatrix@C=1.4pc@R=2pc{
      & A' \comma_C B' \ar[dl]_{p'_1} \ar[dr]^{p'_2} \\
      A' \ar[dr]_{f'} \ar@{}[rr]_(.35){}="x"_(.65){}="y"
      \ar@2"x";"y"^{\gamma'} 
      & & B' \ar[dl]^{g'} \\
      & C & \dpbox,
    }
    \qquad
    \xymatrix@C=1.5pc@R=2pc{
      & A \comma_C B \ar[dl]_{p_1} \ar[dr]^{p_2} \\
      A \ar[dr]_f \ar@{}[rr]_(.35){}="x"_(.65){}="y"
      \ar@2"x";"y"^{\gamma} 
      & & B \ar[dl]^g \\
      & C &
    }
  \]
  the two universal $2$-squares involved. By definition, we have
  \[
    \begin{split}
      (\var \comma_C \var)(u', \alpha', \beta', v')
      & =
      (u' \compG p'_1, (\beta' \compr p'_2) \compc \gamma' \compc (\alpha' \compr
      p'_1), v' \compG p'_2)
      \mpbox,
      \\
      (\var \comma_C \var)(u, \alpha, \beta, v)
      & =
      (u \compG p_1, (\beta \compr p_2) \compc \gamma \compc (\alpha \compr
      p_1), v \compG p_2) \mpbox,
    \end{split}
  \]
  and we have to compute the composition of these two cells in
  $\ooCatOpLax$. We have
  \[
    \begin{split}
      \MoveEqLeft
      N\big((u', \alpha', \beta', v') \otimes (u, \alpha, \beta,
      v)\big)\\
      &=
      \Big(
        u' \compG p'_1, (\beta' \compr p'_2) \compc \gamma' \compc (\alpha' \compr
      p'_1), v' \compG p'_2
      \Big)
      \\
      &
      \qquad
      \compG
      \Big(
        u \compG p_1, (\beta \compr p_2) \compc \gamma \compc (\alpha \compr
        p_1), v \compG p_2
      \Big)
      \\
      & =
      \Big(
        u' \compG (u \compG p_1),
        \\
      & \qquad
        \big[\beta' \compr (v \compG p_2)\big]
        \compc
        \big[
          (\beta \compr p_2) \compc \gamma \compc (\alpha \compr p_1)
        \big]
        \compc
        \big[\alpha' \compr (u \compG p_1)\big],
        \\
      & \qquad
        v' \compG (v \compG p_2)
      \Big)
      \mpbox.
    \end{split}
  \]
  The axiom we are checking is thus equivalent to the equalities
  \[
        u' \compG (u \compG p_1)
        =
        (u' \compG u) \compG p_1
        \mpbox,
        \qquad
        v' \compG (v \compG p_2)
        =
        (v' \compG v) \compG p_2
  \]
  and
  \[
    \begin{split}
      \MoveEqLeft
        \big[\beta' \compr (v \compG p_2)\big]
        \compc
        \big[
          (\beta \compr p_2) \compc \gamma \compc (\alpha \compr p_1)
        \big]
        \compc
        \big[\alpha' \compr (u \compG p_1)\big]
        \\
        & =
        \big[((\beta' \compr v) \compc \beta) \compr p_2\big]
        \compc \gamma \compc
        \big[(\alpha \compc (\alpha' \compr u)) \compr p_1\big]
        \mpbox.
    \end{split}
  \]
  The two first equalities are obviously true. As for the last one, we have
  \[
    \begin{split}
      \MoveEqLeft
        \big[\beta' \compr (v \compG p_2)\big]
        \compc
        \big[
          (\beta \compr p_2) \compc \gamma \compc (\alpha \compr p_1)
        \big]
        \compc
        \big[\alpha' \compr (u \compG p_1)\big]
        \\
        & =
        \Big(
        \big[\beta' \compr (v \compG p_2)\big]
        \compc
        \big[
        \beta \compr p_2 \big]
        \Big)
        \compc \gamma \compc
        \Big(
        \big[\alpha \compr p_1\big]
        \compc
        \big[\alpha' \compr (u \compG p_1)\big]
        \Big)
        \\
        & =
        \Big(
        \big[(\beta' \compr v) \compr p_2\big]
        \compc
        \big[
        \beta \compr p_2 \big]
        \Big)
        \compc \gamma \compc
        \Big(
        \big[\alpha \compr p_1\big]
        \compc
        \big[(\alpha' \compr u) \compr p_1\big]
        \Big)
        \\
        & =
        \big[((\beta' \compr v) \compc \beta) \compr p_2\big]
        \compc \gamma \compc
        \big[(\alpha \compc (\alpha' \compr u)) \compr p_1\big]
        \mpbox,
    \end{split}
  \]
  where the first equality follows from the associativity of $\compc$, the
  second from the fact that~$\compr$ is a right module action
  (see~\ref{paragr:compr}) and the last one from the fact that this action
  is compatible with~$\compc$ (see the last square of
  \ref{paragr:cyl_square_map_modules}).
\end{proof}

\renewcommand\baselinestretch{1}

\begin{corollary}\label{coro:comma_Gray_one_var}
   If $B \to C$ is an \oo-functor, then we have a Gray \oo-functor
      \[
        \var \comma_C B
        \colon
        \tr{\ooCatOpLax}{C}
        \to \ooCatOpLax
        \mpbox,
      \]
  and if $A \to C$ is an \oo-functor, then we have a Gray \oo-functor
      \[
        A \comma_C \var
        \colon
        \trto{\ooCatOpLax}{C}
        \to \ooCatOpLax
        \mpbox.
      \]
\end{corollary}

\begin{remark}\label{rem:comma_restr_fib}
  If $A$ and $B$ are two fixed \oo-categories, there is a canonical
  embedding
  \[
    \HomOpLax(A, C)^\o \times \HomOpLax(B, C)^\dto
    \hookto
    \tr{\ooCatOpLax}{C} \times \trto{\ooCatOpLax}{C}
    \mpbox.
  \]
  Indeed, by Proposition~\ref{prop:fib_forgetful}, the \oo-category
  $\HomOpLax(A, C)^\o$ canonically embeds in $\tr{\ooCatOpLax}{C}$, and, by
  duality, this implies that
  \[
    \big(\Homi_{(\ooCatOpLax)^\dto}(B, C)^\o\big)^\dto
    =
    \Homi_{(\ooCatOpLax)^\dto}(B, C)^\dt
    =
    \Homi_{\ooCatOpLax}(B, C)^\dto
  \]
  embeds in $\trto{\ooCatOpLax}{C}$.

  The comma construction thus restricts to a Gray \oo-functor
  \[
    \var \comma_C \var \colon
    \HomOpLax(A, C)^\o \times \HomOpLax(B, C)^\dto
    \to
    \ooCatOpLax
    \mpbox.
  \]
\end{remark}

\begin{remark}\label{rem:dual_bifun_lax}
  By duality, one can deduce that the \emph{lax} comma construction
  $\var \commalax \var$ (see~\ref{paragr:dual_comma}) defines an
  \emph{anti} Gray \oo-functor. This follows from the formula $A \commalax_C
  B \simeq {\big(A^\co \comma_{C^\co} B^\co\big)}^\co$
  (see again~\ref{paragr:dual_comma}). Indeed, by~\ref{paragr:duality_Gray},
  the operation $X \mapsto X^\co$ defines an isomorphism of Gray \oo-categories
  $D_\co \colon (\ooCatLax)^\dtop \to \ooCatOpLax$ and an isomorphism of
  anti Gray \oo-categories $D'_\co \colon (\ooCatOpLax)^\dtop \to
  \ooCatLax$, and we can consider the chain of
  anti Gray \oo-functors
  \[
    \small
    \begin{tikzcd}
      {\big(\tr{(\ooCatLax)^\dtop}{C}\big)}^\dtop
      \times
      {\big(\trto{(\ooCatLax)^\dtop}{C}\big)}^\dtop
      \ar[d]
      \\
      {\big(\tr{\ooCatOpLax}{C^\co}\big)}^\dtop
      \times
      {\big(\trto{\ooCatOpLax}{C^\co}\big)}^\dtop
      \ar[r, "\var \comma_{C^\co} \var"]
      &
      (\ooCatOpLax)^\dtop
      \ar[r, "D'_\co"]
      &
      \ooCatLax
      \dpbox,
    \end{tikzcd}
  \]
  where the vertical arrow is induced by $D_\co$.
 Composing this chain, we get an anti Gray \oo-functor
  \[
      \var \commalax_C \var
      \colon
      \trD{\dtop}{\ooCatLax}{C}
      \times
      \trD{\dco}{\ooCatLax}{C}
      \to
      \ooCatLax
      \mpbox.
  \]
\end{remark}

\begin{remark}
  The two Gray \oo-functors of Corollary~\ref{coro:comma_Gray_one_var} can
  be deduced from each other using the formula $A \comma_C B \simeq
  {\big(B^\o \comma_{C^\o} A^\o\big)}^\o$ of \ref{paragr:dual_comma}. For
  instance, if $u \colon A \to C$ is an \oo-functor, the Gray \oo-functor $A
  \comma_C \var$ can be identified with the composite of the Gray
  \oo-functors
  \[ 
    \small
    \begin{tikzcd}[column sep=2.1pc]
        {\big(\tr{(\ooCatOpLax)^\dto}{C}\big)}^\dto
        \ar[r, "D_\o"]
        &
        {\big(\tr{\ooCatOpLax}{C^\o}\big)}^{\!\dto}
        \ar[r, "{\var \comma_{C^\o} \! A^\o}"]
        &
        (\ooCatOpLax)^\dto
        \ar[r, "D_\o"]
        &
        \ooCatOpLax
        \mpbox.
      \end{tikzcd}
  \]
\end{remark}

We end the section by expressing that the comma construction of $A \comma_C
B$ is above~$A$ and~$B$ via a strict transformation.

\begin{proposition}\label{prop:trans_comma}
  Let $C$ be an \oo-category.
  The canonical projection
  \[ p = (p_1, p_2) \colon A \comma_C B \to A \times B \]
  is natural in
  \[
      \begin{tikzcd}
        A \ar[r] & C & \ar[l] B \mpbox,
      \end{tikzcd}
  \]
  in the sense that it defines a strict transformation (see
  \ref{paragr:def_trans_Gray}) from the Gray \oo-functor
  \[
    \var \comma_C \var
    \colon
    \tr{\ooCatOpLax}{C} \times \trto{\ooCatOpLax}{C}
    \to \ooCatOpLax
  \]
  to the Gray \oo-functor obtained by composing
  \[
    \begin{tikzcd}
      \tr{\ooCatOpLax}{C} \times \trto{\ooCatOpLax}{C}
      \ar[r, "U \times U"]
      &
      \ooCatOpLax \times \ooCatOpLax
      \ar[r, "\times"]
      &
      \ooCatOpLax
      \mpbox,
    \end{tikzcd}
  \]
  where $U$ denotes the forgetful Gray \oo-functor (see~\ref{paragr:def_U})
  and $\times$ is the product Gray \oo-functor
  (see~\ref{paragr:times_enriched}).
\end{proposition}

\begin{proof}
  \newcommand\HomC{\Homi_{\tr{}C}}
  \newcommand\HomtC{\Homi_{\trto{}C}}
  \newcommand\HomOl{\Homi_{\mathrm{ol}}}
  Let $A$, $A'$, $B$ and $B'$ be four \oo-categories above $C$.
  Using the same abbreviation as in the proof of
  Theorem~\ref{thm:comma_Gray}, we have to show the commutativity of the
  square
  \[
    \begin{tikzcd}[column sep=4pc]
      \HomC(A, A') \times \HomtC(B, B')
      \ar[r, "{\var \comma_C \var}"]
      \ar[d, "\times"']
      &
      \HomOl(A \comma_C B, A' \comma_C B')
      \ar[d, "{\HomOl(A \comma_C B, p')}"]
      \\
      \HomOl(A \times B, A' \times B')
      \ar[r, "{\HomOl(p, A' \times B')}"']
      &
      \HomOl(A \comma_C B, A' \times B')
      \dpbox,
    \end{tikzcd}
  \]
  where we denoted simply by $\times$ the target Gray \oo-functor of the
  statement. So let $(u, \alpha, \beta, v)$ be a cell in the source
  \oo-category of this square. Using the formula defining the Gray comma
  construction, we get that this cell is sent to $(u \compG p_1, v \compG
  p_2)$ in~$\HomOl(A \comma_C B, A' \times B')$ by the upper path of the
  square. But the Gray \oo-functor $\times$ sends this same cell to $(u
  \compG q_1, v \compG q_2)$ in $\HomOl(A \times B, A' \times B')$, where
  $q_1 \colon A \times B \to A$ and $q_2 \colon A \times B \to B$ are the
  two projections, and since
  \[ 
    (u \compG q_1, v \compG q_2) \compG p
    =
    (u \compG q_1 \compG p, v \compG q_2 \compG p)
    =
    (u \compG p_1, v \compG p_2)
    \mpbox,
  \]
  the square indeed commutes.
\end{proof}

\section{Strict functorialities of the comma construction}

The purpose of this section is to study the functorialities of the comma
construction when restricted to higher strict transformations.

\begin{paragraph}
  Fix $C$ an \oo-category. The inclusion $\ooCatCart \hookto \ooCatOpLax$
  induces inclusions
  \[
    \tr{\ooCatCart}{C} \hookto \tr{\ooCatOpLax}{C}
    \quadand
    \trto{\ooCatCart}{C} \hookto \trto{\ooCatOpLax}{C}
    \mpbox.
  \]
  In particular, we can restrict the Gray \oo-functor
  \[
    \var \comma_C \var
    \colon
    \tr{\ooCatOpLax}{C} \times \trto{\ooCatOpLax}{C}
    \to \ooCatOpLax
  \]
  to a Gray \oo-functor
  \[
    \var \comma_C \var
    \colon
    \tr{\ooCatCart}{C} \times \trto{\ooCatCart}{C}
    \to \ooCatOpLax
    \mpbox.
  \]
  The goal of what follows is to prove that this Gray \oo-functor
  actually lands in~$\ooCatCart$.

  The strategy is obvious. We gave in~\ref{paragr:formula_higher_comma} a
  formula for this Gray \oo-functor and it suffices to check that the
  formula defines a cell of $\ooCatCart$. But this formula involves
   the oplax transformation $\gamma$ of the universal $2$-square
  which does not live in $\ooCatCart$! Nevertheless, we will see that the
  result is indeed in $\ooCatCart$. To do so, we will introduce an
  intermediate slice \oo-category
  \[
    \tr{\ooCatCart}{C} \hookto \wtr{\ooCatCart}{C} \hookto
    \tr{\ooCatOpLax}{C}
  \]
  based on the inclusions
  \[
    \Cyl{\Homi(A,B)}
    \hookto
    \Homi(A,\Cyl{B})
    \hookto
    \Cyl{\HomOpLax(A,B)}
  \]
  of~\ref{paragr:intermed_cyl}.
\end{paragraph}

\begin{paragraph}\label{paragr:cyl_sub_mod}
  Let $A$ and $B$ be two \oo-categories. As mentioned above, we defined
  in~\ref{paragr:intermed_cyl} inclusions
  \[
    \Cyl{\Homi(A,B)}
    \hookto
    \Homi(A,\Cyl{B})
    \hookto
    \Cyl{\HomOpLax(A,B)}
  \]
  factorizing the canonical inclusion.

  If $B$ is fixed and $A$ varies, we get inclusions
  \[
    \Cyl{\Homi(\var,B)}
    \hookto
    \Homi(\var,\Cyl{B})
    \hookto
    \Cyl{\HomOpLax(\var,B)}
  \]
  of anti Gray \oo-functors from $(\ooCatCart)^\trans$ to $\ooCatLax$, and
  hence by Proposition~\ref{prop:modules_as_functors}, inclusions of right
  sub-$\ooCatCart$-modules. In particular, this means that if $\alpha$ is a
  cell of $\Homi(A',\Cyl{B})$ and $u$ is a cell of $\Homi(A, A')$, then
  $\alpha \compr u$, where $\compr$ denotes the right action of
  \ref{paragr:compr}, is a cell of $\Homi(A, \Cyl{B})$.

  Moreover, for the same reasons as in \ref{paragr:cyl_map_modules}, each
  of the three \oo-categories
  \[
    \Cyl{\Homi(A,B)}
    \hookto
    \Homi(A,\Cyl{B})
    \hookto
    \Cyl{\HomOpLax(A,B)}
  \]
  is the object of morphisms of a category
  internal to $\ooCat$ and, by naturality, the two inclusions are morphisms
  of internal categories. In particular, if $\alpha$ and $\beta$ are
  $k$-cells of $\Homi(A,\Cyl{B})$ such that $\tt(\beta) = \ss(\alpha)$, then
  $\beta \compc \alpha$, where $\compc$ denotes the internal composition of
  cylinders of \ref{paragr:compc}, is a $k$-cell of $\Homi(A,\Cyl{B})$.
\end{paragraph}

\begin{paragraph}\label{paragr:compat_wcyl}
  Fix $C$ an \oo-category. We define an \oo-category $\wtr{\ooCatCart}{C}$
  in the following way:

  The objects of $\wtr{\ooCatCart}{C}$ are the same as the ones of
  $\tr{\ooCatCart}{C}$ or of $\tr{\ooCatOpLax}{C}$, that is, \oo-categories
  $A$ endowed with an \oo-functor $f \colon A \to C$.

  If $(A, f \colon A \to C)$ and $(A', f' \colon A' \to C)$ are two such objects,
  we set
  \[
    \begin{split}
      \MoveEqLeft
      \Homi_{\wtr{\ooCat}{C}}((A, f), (A', f')) \\
      & =
    \Homi(A, A') \times_{\Homi(A, C)} \Homi(A, \Cyl C) \times_{\Homi(A, C)}
    \{ f \} \mpbox.
    \end{split}
  \]
  For the moment, $\wtr{\ooCatCart}{C}$ has only been defined as a
  graph enriched in $\ooCat$.

  Recall that
  \[
    \begin{split}
      \MoveEqLeft
    \Homi_{\tr{\ooCatCart}{C}}((A, f), (A', f')) \\
      & =
      \Homi(A, A') \times_{\Homi(A, C)} \Cyl{\Homi(A, C)} \times_{\Homi(A, C)} \{ f \}
    \end{split}
  \]
  and
  \[
    \begin{split}
      \MoveEqLeft
    \Homi_{\tr{\ooCatOpLax}{C}}((A, f), (A', f')) \\
      & =
    \HomOpLax(A, A') \times_{\HomOpLax(A, C)} \Cyl{\HomOpLax(A, C)}
    \times_{\HomOpLax(A, C)} \{ f \} \mpbox.
    \end{split}
  \]
  Therefore, the monomorphisms
  \[
    \Cyl{\Homi(A,C)}
    \hookto
    \Homi(A,\Cyl{C})
    \hookto
    \Cyl{\HomOpLax(A,C)}
    \mpbox,
  \]
  as they are compatible with the source and target operations of internal
  categories, induce monomorphisms between these fiber products and hence
  monomorphisms of graphs enriched in $\ooCat$
  \[
    \tr{\ooCatCart}{C} \hookto \wtr{\ooCatCart}{C} \hookto \tr{\ooCatOpLax}{C}
    \mpbox.
  \]
  We will consider these monomorphisms as inclusions. Note that
  $\wtr{\ooCatCart}{C}$ have not only the same objects as
  $\tr{\ooCatOpLax}{C}$ but also the same $1$-cells. (But their $k$-cells
  differ for $k > 1$.)
\end{paragraph}

\goodbreak

\begin{proposition}
  If $C$ is an \oo-category, then $\wtr{\ooCatCart}{C}$ is a sub-Gray
  \oo-category of $\tr{\ooCatOpLax}{C}$. It is actually a strict
  \oo-category.
\end{proposition}

\begin{proof}
  Let $(A, f \colon A \to C)$, $(A', f' \colon A' \to C)$ and $(A'', f'' \colon
  A'' \to C)$ be three objects of~$\wtr{\ooCat}{C}$ and let $(u, \alpha)$ be
  a cell of~$\Homi_{\wtr{\ooCat}{C}}((A,f), (A',f'))$ and $(u', \alpha')$ a
  cell of~$\Homi_{\wtr{\ooCat}{C}}((A',f'), (A'',f''))$. By definition,
  their composition in~$\tr{\ooCatOpLax}{C}$ is given by
  \[
    (u', \alpha') \compG (u, \alpha)
    =
    \big(u' \compG u,
    \alpha \compc (\alpha' \compr u)\big) \mpbox,
  \]
  where $\compr$ denotes the right action of \ref{paragr:compr} and
  $\compc$ denotes the internal composition on~$\Cyl{\HomOpLax(A,C)}$ of
  \ref{paragr:compc}. Since $\ooCatCart$ is a sub-Gray \oo-category of
  $\ooCatOpLax$, the cell $u' \compG u$ lives in $\ooCatCart$. Moreover, by
  \ref{paragr:cyl_sub_mod}, the cell  $\alpha' \compr u$ is in $\Homi(A,
  \Cyl{C})$ and hence so is $\alpha \compc (\alpha' \compr u)$, thereby
  proving the stability under composition of~$\wtr{\ooCatCart}{C}$. The
  compatibility with units is obvious.

  The fact that $\wtr{\ooCatCart}{C}$ is a strict \oo-category follows from
  the formula giving the composition and the fact that $\ooCatCart$ is a
  strict \oo-category.
\end{proof}

\begin{paragraph}
  Let $C$ be an \oo-category.
  One defines similarly an \oo-category $\wtrto{\ooCatCart}{C}$ with Gray
  inclusions
  \[
    \trto{\ooCatCart}{C} \hookto \wtrto{\ooCatCart}{C} \hookto \trto{\ooCatOpLax}{C}
    \mpbox.
  \]
  The objects of $\wtrto{\ooCatCart}{C}$ are the \oo-categories $B$ endowed
  with an \oo-functor $g \colon B \to C$, and if $(B, g \colon B \to C)$ and
  $(B', g' \colon B' \to C)$ are two such objects, we have
  \[
    \begin{split}
      \MoveEqLeft
      \Homi_{\wtrto{\ooCat}{C}}((B, g), (B', g')) \\
      & =
    \{ g \} 
    \times_{\Homi(B, C)} \Homi(B, \Cyl C) \times_{\Homi(B, C)}
    \Homi(B, B') 
    \mpbox.
    \end{split}
  \]
\end{paragraph}

\begin{proposition}\label{prop:bifun_strict}
  Let $C$ be an \oo-category. The
  Gray \oo-functor
  \[
    \var \comma_C \var
    \colon
    \tr{\ooCatOpLax}{C} \times \trto{\ooCatOpLax}{C}
    \to \ooCatOpLax
  \]
  induces an \oo-functor
  \[
    \var \comma_C \var
    \colon
    \wtr{\ooCatCart}{C} \times \wtrto{\ooCatCart}{C}
    \to \ooCatCart
  \]
  and in particular an \oo-functor
  \[
    \var \comma_C \var
    \colon
    \tr{\ooCatCart}{C} \times \trto{\ooCatCart}{C}
    \to \ooCatCart
    \mpbox.
  \]
\end{proposition}

\begin{proof}
  Consider $A$, $A'$, $B$ and $B'$ four \oo-categories over $C$. 
  We have to show that the composite \oo-functor
  \[
    \begin{tikzcd}
    \Homi_{\wtr{\ooCatCart}{C}}((A, f), (A', f'))
      \times
    \Homi_{\wtrto{\ooCatCart}{C}}((B, g), (B', g'))
    \ar[d, hook]
    \\
    \Homi_{\tr{\ooCatOpLax}{C}}((A, f), (A', f'))
      \times
    \Homi_{\trto{\ooCatOpLax}{C}}((B, g), (B', g'))
    \ar[d, "\var \comma_C \var"']
    \\
    \displaystyle\HomOpLax(A \comma_C B ,A')
    \commabig_{\HomOpLax(A \comma_C B,C)} \HomOpLax(A \comma_C B, B')
    \ar[d, phantom, "\simeq"]
    \\
    \HomOpLax(A \comma_C B, A' \comma_C B')
    \end{tikzcd}
  \]
  factors through
  \[ \Homi(A \comma_C B, A' \comma_C B') \mpbox. \]

  Using~\ref{paragr:formula_higher_comma}, and with its notation,
  this \oo-functor is given on $k$-cells by
  \[
    (u, \alpha, \beta, v)
    \mapsto
    (u \compG p_1, (\beta \compr p_2) \compc \gamma \compc (\alpha \compr
    p_1), v \compG p_2)
    \mpbox,
  \]
  where
  \[
    (u, \alpha, \beta, v)
    \,\,\,\text{is in}\,\,
    \Homi(A, A')_k \times
    \Homi(A,\Cyl C)_k 
    \times
    \Homi(B,\Cyl C)_k 
    \times \Homi(B,B')_k
    \mpbox.
  \]
  Note that $\gamma$ can be seen as a $0$-cell of $\Homi(A \comma_C B,
  \Cyl C)$. Using~\ref{paragr:cyl_sub_mod}, we get that
  \[
    (\beta \compr p_2) \compc \gamma \compc (\alpha \compr p_1)
  \]
  is actually a $k$-cell in $\Homi(A \comma_C B, \Cyl C)$, and in
  particular,
  \[
    (u \compG p_1, (\beta \compr p_2) \compc \gamma \compc (\alpha \compr
    p_1), v \compG p_2)
  \]
  is a $k$-cell of
  {%
    \newcommand\nspace{\mskip-30mu}%
    \newcommand\lprodfib[1]{\nspace\prodfib_{\raisebox{-4pt}{$\scriptstyle
    #1$}}\nspace}%
  \[
      \displaystyle
      \Homi(A\comma_CB,A')\lprodfib{\Homi(A\comma_C B,C)}\Homi(A\comma_CB,
      \Cyl C)\lprodfib{\Homi(A\comma_C B,C)}\Homi(A\comma_CB,B')
      \mpbox,
  \]
  }%
  which is canonically isomorphic to
  \[
    \Homi(A \comma_C B, A' \times_C \Cyl C \times_C B')
    \simeq
    \Homi(A \comma_C B, A' \comma_C B')
    \mpbox.
  \]
  As this canonical isomorphism is compatible with the canonical isomorphism
  between
  {%
    \newcommand\nspace{\mskip-40mu}%
    \newcommand\lprodfib[1]{\nspace\prodfib_{\raisebox{-4pt}{$\scriptstyle
    #1$}}\nspace}%
  \[
      \displaystyle
      \HomOpLax(A\comma_CB,A')\lprodfib{\HomOpLax(A\comma_C B,C)}\HomOpLax(A\comma_CB,
      \Cyl C)\lprodfib{\HomOpLax(A\comma_C B,C)}\HomOpLax(A\comma_CB,B')
  \]
  }
  and
  \[
    \HomOpLax(A \comma_C B, A' \times_C \Cyl C \times_C B')
    \simeq
    \HomOpLax(A \comma_C B, A' \comma_C B')
    \mpbox,
  \]
  this proves the result.
\end{proof}

\begin{remark}\label{rem:comma_restr_fib_str}
  As in Remark~\ref{rem:comma_restr_fib}, if $A$ and $B$ are two fixed
  \oo-categories, there is a canonical embedding
  \[
    \Homi(A, C)^\o \times \Homi(B, C)^\dto
    \hookto
    \tr{\ooCatCart}{C} \times \trto{\ooCatCart}{C}
  \]
  and the comma construction thus restricts to an \oo-functor
  \[
    \var \comma_C \var \colon
    \Homi(A, C)^\o \times \Homi(B, C)^\dto
    \to
    \ooCatCart
    \mpbox.
  \]

  Similarly, if $A$ is an \oo-category, let us denote by $\HomCyl(A, B)$ the
  total dual of the fiber at $A$ of the forgetful \oo-functor
  \[ U \colon \wtr{\ooCatCart}{C} \to \ooCatCart \mpbox, \]
  so that we have inclusions
  \[ \Homi(A, C) \hookto \HomCyl(A, C) \hookto \HomOpLax(A, C) \mpbox. \]
  By definition (and duality), we have a canonical embedding
  \[
    \HomCyl(A, C)^\o \times \HomCyl(B, C)^\dto
    \hookto
    \tr{\ooCatCart}{C} \times \trto{\ooCatCart}{C}
  \]
  and the comma construction also restricts to an \oo-functor
  \[
    \var \comma_C \var \colon
    \HomCyl(A, C)^\o \times \HomCyl(B, C)^\dto
    \to
    \ooCatCart
    \mpbox.
  \]
\end{remark}

\begin{corollary}\label{coro:comma_Gray}
   If $B \to C$ is an \oo-functor, then we have an \oo-functor
      \[
        \var \comma_C B
        \colon
        \wtr{\ooCatCart}{C}
        \to \ooCatCart
      \]
  and in particular an \oo-functor
      \[
        \var \comma_C B
        \colon
        \tr{\ooCatCart}{C}
        \to \ooCatCart
        \mpbox,
      \]
  and if $A \to C$ is an \oo-functor, then we have an \oo-functor
      \[
        A \comma_C \var
        \colon
        \wtrto{\ooCatCart}{C}
        \to \ooCatCart
      \]
  and in particular an \oo-functor
      \[
        A \comma_C \var
        \colon
        \trto{\ooCatCart}{C}
        \to \ooCatCart
        \mpbox.
      \]
\end{corollary}

\begin{remark}
  By duality (see Remark~\ref{rem:dual_bifun_lax}),
  Proposition~\ref{prop:bifun_strict} implies that the lax comma
  construction induces an \oo-functor
  \[
      \var \commalax_C \var
      \colon
      \trD{\dtop}{\ooCatCart}{C}
      \times
      \trD{\dco}{\ooCatCart}{C}
      \to
      \ooCatCart
      \mpbox,
  \]
  and, more generally, an \oo-functor
  \[
      \var \commalax_C \var
      \colon
      \wptrD{\dtop}{\ooCatCart}{C}
      \times
      \wptrD{\dco}{\ooCatCart}{C}
      \to
      \ooCatCart
      \mpbox,
  \]
  where the \oo-categories
  \[ \wptrD{\dtop}{\ooCatCart}{C} \quadand \wptrD{\dco}{\ooCatCart}{C} \]
  are defined from their undecorated-by-$\Cylp$ analogue by replacing, in the definition
  of the \oo-category of morphisms, $\Cylp\Homi(\var, C)$ by $\Homi(\var,
  \Cylp C)$. (Recall that we set~$\Cylp X = \HomOpLax(\Dn{1}, X)$.)
\end{remark}

\section{Application: Grothendieck construction for \pdfoo-categories}

Our main motivation for studying the functorialities of the comma
construction was the Grothendieck construction for \oo-categories, to which
we will devote a separate paper \cite{AraGagnaGuettaIntegration}. In this short
final section, we define the Grothendieck construction for \oo-categories in
terms of comma \oo-categories and we deduce functoriality results for the
Grothendieck construction.

\begin{paragraph}\label{paragr:def_integr}
  Let $I$ be an \oo-category and let $F \colon I^\o \to \ooCatCart$ be an
  \oo-functor. We define the \ndef{(contravariant) Grothendieck
  construction} $\Gro{I} F$ of $F$ to be the total dual of the comma
  construction of the diagram
  \[
    \begin{tikzcd}
      \Dn{0} \ar[r, "\Dn{0}"] & \ooCatCart & I^\o \ar[l, "F"'] \dpbox,
    \end{tikzcd}
  \]
  where the left arrow corresponds to the object $\Dn{0}$ of
  $\ooCatCart$. In other words, we have
  \[
    \Gro{I} F = \big(\Dn{0} \comma_{\ooCatCart} F\big)^\o \mpbox.
  \]
  The second projection of the comma construction induces an \oo-functor
  $p \colon \Gro{I} F \to I$.
\end{paragraph}

\begin{remark}
  Although the \oo-category $\ooCatCart$ is not small, the comma
  construction $\Dn{0} \comma F$ makes sense and is a small \oo-category.
\end{remark}

\begin{remark}\label{rem:integration_is_slice}
  By Example~\ref{ex:slices_as_commas}, the Grothendieck construction of an
  \oo-functor
  \[ F \colon I^\o \to \ooCatCart \]
  is the total dual of a relative slice:
  \[
    \Gro{I} F = \big(\Dn{0} \comma F\big)^\o
    = \big(\cotr{I^\o}{\Dn{0}}\big)^\o
    \mpbox.
  \]
  In other words, we have a pullback square
  \[
    \begin{tikzcd}
      (\Gro{I} F)^\o \ar[d] \ar[r] &
      \cotr{\ooCatCart}{\Dn{0}} \ar[d, "U"] \\
      I^\o \ar[r,"F"']& \ooCatCart \dpbox{,}
      \ar[from=1-1,to=2-2,phantom,"\lrcorner" description, very near start]
    \end{tikzcd}
  \]
  where $U$ denotes the forgetful \oo-functor.
\end{remark}

\begin{paragraph}
  Let $F \colon I^\o \to \ooCatCart$ be an \oo-functor. To convince
  ourselves that our definition of the Grothendieck construction is
  reasonable, let us concretely describe the cells of~$\Gro{I} F$ in low
  dimensions.

  By definition, an object of $\Gro{I} F$ corresponds to an object $i$ of
  $I$ and an \oo-func\-tor~$x \colon \Dn{0} \to F(i)$. These objects can
  thus be identified with pairs $(i, x)$, where $i$ is an object of~$I$ and
  $x$ an object of $F(i)$.

  A $1$-cell of $\Gro{I} F$ corresponds to a $1$-cell $f \colon i \to i'$ of
  $I$ and a $2$-triangle
  \[
    \xymatrix@C=1.5pc@R=3pc{
      & \Dn{0}
    \ar[dl]_{x'}_{}="f" \ar[dr]^{x}_{}="s" \\
      F(i') \ar[rr]_{F(f)} & & F(i) 
      \ar@{}"s";[ll]_(.15){}="ss"
      \ar@{}"s";[ll]_(.55){}="tt"
      \ar@<0.0ex>@2"ss";"tt"_{\alpha}
    }
  \]
  in $\ooCatOpLax$. But the data of such an oplax transformation $\alpha$ is
  equivalent to the data of a $1$-cell $\alpha \colon x \to F(f)(x')$ of
  $F(i)$. The $1$-cells from $(i, x)$ to $(i', x')$ can thus be identified
  with pairs $(f \colon i \to i', \alpha \colon x \to F(f)(x'))$.

  A $2$-cell of $\Gro{I} F$ corresponds to a $2$-cell $\gamma \colon f \tod f'
  \colon i \to i'$ of $I$ and a cone
  \[
    \xymatrix@C=2pc@R=4pc{
      & \Dn{0}
    \ar[dl]_{x'}_{}="f" \ar[dr]^{x}_{}="s" \\
      F(i')
      \ar@/^2ex/@{.>}[rr]_(.30){F(f')}^{}="0"
      \ar@/_2ex/[rr]_(.33){F(f)}^{}="1"
      \ar@{:>}"0";"1"^{\,F(\gamma)}
      & & F(i)
      \ar@{}"s";[ll]_(.15){}="ss"
      \ar@{}"s";[ll]_(.55){}="tt"
      \ar@<-1.5ex>@/^-1ex/@{:>}"ss";"tt"_(.30){\alpha'}_{}="11"
      \ar@<-0ex>@/^1ex/@2"ss";"tt"^(.20){\!\!\alpha}^{}="00"
      \ar@3"00";"11"_{\,\Lambda}
    }
  \]
  in $\ooCatOpLax$. But the data of such an oplax $2$-transformation
  $\Lambda$ is equivalent to the data of a $2$-cell
  $\Lambda \colon \alpha \tod F(\gamma)_{x'} \comp_0 \alpha'$
  of $F(i)$. The $2$-cells from $(f, \alpha)$ to $(f', \alpha')$
  can thus be identified with these pairs $(\gamma, \Lambda)$.
\end{paragraph}

\begin{remark}
  The Grothendieck construction for \oo-categories was first defined by
  Warren \cite{Warren} using explicit formulas. In
  \cite{AraGagnaGuettaIntegration}, we will show that our definition is
  equivalent to Warren's one (up to some duality, as Warren defines the
  \emph{covariant} Grothendieck construction).
\end{remark}

\begin{paragraph}\label{paragr:large_slice}
  We will denote by $\ooCATOpLax$ the (very large) Gray \oo-category of
  possibly large \oo-categories, \oo-functors, oplax transformations and
  higher oplax transformations between them. We have a fully faithful
  inclusion $\ooCatOpLax \hookto \ooCATOpLax$. The \oo-category $\ooCatCart$
  is an object of $\ooCATOpLax$. When we consider $\ooCatCart$ as an
  object of $\ooCATOpLax$, we will denote it by $\{\ooCatCart\}$. In
  particular, we will write
  \[ \trto{\ooCatOpLax}{\{\ooCatCart\}} \]
  for the Gray \oo-category defined by the pullback
  \[
    \begin{tikzcd}
     \trto{\ooCatOpLax}{\{\ooCatCart\}}
     \ar[d] \ar[r]
     &
     \trto{\ooCATOpLax}{\{\ooCatCart}\}
     \ar[d, "U"]
     \\
     \ooCatOpLax 
     \ar[r, hook]
     &
     \ooCATOpLax
     \dpbox{,}
     \ar[from=1-1,to=2-2,phantom,"\lrcorner" description, very near start]
    \end{tikzcd}
  \]
  where $\trto{\ooCATOpLax}{\{\ooCatCart}\}$ is one of the (very large) slice
  Gray \oo-categories of~\ref{paragr:dual_slice} and $U$ is the forgetful
  Gray \oo-functor.
\end{paragraph}

\begin{theorem}
  The Grothendieck construction defines a Gray \oo-functor
  \[
    \begin{split}
      \Gro{} \colon 
      \big(\trto{\ooCatOpLax}{\{\ooCatCart\}}\big)^\dto
      & \to \ooCatOpLax
      \\
      F \colon I^\o \to \ooCatCart \quad\,\,\,
      & \mapsto \Gro{I} F \qquad\mpbox.
    \end{split}
  \]
\end{theorem}

\begin{proof}
  The Grothendieck construction factors as
  \[
    \begin{tikzcd}
      \big(\trto{\ooCatOpLax}{\{\ooCatCart\}}\big)^\dto
      \ar[r, "{\Dn{0} \comma \var}"]
      & 
      \big(\ooCatOpLax\big)^\dto
      \ar[r, "D_\o"]
      & 
      \ooCatOpLax \mpbox,
    \end{tikzcd}
  \]
  where $D_\o$ is the total dual. But the left arrow is a Gray \oo-functor
  by Theorem~\ref{thm:comma_Gray} (and more precisely
  Corollary~\ref{coro:comma_Gray_one_var}) and the right arrow is a Gray \oo-functor
  by~\ref{paragr:duality_Gray}.
\end{proof}

\begin{proposition}
  The \oo-functor $p \colon \Gro{I} F \to I$ is natural in $F \colon I^\o
  \to \ooCatCart$ in the sense that it defines a strict transformation
  from the Gray \oo-functor
  \[
    \Gro{} \colon 
      \big(\trto{\ooCatOpLax}{\{\ooCatCart\}}\big)^\dto
      \to \ooCatOpLax
  \]
  to the Gray \oo-functor obtained by composing
  \[
    \begin{tikzcd}
      \big(\trto{\ooCatOpLax}{\{\ooCatCart\}}\big)^\dto
      \ar[r, "U"]
      &
      \big(\ooCatOpLax)^\dto
        \ar[r, "D_\o"]
      &
      \ooCatOpLax \mpbox,
    \end{tikzcd}
  \]
  where $D_o$ is the total dual.
\end{proposition}

\begin{proof}
  This is a particular case of the analogous result for the comma
  construction, that is, Proposition~\ref{prop:trans_comma}.
\end{proof}

\begin{proposition}
  If $I$ is a fixed \oo-category, the Grothendieck construction restricts to
  a Gray \oo-functor
  \[
    \Gro{I} \colon \HomOpLax(I^\o, \ooCatCart) \to \ooCatOpLax \mpbox.
  \]
\end{proposition}

\begin{proof}
  By Remark~\ref{rem:comma_restr_fib}, there is a canonical embedding
  \[
    \HomOpLax(I^\o, \ooCatCart)^\dto
    \hookto
    \trto{\ooCatOpLax}{\{\ooCatCart\}}
    \mpbox,
  \]
  hence the result by the previous proposition.
\end{proof}

\begin{paragraph}
  We are now going to state the strict functorialities of the Grothendieck
  construction. We can define, as in~\ref{paragr:large_slice}, a large strict
  \oo-category
  \[
    \begin{tikzcd}
     \trto{\ooCatCart}{\{\ooCatCart\}}
     \ar[d] \ar[r]
     &
     \trto{\ooCATCart}{\{\ooCatCart}\}
     \ar[d, "U"]
     \\
     \ooCatCart 
     \ar[r, hook]
     &
     \ooCATCart
     \dpbox{,}
     \ar[from=1-1,to=2-2,phantom,"\lrcorner" description, very near start]
    \end{tikzcd}
  \]
  where $\ooCATCart$ is the (very large) strict \oo-category of
  possibly large \oo-categories, \oo-functors, strict transformations and
  higher strict transformations between them.
\end{paragraph}

\begin{proposition}
  The Grothendieck construction restricts to a strict \oo-functor
  \[
    \Gro{} \colon \big(\trto{\ooCatCart}{\{\ooCatCart\}}\big)^\dto
      \to \ooCatCart \mpbox,
  \]
  and, if $I$ is a fixed \oo-category, to a strict \oo-functor
  \[
    \Gro{I} \colon \Homi(I^\o, \ooCatCart) \to \ooCatCart \mpbox.
  \]
\end{proposition}

\begin{proof}
  This follows from the analogous result for the comma construction, that
  is, Proposition~\ref{prop:bifun_strict}.
\end{proof}

\begin{remark}
  By Remark~\ref{rem:comma_restr_fib_str} and using the same notation, the
  second \oo-functor of the above proposition actually extends to an
  \oo-functor
  \[
    \Gro{I} \colon \HomCyl(I^\o, \ooCatCart) \to \ooCatCart \mpbox.
  \]
\end{remark}

We end the paper by an opening on a definition of the Grothendieck
construction for (anti) Gray \oo-functors, based on
Remark~\ref{rem:integration_is_slice}.

\begin{paragraph}
  Let $\IG$ be a Gray \oo-category, so that $\IG^\o$ is an anti Gray
  \oo-category, and fix $F \colon \IG^\o \to \ooCatLax$ an anti Gray
  \oo-functor. We define the \ndef{Grothendieck construction}~$\Gro{\IG} F$
  of $F$ as the Gray \oo-category obtained by taking the total dual of the
  pullback of anti Gray \oo-categories
  \[
    \begin{tikzcd}
      (\Gro{\IG} F)^\o \ar[d] \ar[r] &
      \cotr{\ooCatLax}{\Dn{0}} \ar[d, "U"] \\
      \IG^\o \ar[r,"F"']& \ooCatLax \dpbox{,}
      \ar[from=1-1,to=2-2,phantom,"\lrcorner" description, very near start]
    \end{tikzcd}
  \]
  where $\cotr{\ooCatLax}{\Dn{0}}$ is one of the slice anti Gray
  \oo-categories defined in~\ref{paragr:dual_slice} and $U$ is the forgetful
  anti Gray \oo-functor.

  In the case where $\IG$ comes from a strict \oo-category $I$ and the
  \oo-functor $F$ factors through the inclusion $\ooCatCart \hookto
  \ooCatLax$, we recover the Grothendieck construction defined
  in~\ref{paragr:def_integr}. Indeed, consider the commutative diagram
  \[
    \begin{tikzcd}
      (\Gro{I} F)^\o \ar[d] \ar[r] &
      \cotr{\ooCatCart}{\Dn{0}} \ar[d, "U"] \ar[r, hook] &
      \cotr{\ooCatLax}{\Dn{0}} \ar[d, "U"]
      \\
      I^\o \ar[r,"F"']& \ooCatCart \ar[r, hook] &
      \ooCatLax \dpbox{.}
    \end{tikzcd}
  \]
  The left square is cartesian by Remark~\ref{rem:integration_is_slice}.
  As for the right square, it is easily seen to be cartesian as well using
  the canonical isomorphism $\HomLax(\Dn{0}, C) \simeq \Homi(\Dn{0}, C)$,
  for any \oo-category $C$. It follows that the composite of these two
  squares is cartesian, showing the compatibility of the two Grothendieck
  constructions.
\end{paragraph}

In future work, we plan to investigate this Grothendieck construction for
(anti) Gray \oo-functors.

\appendix

\section{Non-existence of the Gray \pdfoo-category of cylinders}
\label{app}

The purpose of this appendix is to explain why, for $\CG$ a Gray
\oo-category, there is no natural Gray \oo-category of cylinders $\Cyl\CG$.

\begin{paragraph}
  By definition, a $k$-cylinder in an \oo-category $C$ is an \oo-functor
  $\Dn{1} \otimes \Dn{k} \to C$. This definition extends immediately to Gray
  \oo-categories: if $\CG$ is a Gray \oo-category, a $k$-cylinder in $\CG$
  is a Gray \oo-functor $\Dn{1} \otimes \Dn{k} \to \CG$, where
  $\Dn{1} \otimes \Dn{k}$ is considered as a Gray \oo-category. Moreover,
  the operations source, target and units of cylinders are still defined and
  if $\CG$ is a Gray \oo-category we get a structure of reflexive \oo-graph
  on cylinders in $\CG$, extending the one of $\Cyl C$ when $\CG$ comes from
  an \oo-category $C$.
\end{paragraph}

\begin{paragraph}
  The combinatorial description of $k$-cylinders (see for instance
  \cite[Proposition~B.1.6]{AraMaltsiJoint}) still applies in the Gray world.
  In other words, giving a $k$-cylinder in a Gray \oo-category $\CG$ amounts
  to the following data:
  \begin{itemize}
    \item two $k$-cells $x$ and $y$ of $\CG$,
    \item for every $0 \le l < k$, two $(k+1)$-cells
      \[
        \begin{split}
          \gamma^-_l & \colon
          \gamma^+_{l-1} \comp_{l-1} \dots \comp_1 \gamma^+_0 \comp_0 s_l(x)
          \to
          s_l(y) \ast_0 \gamma^-_0 \ast_1 \dots \ast_{l-1} \gamma^-_{l-1}
          \\
          \gamma^+_l & \colon
          \gamma^+_{l-1} \comp_{l-1} \dots \comp_1 \gamma^+_0 \comp_0 t_l(x)
          \to
          t_l(y) \ast_0 \gamma^-_0 \ast_1 \dots \ast_{l-1} \gamma^-_{l-1}
        \end{split}
      \]
      of $\CG$,
    \item a $(k+1)$-cell
      \[
        \gamma_k \colon
        \gamma^+_{k-1} \comp_{k-1} \dots \comp_1 \gamma^+_0 \comp_0 x
        \to
        y \ast_0 \gamma^-_0 \ast_1 \dots \ast_{k-1} \gamma^-_{k-1}
      \]
      of $\CG$.
  \end{itemize}
  In the above formulas we use the convention that $\comp_i$ has priority
  over $\comp_j$ whenever~$i < j$.

  In particular, $0$-cylinders are $1$-cells of $\CG$, $1$-cylinders are
  $2$-squares in $\CG$,
  \[
   \xymatrix@R=3pc{
   x \ar[d]_{\alpha_0} \\
   y \dpbox,
   }
   \qquad
   \qquad
    \xymatrix@C=3pc@R=3pc{
      s(x) \ar[r]^x \ar[d]_{\alpha^-_0} &
      t(x) \ar[d]^{\alpha^+_0} \\
      s(y) \ar[r]_y & t(y)
      \ar@{}[u];[l]_(.30){}="s"
      \ar@{}[u];[l]_(.70){}="t"
      \ar@2"s";"t"_{\alpha_1}
      \dpbox,
    }
  \]
  and $2$-cylinders are diagrams in $\CG$ of the shape
  \[
    \xymatrix@C=3pc@R=3pc{
      \bullet
      \ar@/^2ex/[r]^(0.70){}_{}="0"
      \ar@/_2ex/[r]_(0.70){}_{}="1"
      \ar[d]_{}="f"_{\alpha^-_0}
      \ar@2"0";"1"_{x\,\,}
      &
      \bullet
      \ar[d]^{\alpha^+_0} \\
      \bullet
      \ar@{.>}@/^2ex/[r]^(0.30){}_{}="0"
      \ar@/_2ex/[r]_(0.30){}_{}="1"
      \ar@{:>}"0";"1"_{y\,\,}
      &
      \bullet
      \ar@{}[u];[l]_(.40){}="x"
      \ar@{}[u];[l]_(.60){}="y"
      \ar@<-1.5ex>@/_1ex/@{:>}"x";"y"_(0.58){\alpha^-_1\!}_{}="0"
      \ar@<1.5ex>@/^1ex/@2"x";"y"^(0.40){\!\alpha^+_1}_{}="1"
      \ar@{}"1";"0"_(.05){}="z"
      \ar@{}"1";"0"_(.95){}="t"
      \ar@3{>}"z";"t"_(0.40){\!\alpha_2}
    }
    \qquad
    \raisebox{-2pc}{i.e.,}
    \qquad
    \xymatrix@C=3pc@R=3pc{
      \bullet
      \ar@/^2ex/[r]^(0.70){}_{}="0"
      \ar@/_2ex/[r]_(0.70){}_{}="1"
      \ar[d]_{}="f"_{\alpha^-_0}
      \ar@2"0";"1"_{x\,\,}
      &
      \bullet
      \ar[d]^{\alpha^+_0} \\
      \bullet
      \ar@/_2ex/[r]
      &
      \bullet
      \ar@{}[u];[l]_(.35){}="0"
      \ar@{}[u];[l]_(.75){}="1"
      \ar@<0.5ex>@2"0";"1"_(0.50){\alpha^+_1}_{}="1"
    }
    \,
    \raisebox{-2pc}{$\overset{\alpha_2}{\Rrightarrow}$}
    \,
    \xymatrix@C=3pc@R=3pc{
      \bullet
      \ar@/^2ex/[r]
      \ar[d]_{}="f"_{\alpha^-_0}
      &
      \bullet
      \ar[d]^{\alpha^+_0} \\
      \bullet
      \ar@/^2ex/[r]^(0.30){}_{}="0"
      \ar@/_2ex/[r]_(0.30){}_{}="1"
      \ar@2"0";"1"_{y\,\,}
      &
      \bullet
      \ar@{}[u];[l]_(.30){}="0"
      \ar@{}[u];[l]_(.70){}="1"
      \ar@<0.0ex>@2"0";"1"_(0.50){\!\alpha^-_1}_{}="1"
      \dpbox,
    }
  \]
  with
  \[ \alpha_2 \colon \alpha^+_1 \comp_1 \alpha^+_0 \comp_0 x \to y \comp_0
  \alpha^-_0 \comp_1 \alpha^-_1 \mpbox. \]
\end{paragraph}

\begin{paragraph}
  If $\CG$ is a Gray \oo-category (or more generally a sesquicategory),
  $0$-cylinders and $1$-cylinders in $\CG$ organize themselves as a
  $1$-category. Indeed, if $\alpha$ and $\beta$ are two $1$-cylinders in
  $\CG$ that are composable,
  \[
    \xymatrix@C=3pc@R=3pc{
      \bullet \ar[r]^f \ar[d]_{\alpha^-_0} &
      \bullet \ar[d]_{\alpha^+_0}^{\beta^-_0} \ar[r]^h &
      \bullet \ar[d]^{\beta^+_0}
      \\
      \bullet \ar[r]_g &
      \bullet \ar[r]_i
      \ar@{}[u];[l]_(.30){}="s"
      \ar@{}[u];[l]_(.70){}="t"
      \ar@2"s";"t"_{\alpha_1}
      &
      \bullet 
      \ar@{}[u];[l]_(.30){}="s"
      \ar@{}[u];[l]_(.70){}="t"
      \ar@2"s";"t"^{\beta_1}
      \dpbox,
    }
  \]
  so that $\alpha^+_0 = \beta^-_0$, then their composite $\gamma = \beta \comp_0
  \alpha$ is defined as in the strict case, that
  is, by
  \[
    \xymatrix@C=3pc@R=3pc{
      \bullet \ar[r]^{h \comp_0 f} \ar[d]_{\gamma^-_0 = \alpha^-_0} &
      \bullet \ar[d]^{\gamma^+_0 = \beta^+_0} \\
      \bullet \ar[r]_{i \comp_0 g} & \bullet
      \ar@{}[u];[l]_(.30){}="s"
      \ar@{}[u];[l]_(.70){}="t"
      \ar@2"s";"t"_{\gamma_1}
    }
  \]
  with 
  \[ \gamma_1 = (i \comp_0 \alpha_1) \comp_1 (\beta_1 \comp_0 f) \mpbox. \]
  It is immediate that this indeed defines a $1$-cylinder with the
  appropriate source and target, and that $0$-cylinders and
  $1$-cylinders in $\CG$, together with this composition, form a
  $1$-category.
\end{paragraph}

We will now explain why this $1$-category does not extend to higher
dimensions.

\begin{paragraph}
  Let $\CG$ be a Gray \oo-category and consider a $1$-cylinder $\alpha$ and
  a $2$-cylinder $\Lambda$ in~$\CG$ such that $t_0(\alpha) = s_0(\Lambda)$,
  that is, a diagram
  \[
    \xymatrix@C=3pc@R=3pc{
      \bullet \ar[d]_(0.55){\alpha^-_0} \ar[r]^f
      &
      \bullet
      \ar@/^2ex/[r]^(0.70){}_{}="0"
      \ar@/_2ex/[r]_(0.70){}_{}="1"
      \ar[d]_{}="f"_(0.55){\alpha^+_0}^(0.55){\Lambda^-_0}
      \ar@2"0";"1"_{\gamma\,\,}
      &
      \bullet
      \ar[d]^(0.55){\Lambda^+_0}
      \\
      \bullet \ar[r]_g
      &
      \bullet
      \ar@{}[u];[l]_(.30){}="s"
      \ar@{}[u];[l]_(.70){}="t"
      \ar@2"s";"t"_{\alpha_1}
      \ar@{.>}@/^2ex/[r]^(0.30){}_{}="0"
      \ar@/_2ex/[r]_(0.30){}_{}="1"
      \ar@{:>}"0";"1"_{\delta\,\,}
      &
      \bullet
      \ar@{}[u];[l]_(.40){}="x"
      \ar@{}[u];[l]_(.60){}="y"
      \ar@<-1.5ex>@/_1ex/@{:>}"x";"y"_(0.58){\Lambda^-_1\!}_{}="0"
      \ar@<1.5ex>@/^1ex/@2"x";"y"^(0.40){\!\Lambda^+_1}_{}="1"
      \ar@{}"1";"0"_(.05){}="z"
      \ar@{}"1";"0"_(.95){}="t"
      \ar@3{>}"z";"t"_(0.40){\!\Lambda_2}
      \dpbox,
    }
  \]
  with $\alpha^+_0 = \Lambda^-_0$. Let us try to define the composite
  $\Gamma = \Lambda \comp_0 \alpha$. This $2$-cylinder $\Gamma$ must satisfy
  \[ 
    s(\Gamma) =
    s(\Lambda \comp_0 \alpha) = s(\Lambda) \comp_0 \alpha
    \quadand
    t(\Gamma) =
    t(\Lambda \comp_0 \alpha) = t(\Lambda) \comp_0 \alpha
    \mpbox,
  \]
  as in any sesquicategory. In other words, its ``back face'' has to be the
  composite of
  \[
    \xymatrix@C=3pc@R=3pc{
      \bullet \ar[r]^f \ar[d]_{\alpha^-_0}
      &
      \bullet
      \ar@/^2ex/[r]^{s(\gamma)}
      \ar[d]
      &
      \bullet
      \ar[d]^{\Lambda^+_0}
      \\
      \bullet \ar[r]_g
      &
      \bullet
      \ar@/^2ex/[r]_{s(\delta)}
      \ar@{}[u];[l]_(.30){}="s"
      \ar@{}[u];[l]_(.70){}="t"
      \ar@2"s";"t"_{\alpha_1}
      &
      \bullet
      \ar@{}[u];[l]_(.30){}="0"
      \ar@{}[u];[l]_(.70){}="1"
      \ar@<0.0ex>@2"0";"1"_(0.50){\!\Lambda^-_1}_{}="1"
      \dpbox,
    }
  \]
  and so
  \[
    \Gamma^-_1 = (s(\delta) \comp_0 \alpha_1) \comp_1 (\Lambda^-_1 \comp_0 f)
    \mpbox,
  \]
  and its ``front face'' the composite of
  \[
    \xymatrix@C=3pc@R=3pc{
      \bullet \ar[d]_{\alpha^-_0} \ar[r]^f
      &
      \bullet
      \ar@/_2ex/[r]^{t(\gamma)}
      \ar[d]
      &
      \bullet
      \ar[d]^{\Lambda^+_0}
      \\
      \bullet \ar[r]_g
      &
      \bullet
      \ar@{}[u];[l]_(.30){}="s"
      \ar@{}[u];[l]_(.70){}="t"
      \ar@2"s";"t"_{\alpha_1}
      \ar@/_2ex/[r]_{t(\delta)}
      &
      \ar@{}[u];[l]_(.35){}="0"
      \ar@{}[u];[l]_(.75){}="1"
      \ar@<0.5ex>@2"0";"1"_(0.50){\Lambda^+_1}_{}="1"
      \bullet
      \dpbox,
    }
  \]
  and so
  \[
    \Gamma^+_1 = (t(\delta) \comp_0 \alpha_1) \comp_1 (\Lambda^+_1 \comp_0 f)
    \mpbox.
  \]
  In particular,
  \[ 
    \Gamma^-_0 = \alpha^-_0
    \quadand
    \Gamma^+_0 =  \Lambda^+_0 \mpbox.
  \]
  Moreover, to be compatible with the strict case, the ``top face'' of
  $\Gamma$ has to be $\gamma \comp_0 f$ and its ``bottom face'' $\delta \comp_0 g$.

  Therefore, all the defining data of $\Gamma$ are already specified, except
  $\Gamma_2$. Now, by definition of a $2$-cylinder, we have
  \[
    \begin{split}
      s(\Gamma_2)
      & = \Gamma^+_1 \comp_1 \Gamma^+_0 \comp_0 (\gamma \comp_0 f) \\
      & = \big[(t(\delta) \comp_0 \alpha_1) \comp_1 (\Lambda^+_1 \comp_0
      f)\big] \comp_1 \Lambda^+_0 \comp_0 (\gamma \comp_0 f) \\
      & = (t(\delta) \comp_0 \alpha_1) \comp_1 (\Lambda^+_1 \comp_0
      f) \comp_1 (\Lambda^+_0 \comp_0 \gamma \comp_0 f) \\
      & = (t(\delta) \comp_0 \alpha_1) \comp_1 \big[(\Lambda^+_1 \comp_1
      (\Lambda^+_0 \comp_0 \gamma)) \comp_0 f\big] \\
      & = (t(\delta) \comp_0 \alpha_1) \comp_1 (s(\Lambda_2) \comp_0 f) \\
    \end{split}
  \]
  and
  \[
    \begin{split}
      t(\Gamma_2)
      & = (\delta \comp_0 g) \comp_0 \Gamma^-_0 \comp_1 \Gamma^-_1 \\
      & = (\delta \comp_0 g) \comp_0 \alpha^-_0 \comp_1 \big[(s(\delta) \comp_0 \alpha_1)
      \comp_1 (\Lambda^-_1 \comp_0 f)\big]
            \\
      & = \big[(\delta \comp_0 (g \comp_0 \alpha^-_0) \comp_1 (s(\delta) \comp_0
      \alpha_1)\big] \comp_1 (\Lambda^-_1 \comp_0 f) \\
      & = \big[(\delta \comp_0 t(\alpha_1)) \comp_1 (s(\delta) \comp_0
      \alpha_1)\big] \comp_1 (\Lambda^-_1 \comp_0 f) \mpbox.
    \end{split}
  \]
  In a strict \oo-category, we could apply the exchange rule to carry on this
  computation but in a Gray \oo-category all we have is a $3$-cell
  \[ 
    \delta \comp_0 \alpha_1 \colon
    (\delta \comp_0 t(\alpha_1)) \comp_1 (s(\delta) \comp_0 \alpha_1)
    \Rrightarrow
    (t(\delta) \comp_0 \alpha_1) \comp_1 (\delta \comp_0 s(\alpha_1))
    \mpbox.
  \]
  We thus get
  \[
    t(\Gamma_2) = s(\delta \comp_0 \alpha_1) \comp_1 (\Lambda^-_1 \comp_0 f)
  \]
  and
  \[
    \begin{split}
     t(\delta \comp_0 \alpha_1) \comp_1 (\Lambda^-_1 \comp_0 f)
     & =
     \big[
       (t(\delta) \comp_0 \alpha_1) \comp_1 (\delta \comp_0 s(\alpha_1))
     \big]
     \comp_1 (\Lambda^-_1 \comp_0 f) \\
     & =
       (t(\delta) \comp_0 \alpha_1) \comp_1 (\delta \comp_0 \alpha^+_0
       \comp_0 f) \comp_1 (\Lambda^-_1 \comp_0 f) \\
     & =
       (t(\delta) \comp_0 \alpha_1) \comp_1 \big[((\delta \comp_0 \alpha^+_0)
     \comp_1 \Lambda^-_1) \comp_0 f \big]\\
     & =
       (t(\delta) \comp_0 \alpha_1) \comp_1 \big[((\delta \comp_0 \Lambda^-_0)
       \comp_1 \Lambda^-_1) \comp_0 f \big] 
     \\
     & =
       (t(\delta) \comp_0 \alpha_1) \comp_1 (t(\Lambda_2) \comp_0 f) \mpbox.
    \end{split}
  \]
  Therefore we have a zigzag of $3$-cells
  \[
    s(\Gamma_2)
    \overset{}{\Rrightarrow}
    t(\delta \comp_0 \alpha_1) \comp_1 (\Lambda^-_1 \comp_0 f)
    \overset{}{\Lleftarrow}
    t(\Gamma_2)
    \mpbox,
  \]
  the left and right $3$-cells being
  \[
    (t(\delta) \comp_0 \alpha_1) \comp_1 (\Lambda_2 \comp_0 f)
    \quadand
     (\delta \comp_0 \alpha_1) \comp_1 (\Lambda^-_1 \comp_0 f) \mpbox,
  \]
  respectively. In general, we cannot compose this zigzag and there is no
  $3$-cell $\Gamma_2$ with the correct source and target.
\end{paragraph}

\begin{remark}
  In an anti Gray \oo-category, the orientation of the $3$-cell $\delta
  \comp_0 \alpha_1$ would be reverted so that we get two composable
  $3$-cells
  \[
    s(\Gamma_2)
    \overset{}{\Rrightarrow}
    s(\delta \comp_0 \alpha_1) \comp_1 (\Lambda^-_1 \comp_0 f)
    \overset{}{\Rrightarrow}
    t(\Gamma_2)
    \mpbox,
  \]
  allowing to define $\Gamma_2$. This means that we can make sense of the
  $2$-cylinder $\Lambda \comp_0 \alpha$ in an anti Gray \oo-category.

  Dually, one can make sense of the composite
  $\beta \comp_0 \Lambda$, when $\beta$ is a $1$-cylinder, $\Lambda$ a
  $2$-cylinder and $t_0(\Lambda) = s_0(\beta)$, in a Gray \oo-category but
  not in an anti Gray \oo-category.

  Overall, one cannot make sense of the composite $\beta \comp_0 \Lambda \comp_0
  \alpha$ in a Gray \oo-category or in an anti Gray \oo-category.
\end{remark}

\vspace{-3mm}
\enlargethispage{1.7\baselineskip}

\bibliography{biblio}
\bibliographystyle{mysmfplain}

\end{document}